%% file: ex_article.tex
\documentclass[hidelinks,onefignum,onetabnum]{siamart251216}

\input{ex_shared}

\ifpdf
\hypersetup{
  pdftitle={Faster Stochastic ADMM for Nonsmooth Composite Convex Optimization in Hilbert Spaces},
  pdfauthor={Weihua Deng, Haiming Song, Hao Wang, and Jinda Yang}
}
\fi

\newsiamthm{assumption}{Assumption}

\usepackage{graphicx,epstopdf} % <- Preamble
\usepackage[caption=false]{subfig}

\begin{document}

\maketitle

% REQUIRED
\begin{abstract}
	In this paper, a stochastic alternating direction method of multipliers (ADMM) is proposed for a class of nonsmooth composite and stochastic convex optimization problems in Hilbert spaces, motivated by optimization problems constrained by partial differential equation (PDE) with random coefficients. We prove the strong convergence of the proposed ADMM algorithm in the strongly convex case, and show the faster nonergodic convergence rates in terms of functional values and feasibility violation for both strongly convex and general convex cases.  We demonstrate  the application of the proposed method to solve certain model problems, along with its associated probability bound of large deviation. Some preliminary numerical results illustrate the efficiency of our method. 
\end{abstract}

% REQUIRED

% REQUIRED

\section{Introduction}
\label{sec: introduction}

We consider the following composite convex optimization problem: 
\vspace{-15pt}
\begin{equation}\label{eq: general model problem}
	\vspace{-8pt}
	\underset{u\in U_{ad}}{\min}~  f(u) + g(u),
\end{equation}
where the admissible set $U_{ad}$ is a nonempty closed, and convex subset of a Hilbert space $U$. The functional $f$ is of the form 
\vspace{-10pt}
\begin{equation}\label{eq: stochastic form of $f$}
	\vspace{-8pt}
	f(u) = \mathbb{E}\left[F(u,\xi)\right] = 
	\int_{\Omega} F(u,\xi(\omega)) ~\mathrm{d}\mathbb{P}(\omega),
\end{equation}
in which the vector-valued random variable $\xi:\Omega\rightarrow \Xi \subset \mathbb{R}^{n}$ is defined on a probability space $(\Omega,\mathcal{F},\mathbb{P})$, and the expectation is well defined and finite-valued for every $u\in U$. For every $\xi \in \Xi$, the functional $F(\cdot,\xi)$ is convex and Fr\'{e}chet differentiable on $U_{ad}$, then it follows that $f$ is convex and  Fr\'{e}chet differentiable. The function $g$ is proper, lower semicontinuous, convex but generally nonsmooth. 

Problem $(\ref{eq: general model problem})$ arises in many applications involving PDE-constrained optimization under uncertainty, where the functional $f$ is a measure of fit depending implicitly on random PDEs with some prescribed or observed data, and the functional $g$ is often used to impose regularization and penalty that enforce structure to the solution. We refer to, e.g., \cite{biccari2023two,cao2022adaptive,geiersbach2019projected,heinkenschloss2025optimization,martin2019stochastic} for a few references. For instance, a commonly encountered functional $f$ in optimal control is the following tracking-type functional with smooth Tikhonov regularization
\vspace{-5pt}
\begin{equation*}\label{eq: example of f}
	f(u) = \mathbb{E}\left[\frac{1}{2}\left\| y(u,\xi(\omega)) - y_{d}\right\|^{2} \right] + \frac{\alpha}{2}\left\| u \right\|^{2},
	\vspace{-8pt}
\end{equation*}
for a given desired state $y_{d}$ and nonnegative $\alpha$, where the norm $\|\cdot\|$ is induced by the inner product $\langle\cdot,\cdot\rangle$ in $U$, and the parametrized control-to-state mapping $ u\longmapsto y(u,\xi(\omega))$ is well-defined by the well-posedness of the corresponding PDE under uncertainty. 

While problem (\ref{eq: general model problem}) appears to be deterministic in form, the functional value $f(u)$ is a statistical quantity. Compared to deterministic optimization, the challenge in solving problem (\ref{eq: general model problem}) is that it is usually impossible or very expensive to accurately evaluate the functional $f$ and its gradient by computing the expectation.

\subsection{Motivation}
%\subsection{ADMM formulation and motivation}
\label{subsec: motivation}
A standard approach for solving problem (\ref{eq: general model problem}) is the stochastic approximation (SA) \cite{nemirovski2009robust,robbins1951stochastic} via subsequent calls to a stochastic first-order oracle (SFO). Precisely,  given an input $u \in U$ and an independent identically distributed (i.i.d) random variable $\xi$ (also independent of the input $u$), the SFO outputs an unbiased stochastic gradient $\nabla F(u,\xi)$, and at the $k$-th iteration, the classic SA iteration takes the update
\vspace{-4pt}
\begin{equation}\label{eq: classic SA}
	u_{k+1} = \underset{u\in U_{ad}}{\arg\min}
	\left\lbrace 
	\left\langle \nabla F(u_{k},\xi_{k}), u - u_{k} \right\rangle + \frac{1}{2\eta_{k}}\|u-u_k\|^{2} + g(u)
	\right\rbrace,
	\vspace{-4pt}
\end{equation}
where the parameter $\eta_{k}$ is a positive stepsize and the realization $\xi_{k}\in \Xi$ is generated by sampling from $\mathbb{P}$. In the unconstrained case $U_{ad} = U$, with the definition of proximal operator, the SA iteration (\ref{eq: classic SA}) can be understood as applying the proximal operator associated with the nonsmooth $g$ to the stochastic gradient descent step for the smooth $f$, which is a basic version of the stochastic proximal gradient methods \cite{atchade2017perturbed,jofre2019variance,juditsky2012first,necoara2025efficiency,rosasco2020convergence,xiao2014proximal}. Among the most important first-order methods, stochastic proximal gradient (SPG) method does not require the subdifferentiability of the nonsmooth functional $g$ and relies only on the computation of the associated proximal operator. In the constrained case $U_{ad} \subset U$, the proximal evaluation of the nonsmooth $g$ subject to the constraint $u\in U_{ad}$ is generally difficult to compute, and due to this challenge, the stochastic subgradient (SSG) method is widely applied \cite{davis2020stochastic,duchi2011adaptive,necoara2021minibatch,nedic2014stochastic,nemirovski2009robust,ram2009incremental,zhang2025stochastic}. Assuming that a subgradient $g^{\prime}(u_{k})$ of $g$ at $u_{k}$ can be obtained, a basic iteration of stochastic subgradient method can be written as
\vspace{-3pt}
\begin{equation}\label{eq: basic SSG}
	u_{k+1} = \underset{u\in U_{ad}}{\arg\min}
	\left\lbrace 
	\left\langle \nabla F(u_{k},\xi_{k}) + g^{\prime}(u_{k}), u - u_{k} \right\rangle + \frac{1}{2\eta_{k}}\|u-u_k\|^{2}
	\right\rbrace.
	\vspace{-3pt}
\end{equation}  
Clearly, it works on the premise that the admissible set $U_{ad}$ has a structure conducive to the efficient computation of $u_{k+1}$, such as when a closed form of $u_{k+1}$ is available. The convergence results of the variants of the SA approach to problem (\ref{eq: general model problem}) are extensively studied under some regularity assumptions often made on the functional $f$ and $g$ (see \cite{ghadimi2012optimal,ghadimi2013optimal,grimmer2019convergence,lan2012optimal,necoara2021general,nedic2014stochastic,rosasco2020convergence,shamir2013stochastic} and the references
therein). In this context, sublinear (strong) convergence is derived for SPG and SSG in the (strongly) convex case. For instance (see, e.g., \cite{ghadimi2012optimal,lan2012optimal,nedic2014stochastic,rosasco2020convergence}), for problems where the functional $f$ has Lipschitz continuous gradients and the stochastic gradients $\nabla F(x,\xi)$ have bounded magnitude, both SPG and SSG can achieve the ergodic rate of $O(1/\sqrt{K})$ in convex case and $O(1/K)$ in strongly convex case by using subsequent averaging of the obtained iterates, where longer stepsizes are suggested but often chosen in relation to the Lipschitz constant or gradient bound. In practice, proximal gradient iteration is typically faster than subgradient iteration \cite{nesterov2009primal,nesterov2014convergent,parikh2014proximal}, while computing the proximal mapping could be much more expensive than calculating a subgradient with additional constraints. 

The aforementioned concerns have motivated us to shift our focus to the alternating direction method of multipliers (ADMM). As a popular Lagrangian-based approach widely used in modern applications \cite{boyd2011distributed,glowinski2014alternating,glowinski2017splitting}, the ADMM often exhibits faster convergence than the traditional primal-dual type methods \cite{bertsekas2016nonlinear,luenberger2016linear} and the
method of multipliers \cite{bertsekas2014constrained}, and specifically, has been advocated for solving the problem with PDE-constraint in the deterministic form of (\ref{eq: general model problem}) \cite{glowinski2020admm,glowinski2022application,song2018two,song2024admm}. By introducing an auxiliary variable $z\in U$, we can reformulate the problem (\ref{eq: general model problem}) as the following equality constrained problem:
\vspace{-2pt}
\begin{equation}\label{eq: reformulated model problem}
	\underset{u\in U_{ad}}{\min}~  f(u) + g(z), 
	\enspace \text{s.~t.}\enspace u=z.
\vspace{-4pt}
\end{equation}
For the above constrained problem, the augmented Lagrangian functional $\mathcal{L}_{\rho}$ is given by
\vspace{-2pt}
\begin{equation}\label{eq: Lagrangian functional}
	\mathcal{L}_{\rho}(u,z,\lambda) = 
	f(u) + g(z) -\langle \lambda, u-z\rangle + \frac{\rho}{2}\|u-z\|^{2},
	\vspace{-2pt}
\end{equation}
where $\lambda \in U (= U^{\ast})$ is the Lagrange multiplier associated to the constraint $u=z$. Following \cite{glowinski2020admm,glowinski2022application}, a classic ADMM iteration can be unified as
\begin{subequations}\label{eq: classic ADMM}
	\begin{align}
		u_{k+1} 
		&= 
		\underset{u\in  U_{ad}}{\arg\min}~ \mathcal{L}_{\rho}(u,z_{k},\lambda_{k}),
		\label{eq: classic u-subproblem}\\
		z_{k+1} 
		&= 
		\underset{z\in  U}{\arg\min}~ \mathcal{L}_{\rho}(u_{k+1},z,\lambda_{k}),
		\label{eq: classic z-subproblem}\\
		\lambda_{k+1} 
		&= \lambda_{k} -\rho(u_{k+1} - z_{k+1}).
		\label{eq: classic lambda-subproblem}
	\end{align}
\end{subequations}
Utilizing the separable structure of (\ref{eq: reformulated model problem}), the ADMM iteration (\ref{eq: classic ADMM}) alternately executes minimization steps for the primal variable $u$ and $z$, and then updates the dual variable $\lambda$ by a ascent step. An important feature of the ADMM iteration (\ref{eq: classic ADMM}) is that it decouples the smooth $f$ and nonsmooth $g$, thus allowing us to treated each component individually in the iterations. The strong convergence and the worst-case $O(1/K)$ convergence rate have been established in \cite{glowinski2020admm,glowinski2022application} for the application of ADMM to different PDE-constrained problems. Since the functional $f$ is usually complicated because of the inherent PDE-constraint, several nested internal iterations and criteria have been proposed to solve the $u$-subproblem (\ref{eq: classic u-subproblem}) and the efficiency has been numerically validated.

\subsection{Algorithmic framework}
\label{subsec: algorithmic framework}
In order to leverage the structure of (\ref{eq: general model problem}) and take advantage of ADMM, we present a stochastic linearized framework of ADMM by incorporating the SA approach, which is given in Algorithm~\ref{alg:Stochastic Linearized ADMM}.

\begin{algorithm}
\caption{Faster Stochastic ADMM for (\ref{eq: general model problem})}
\label{alg:Stochastic Linearized ADMM}
\begin{algorithmic}
\STATE{{\bf Initialization:} Choose initial points $(u_{0},z_{0})=(v_{0},s_{0}) \in U\times U$, and set the parameter $\mu\in (0,1)$ }
\FOR{$k = 0,1,2,\cdots$}
\STATE{Choose the parameters $\rho_{k},~\eta_{k},~m_{k}>1$ and $\theta_{k}\geq 1$.}
\STATE{Generate independent realizations $\xi_{k,1},\cdots,\xi_{k,m_{k}} \in \Xi_{k}$, and compute the average
            \vspace{-7pt}
			\begin{equation}\label{eq: average SG}
				G_{k} = \frac{1}{m_{k}}\sum_{i=1}^{m_{k}}\nabla F(v_{k},\xi_{k,i}).
				\vspace{-16pt}
			\end{equation}}
\STATE{Perform updates:
\vspace{-10pt}
			\begin{subequations}
				\begin{align}
					s_{k+1} 
					&= 
					\underset{z\in  U}{\arg\min}~ \left\lbrace 
					g(z) + \left\langle \lambda_{k}, z\right\rangle +\frac{\rho_{k}}{2}\|z-v_{k}\|^{2} \right\rbrace,
					\label{eq: z-subproblem}\\
					v_{k+1} 
					&= 
					\underset{u\in  U_{ad}}{\arg\min}~ \left\lbrace 
					\left\langle G_{k}-\lambda_{k}, u\right\rangle +\frac{\rho_{k}}{2}\|u-s_{k+1}\|^{2} +\frac{\eta_{k}}{2}\|u-v_{k}\|^{2} \right\rbrace,
					\label{eq: u-subproblem}\\
					\psi_{k+1} 
					&= \psi_{k} -\mu\rho_{k}(v_{k+1} - s_{k+1}),
					\label{eq: lambda-subproblem}\\
					u_{k+1} 
					&= (1-\theta_{k}^{-1})u_{k} + \theta_{k}^{-1}v_{k+1},
					\label{eq: u-update}  \\
					z_{k+1} 
					&= (1-\theta_{k}^{-1})z_{k} + \theta_{k}^{-1}s_{k+1},
					\label{eq: z-update}\\
					\lambda_{k+1} 
					&= \psi_{k+1} -\mu\rho_{k}\theta_{k}(u_{k+1} - z_{k+1}).
					\label{eq: lambda-update}
				\end{align}	
			\end{subequations}}
			\vspace{-17pt}
\ENDFOR
\end{algorithmic}
\end{algorithm}

%We now make a few comments about the above algorithm.
Instead of solving the $u$-subproblem (\ref{eq: classic u-subproblem}) inexactly with inner-layer iterations, we let $G_{k}$ be a stochastic approximation to $\nabla f(v_{k})$ by independently calling SFO $m_{k}$ times, and then apply a stochastic linearization to (\ref{eq: classic u-subproblem}) in (\ref{eq: u-subproblem}) at the current iterate $v_{k}$, which makes an important modification to the update in (\ref{eq: classic u-subproblem}) adhering to the SA approach as in (\ref{eq: classic SA}). In view of this, the update (\ref{eq: u-subproblem}) can be carried out by projecting onto the admissible set $U_{ad}$. Hence, the decoupled subproblems (\ref{eq: z-subproblem}) and (\ref{eq: u-subproblem}) can be easily solved for the simple function $g$ and the favorably structured set $U_{ad}$.

In general, a larger batch size $m_{k}$ incurs more computation of $G_{k}$ and leads to a lower stochastic variance, which plays a key role in accelerating the convergence of the algorithm. Drawing on the well-known optimal first-order method given by Nesterov \cite{nesterov1983method} for minimizing deterministic smooth convex optimization problems with Lipschitz gradients and its variants \cite{beck2009fast,tseng2008accelerated}, the parameter $\theta_{k}$ is also crucial for accelerating the iterations in the nonergodic sense. To this end, the augmented parameter $\rho_{k}$ and the proximal parameter $\eta_{k}$ are adaptively chosen based on $\theta_{k}$.

\subsection{Related works}
\label{subsec: related works}
Stochastic ADMM has received significant attention and been developed for the large-scale and stochastic optimizations in the field of data science \cite{han2022survey,lin2022alternating}. In the prior works of \cite{ouyang2013stochastic,suzuki2013dual}, the smooth $f$ has been linearized by using the stochastic gradient computed at $x_{k}$ with a single SFO. It was shown that such methods can converge at the ergodic rate $O(1/\sqrt{K})$ for the objective value \cite{suzuki2013dual} or for the combination of the objective value and feasibility violation \cite{ouyang2013stochastic} in convex case with bounded stochastic gradients. When the functional $f$ is further required to be strongly convex, the ergodic convergence rate can be $O(\log K/K)$. With the bounded stochastic gradients and dual iterates, the work of \cite{azadi2014towards} has provided a weighted average of all the iterates and sharpened the ergodic rate to $O(1/K)$ for the strongly convex case. Under the conditions of the Lipschitz smooth functional $f$ and the bounded stochastic gradient variance, the $O(1/K)$ contribution from the smooth $f$ in \cite{ouyang2013stochastic,suzuki2013dual} can be refined to $O(1/K^{2})$ in \cite{azadi2014towards}, but the worst convergence rate is still $O(1/\sqrt{K})$. Due to the noise in the stochastic process, the aforementioned stochastic ADMM methods suffer from slower convergence rates compared to the performance of ADMM for deterministic cases. Later on, the success of variance reduction in stochastic gradient methods greatly propelled the research on stochastic ADMM for the
the empirical risk minimization problem, in which the functional $f$ is replaced by the sample average approximation (SAA) 
\vspace{-7pt}
\begin{equation}\label{eq: SAA functional}
	\bar{f} = \frac{1}{N}\sum_{i=1}^{N} F(u,\xi_{i})
	\vspace{-7pt}
\end{equation} 
based on a limited number of sampled realizations $\{\xi_{i}\}_{i=1}^{N}$ with $\xi_{i} = \xi(\theta_{i}), \theta_{i}\in \Omega$. By using in a specific manner the previously estimated gradients or the periodically computed full gradient, the stochastic ADMM can recover the deterministic convergence rates. For example, when the functional $f$ is Lipschitz smooth, the $O(1/K)$ ergodic rate and the linear ergodic rate are provably guaranteed, respectively, by equipping with stochastic average gradient (SAG) in convex case \cite{zhong2014fast} and with stochastic variance reduced gradient (SVRG) in strongly convex case \cite{zheng2016fast}, both of which use meticulous step sizes that rely on the Lipschitz constant.

In spite of its wide applicability in finite-dimensional Hilbert spaces, the study on stochastic ADMM is still unexploited for composite problems modeled by PDE-constrained optimization under uncertainty. To our best knowledge, most applications of ADMM in PDE-constrained optimization are limited to deterministic problems. For the linear parabolic state constrained optimal control problem, under the condition that $f$ is Lipschitz smooth and strongly convex, the work of \cite{glowinski2020admm} has established the strong convergence and the $O(1/K)$ worst-case convergence rate in the ergodic sense for the sum of the objective value and feasibility, as well as for the primal and dual residuals in the nonergodic sense. In \cite{glowinski2022application}, an inexact ADMM has been designed with a nested internal iteration and a self-acting criterion for solving the control constrained parabolic optimal control problem. At each iteration, the $u$-subproblem (\ref{eq: classic u-subproblem}) has been solved iteratively and inexactly, which can reduce the cost stemming from the immanent PDE-constraint. In accordance with the conditions in \cite{glowinski2020admm}, it shows that the inexact ADMM also enjoys the $O(1/K)$ convergence in both the ergodic and nonergodic senses for the violation of first-order optimality and for the residuals of $z$ and $\lambda$, respectively. 

We would like to mention that several notable works on stochastic gradient (SG) methods have been developed for PDE-constrained optimization under uncertainty. The convergence of SG with bias has been analyzed in \cite{geiersbach2019projected} for the smooth convex case, and in \cite{geiersbach2023stochastic} for the smooth nonconvex case, both governed by random elliptic equations. Besides, with unbiased gradient estimate, the work of \cite{geiersbach2019projected} shows that SG can achieve the $O(1/K)$ nonergodic convergence rate in the strongly convex case and the $O(1/\sqrt{K})$ ergodic convergence rate in the convex case in terms of the objective value. An extension of SPG to the nonsmooth nonconvex case has been provided in \cite{geiersbach2021stochastic} with asymptotic convergence results for semilinear elliptic constraints. By integrating mini-batch or SAGA gradient estimators with tailored spatial discretization, it is shown in \cite{martin2021pde} that SG performs better in the strongly convex case for the Lipschitz-smooth objective functional constrained by random elliptic equations. In addition, an adaptive SG iteration has been proposed in \cite{cao2022adaptive}, scaling step sizes by the cumulative sum of gradients sampled to date for parabolic optimal control problems with random coefficients.

\subsection{Our goals}
\label{subsec: our goals}
Our first purpose is to study the convergence of the stochastic ADMM algorithmic framework applied to the problem (\ref{eq: general model problem}) in its reformulated form (\ref{eq: reformulated model problem}). Specifically, we focus on the convergence analysis in the strongly convex setting of the smooth functional $f$, proving the global strong convergence and deriving faster nonergodic rates in terms of both functional values and feasibility violation. Additionally, we extend the nonergodic rate analysis to the general convex setting. It should be mentioned that, in practice, the outputs of ADMM-type methods are  nonergodic results, while most existing analyses use ergodic convergence rates as the criterion. In fact, the ergodic averaging may destroy the structural properties of the iterates, such as sparsity or low-rankness, and the nonergodic iterates often perform much better.

The literature on nonergodic rate of convergence results still stays limited, especially regarding the aforementioned measures for stochastic problems. A relevant work is the sophisticated two-layer iteration in \cite{fang2017faster}, which integrates Nesterov extrapolation and SVRG in the inner iteration, and elaborately preserves snapshots of primal variables in the outer iteration. Compared with the work of \cite{fang2017faster}, our framework is markedly more simplified, requiring only mild restrictions on the parameters to achieve faster and nonergodic convergence results. Moreover, the implementation of our method does not require periodic full-gradient computations, since it may be impossible or intractable to precisely calculate the expectation \eqref{eq: stochastic form of $f$} necessary for accessing the full gradient with the PDE-constraints.

Another purpose of this work is to explore our framework for the application to PDE-constrained optimization under uncertainty. In addition to offering the upper bounds on the functional values and feasibility violation of the iterates, we also provide the efficiency validation with the bounds on large deviations for solving related problems. 
%for solving the PDE-constrained optimization problems under uncertainty
%the optimal control problems constrained by elliptic or parabolic PDEs with random coefficients. 
To the best of our knowledge, no such large-deviation results have been obtained before for stochastic ADMM and for solving PDE-constrained optimization problems under uncertainty, although some results exist for solving nonsmooth composite convex problems solved by SG methods \cite{ghadimi2012optimal,lan2012optimal}.

\subsection{Organization}
The paper is organized as follows. In Section~\ref{sec: convergence analysis}, the convergence results of the stochastic ADMM framework are presented after some preliminaries. In Section~\ref{sec: application to PDE-constrained optimization under uncertainty}, we consider the model problems of PDE-constrained optimization under uncertainty and derive the large-deviation properties of the framework. The efficiency of our framework is also demonstrated by some numerical results in Section~\ref{sec: application to PDE-constrained optimization under uncertainty}. Then some conclusions are given in Section~\ref{sec:conclusions}.

\section{Convergence analysis}
\label{sec: convergence analysis}
In this section, we analyze the convergence of Algorithm~\ref{alg:Stochastic Linearized ADMM} under different settings. 

\subsection{Preliminaries}
\label{subsec: preliminaries}
Before proceeding with our analysis, we will introduce some notations and some basic facts on convex analysis and stochastic processes. Some assumptions and useful results will also be presented in this subsection.

Given a probability space $(\Omega,\mathcal{F},\mathbb{P})$, we let $\Xi_{k} = \{\xi_{k,1},\ldots,\xi_{k,m_{k}} \}$ denote the set of independent realizations at the $k$-th iteration. We use the $\sigma$-algebras $\left\lbrace \mathcal{F}_{k} \right\rbrace $ for the natural filtration with $\mathcal{F}_{k} = \sigma \left(\Xi_{0},\ldots,\Xi_{k-1} \right)$, and $\mathbb{E}\left[ \phi|\mathcal{F}_{k}\right]$ for the expectation of $\phi$ conditional on $\mathcal{F}_{k}$ for any random variable $\phi$ on $(\Omega,\mathcal{F},\mathbb{P})$.
% $\Xi_{[k]} = \left(\Xi_{1},\cdots,\Xi_{k} \right)$ as the history

From the convexity of the problem (\ref{eq: general model problem}), the existence of the optimal solution $(u^{\ast},z^{\ast})$ to problem (\ref{eq: reformulated model problem}) and an optimal dual Lagrange multiplier $\lambda^{\ast}$ associated with the constraint $u=z$ is guaranteed under the standard constraint qualification for problem (\ref{eq: reformulated model problem}) (see, e.g., \cite{bauschke2020convex,rockafellar1970convex}). With the KKT conditions of problem (\ref{eq: reformulated model problem}),  the optimality of the saddle point $(u^{\ast},z^{\ast},\lambda^{\ast})$ can be characterized by the following variational inequality
\vspace{-2pt}
\begin{equation}\label{eq: VI of reformulated MP}
	g(z) - g(z^{\ast}) + \left\langle \nabla f(u^{\ast}), u-u^{\ast}  \right\rangle
	-  \left\langle \lambda^{\ast}, u-z  \right\rangle \geq 0, 
	\enspace  \forall~u,z,\lambda \in U.
	\vspace{-2pt}
\end{equation}

Throughout this paper,  the following assumption will be made for our analysis.
\begin{assumption}\label{assumption: strongly convex and Lipschitz continuous}
	The functional $f$ is $\alpha$-strongly convex and has Lipschitz continuous gradient with the constant $L$, i.e., 
	\vspace{-5pt}
	\begin{subequations}
		\begin{align}
			&f(u) + \left\langle \nabla f(u), v-u \right\rangle + \frac{\alpha}{2}\|v-u\|^{2} 
			\leq f(v), \enspace \forall~u,v\in U,
			\label{eq: strongly convex}
			\\
			&\|\nabla f(u) -  \nabla f(v)\| \leq L\|u-v\|, \enspace \forall~u,v\in U.
			\label{eq: Lipschitz continuous}
		\end{align}
	\end{subequations}
\end{assumption}
\vspace{-5pt}
\begin{remark}
	In the strongly convex case ($\alpha >0$), the parameter choice in Algorithm~\ref{alg:Stochastic Linearized ADMM} depends only on the strong convexity modulus $\alpha$, which suffices to obtain a faster nonergodic convergence rate. In the general convex case ($\alpha =0$), however, our parameter selection requires the knowledge of the gradient Lipschitz constant $L$, either exactly or via a reliable evaluation.
	
	It is worth noting that, for PDE-constrained optimization problems, although the gradient Lipschitz constant $L$ often exists theoretically, it is typically difficult to compute or evaluate accurately in practice. Moreover, in many applications the regularization term is strongly convex (e.g., $L^{2}$ Tikhonov regularization), and its regularization parameter is prescribed a priori. In such settings, the regularization parameter can serve as a known lower bound for the strong convexity parameter $\alpha$ of the objective functional.
\end{remark}

\begin{assumption}\label{assumption: unbiased estimate and bounded variance}
	The stochastic gradient $\nabla F(u,\xi)$ is unbiased and variance-bounded with the constant $\varrho$, i.e.,
	\begin{subequations}
		\begin{align}
			&\mathbb{E}\left[ \nabla F(u,\xi)  \right]
			= \nabla f(u), \enspace \forall~u\in U,
			\label{eq: unbiased estimate}
			\\
			&\mathbb{E}\left[ \left\| \nabla F(u,\xi) -\nabla f(u)\right\|^{2} \right]
			\leq \varrho^{2}, \enspace \forall~u\in U.
			\label{eq: bounded variance}
		\end{align}
	\end{subequations}
\end{assumption}
%remark
\begin{remark}\label{remark: Oracle assumptions}
	The conditions on SFO in Assumption~\ref{assumption: unbiased estimate and bounded variance} are standard in the literature of the SA approach, used since the seminal work of Robbins and Monro \cite{robbins1951stochastic}. The first condition (\ref{eq: unbiased estimate}) entails that, at each iteration the stochastic gradient $\nabla F(v_{k},\xi)$ is an unbiased estimator of the gradient $\nabla f(v_k)$ of the smooth functional $f$ at $v_{k}$. The second condition (\ref{eq: bounded variance}) is a global variance restriction on the noise caused by the stochastic approximation, and is weaker than the common condition often assumed in the study of SA approach (see, e.g. \cite{lan2012validation,nedic2014stochastic,nemirovski2009robust}), namely the boundedness of $\mathbb{E}\left[ \left\|\nabla F(u,\xi) \right\| ^{2} \right]$. While Assumption~\ref{assumption: unbiased estimate and bounded variance} is a classical assumption in the convergence analysis of the SA approach, it can be relaxed to permit the vanishing bias and the growth-bounded second moments of the stochastic gradient, as in recent work \cite{geiersbach2019projected}.
\end{remark}

Now we establish a technical result for a single iteration of Algorithm~\ref{alg:Stochastic Linearized ADMM}.
\begin{lemma}\label{lemma: one iteration}
	Let the sequences $\left\lbrace  (v_{k}, s_{k}, \psi_{k}) \right\rbrace$ and $\left\lbrace  (u_{k}, z_{k}, \lambda_{k}) \right\rbrace$ be generated by Algorithm~\ref{alg:Stochastic Linearized ADMM}. Then for any $\lambda \in U$, we have
		\vspace{-2pt}
		\begin{equation}\label{eq: one iteration}
		\begin{aligned}
			&
			f(v_{k+1}) + g(s_{k+1}) - f(u^{\ast}) - g(z^{\ast}) 
			-\left\langle \lambda, v_{k+1}-s_{k+1} \right\rangle 
			+\frac{\rho_{k}}{2}\left\|v_{k+1}-s_{k+1}\right\|^{2} 
			\\
			\leq 
			& \left\langle \nabla f(v_{k})-G_{k} , v_{k+1}-u^{\ast}\right\rangle
			+\left(\frac{L}{2}+\frac{\rho_{k}}{2(1-\mu)} 
			-\frac{\eta_{k}+\rho_{k}}{2}\right)\left\|v_{k+1}-v_{k}\right\|^{2} \\
			& 
			+\frac{\eta_{k}+\rho_{k}}{2}\left( 
			\left\|v_{k}-u^{\ast}\right\|^{2} -
			\left\|v_{k+1}-u^{\ast}\right\|^{2} \right) 
			+\frac{1}{2 \mu \rho_{k}} 
			\left(
			\left\|\psi_{k}-\lambda\right\|^{2}-\left\|\psi_{k+1}-\lambda\right\|^{2}
			\right) \\
			&
			-\frac{\alpha}{2}\left\|v_{k}-u^{\ast}\right\|^{2}
			+
			\left\langle \lambda_{k}-\psi_{k},v_{k+1}-s_{k+1}\right\rangle.
				\vspace{-2pt}
		\end{aligned}
	\end{equation}
\end{lemma}
\begin{proof}
	From the optimality conditions of the subproblems (\ref{eq: z-subproblem}) and (\ref{eq: u-subproblem}), it follows that 
		\vspace{-3pt}
	\begin{equation}\label{eq: one iteration 1}
		g(z^{\ast})-g\left(s_{k+1}\right)
		+\left\langle \lambda_{k}-\rho_{k}\left(v_{k}-s_{k+1}\right), z^{\ast}-s_{k+1} \right\rangle \geq 0,
			\vspace{-3pt}
	\end{equation} 
	and
		\vspace{-3pt}
	\begin{equation}\label{eq: one iteration 2}
		\left\langle G_{k}-\lambda_{k}+\rho_{k}\left(v_{k+1}-s_{k+1}\right)
		+\eta_{k}\left(v_{k+1}-v_{k}\right), 
		u^{\ast}-v_{k+1}\right\rangle \geq 0.
			\vspace{-2pt}
	\end{equation}
	By using the equality $u^{\ast} = z^{\ast}$, we have from the above inequalities that
	\begin{equation}\label{eq: one iteration 3}
		\begin{aligned}
			& g\left(s_{k+1}\right)-g\left(z^{\ast}\right)+\left\langle G_{k}, v_{k+1}-u^{\ast} \right\rangle \\
			\leq 
			& \left\langle \lambda_{k}, v_{k+1}-s_{k+1} \right\rangle 
			+ \rho_{k}\left\langle v_{k+1}-s_{k+1}, s_{k+1}-v_{k+1} \right\rangle \\
			& 
			+\rho_{k}\left\langle v_{k+1}-v_{k}, z^{\ast}-s_{k+1} \right\rangle
			+\eta_{k}\left\langle v_{k+1}-v_{k}, u^{\ast}-v_{k+1} \right\rangle \\
			= 
			& \left\langle \lambda_{k}, v_{k+1}-s_{k+1} \right\rangle -\rho_{k}\left\|v_{k+1}-s_{k+1}\right\|^{2}  
			+\rho_{k}\left\langle v_{k+1}-v_{k}, v_{k+1}-s_{k+1} \right\rangle \\
			& 
			+ \left(\rho_{k}+\eta_{k}\right)\left\langle v_{k+1}-v_{k}, u^{\ast}-v_{k+1} \right\rangle.
		\end{aligned}
	\end{equation}
	Hence, it follows that 
	\vspace{-2pt}
	\begin{equation}\label{eq: one iteration 4}
		\begin{aligned}
			%			& \mathcal{L}\left(v_{k+1}, s_{k+1}, \lambda\right)
			%			-\mathcal{L}\left(u^{\ast}, z^{\ast}, \lambda\right) \\
			%			= 
			& 
			f(v_{k+1}) + g(s_{k+1}) - f(u^{\ast}) - g(z^{\ast}) 
			-\left\langle \lambda, v_{k+1}-s_{k+1} \right\rangle \\
			\leq 
			& 
			\left\langle\lambda_{k}-\lambda, v_{k+1}-s_{k+1}\right\rangle 
			+\rho_{k}\left\langle v_{k+1}-v_{k},v_{k+1}-s_{k+1}
			\right\rangle \\
			& 
			+ \left(\rho_{k}+\eta_{k}\right) \left\langle v_{k+1}-v_{k}, u^{\ast}-v_{k+1}\right\rangle-\rho_{k}\left\|v_{k+1}-s_{k+1}\right\|^{2}\\
			& 
			+\left\langle G_{k}, u^{\ast}-v_{k+1}\right\rangle 
			+ f\left(v_{k+1}\right)-f\left(u^{\ast}\right) .
			\vspace{-4pt}
		\end{aligned}
	\end{equation}	
	
	By virtue of the identity
	\vspace{-2pt}
	\begin{equation}\label{eq: 3-points identity}
		2\langle a-b,b-c\rangle = \|a-c\|^{2} - \|a-b\|^{2} - \|b-c\|^{2},
			\vspace{-2pt}
	\end{equation}
	and from (\ref{eq: lambda-subproblem}), we can get that	
	\vspace{-3pt}
	\begin{equation}\label{eq: one iteration 5}
		\begin{aligned}
			\left\langle \lambda_{k}-\lambda, v_{k+1}-s_{k+1} \right\rangle 
			= &
			\left\langle \lambda_{k}-\psi_{k}, v_{k+1}-s_{k+1} \right\rangle 
			+ \frac{1}{\mu \rho_{k}}\left\langle \psi_{k}-\lambda, \psi_{k}-\psi_{k+1}\right\rangle \\
			= &
			\left\langle \lambda_{k}-\psi_{k}, v_{k+1}-s_{k+1} \right\rangle 
			+\frac{\mu \rho_{k}}{2}\left\|v_{k+1}-s_{k+1}\right\|^{2} \\
			& +
			\frac{1}{2 \mu \rho_{k}}\left(\left\|\psi_{k}-\lambda\right\|^{2} 
			-\left\|\psi_{k+1}-\lambda\right\|^{2}\right) .
			\vspace{-4pt}
		\end{aligned}
	\end{equation}	
	In addition, using Young's inequality and the identity (\ref{eq: 3-points identity}) gives
	\vspace{-3pt}
	\begin{equation}\label{eq: one iteration 6}
		\left\langle v_{k+1}-v_{k}, v_{k+1}-s_{k+1} \right\rangle 
		\leq \frac{1}{2(1-\mu)}\left\|v_{k+1}-v_{k}\right\|^{2} 
		+\frac{1-\mu}{2}\left\|v_{k+1}-s_{k+1}\right\|^{2},
		\vspace{-3pt}
	\end{equation}
	and 
	\vspace{-3pt}
	\begin{equation}\label{eq: one iteration 7}
		\left\langle v_{k+1}-v_{k}, u^{\ast}-v_{k+1} \right\rangle 
		= 
		\frac{1}{2}\left(  \left\|v_{k}-u^{\ast}\right\|^{2}-\left\|v_{k+1}-u^{\ast}\right\|^{2}-\left\|v_{k}-v_{k+1}\right\|^{2} \right) ,
		\vspace{-3pt}
	\end{equation}
	respectively. Moreover, it follows from the strong convexity and gradient Lipschitz continuity of $f$ that
	\vspace{-5pt}
	\begin{equation}\label{eq: one iteration 8}
		\begin{aligned}
			& f\left(v_{k+1}\right) - f\left(u^{\ast}\right) 
			+\left\langle G_{k}, u^{\ast}-v_{k+1} \right\rangle \\
			\leq 
			& f\left(v_{k}\right) - f\left(u^{\ast}\right) 
			+ \left\langle \nabla f(v_{k})-G_{k}, v_{k+1}-v_{k} \right\rangle 
			+ \frac{L}{2}\left\|v_{k+1}-v_{k}\right\|^{2} 
			+ \left\langle G_{k}, u^{\ast}-v_{k} \right\rangle \\
			\leq 
			& \left\langle \nabla f(v_{k})-G_{k}, v_{k+1}-u^{\ast}\right\rangle 
			+ \frac{L}{2}\left\|v_{k+1}-v_{k}\right\|^{2} 
			- \frac{\alpha}{2}\left\|v_{k}-u^{\ast}\right\|^{2},
			\vspace{-5pt}
		\end{aligned}
	\end{equation}
	which, together with (\ref{eq: one iteration 4}), (\ref{eq: one iteration 5}), (\ref{eq: one iteration 6}) and (\ref{eq: one iteration 7}), completes the proof.
	%end	
\end{proof}

Below, we will specify the values of the parameters $\rho_{k}$, $\eta_{k}$ and $\theta_{k}$, and provide the convergence analysis of Algorithm~\ref{alg:Stochastic Linearized ADMM} through the above results. 
%For the sake of presentation, 
To present our analysis concisely, we let
\vspace{-7pt}
\begin{equation}\label{eq: notations nu and kappa}
	\nu = \underset{u,v \in U_{ad}}{\max}\|u-v\|^{2}, \enspace
	\kappa = \left\lceil \frac{L(1-\mu)}{\eta(1-\mu)-\rho\mu} \right\rceil,
	%	\enspace \text{and} \enspace 
	%	\phi_{k} = \frac{2-\mu}{ (2-\mu) \eta-\mu \rho -(k+1)^{-1}L(2-\mu)}.
	\vspace{-5pt}
\end{equation}
and
\vspace{-5pt}
\begin{equation}\label{eq: notation phi}
	\phi_{k} = \frac{1-\mu}{ \eta (1-\mu) -\mu \rho -k^{-1}L(1-\mu)}, 
	\enspace  \forall~k\geq 1.
	\vspace{-3pt}
\end{equation}

\subsection{Strongly convex case}
\label{subsec: nonergodic convergence of strongly convex case}
In this subsection, we aim to analyze Algorithm~\ref{alg:Stochastic Linearized ADMM} for strongly convex case to present the global strong convergence and faster convergence rate in a nonergodic sense. To establish the results on the iteration sequences themselves, we let the parameter $\theta_{k}$ be adaptive to the iteration number $k$ and solved by the following equation
\begin{equation}\label{eq: adaptive omega}
	\theta_{k+1}^2 = \theta_{k}^2 + \theta_{k+1},  
	\enspace\text{with}\enspace  
	\theta_{0} = 1 \enspace\text{and}\enspace \theta_{-1} = 0.
\end{equation}
Besides, we set the parameters $\rho_{k}$ and $\eta_{k}$ to  
\begin{equation}\label{eq: nonergodic parameters}
	\rho_{k} = \rho \theta_{k},\enspace \text{and} \enspace
	\eta_{k} = \eta \theta_{k},
\end{equation}
with the positive constants $\rho+\eta \leq \alpha$ and $\eta(1-\mu)>\rho\mu$ for all $k$.

Now we are ready to derive the global strong convergence of the iterates $\{u_{k}\}$ and $\{z_{k}\}$.
\begin{theorem}\label{theorem: strong convergence for nonergodic setting}
	Let the sequence $\left\lbrace  (u_{k}, z_{k}, \lambda_{k}) \right\rbrace$ be generated by Algorithm~\ref{alg:Stochastic Linearized ADMM} with the parameter settings (\ref{eq: adaptive omega}) and (\ref{eq: nonergodic parameters}). Under Assumptions~\ref{assumption: strongly convex and Lipschitz continuous} and \ref{assumption: unbiased estimate and bounded variance}, it holds that
	\vspace{-3pt}
	\begin{equation}\label{eq: strong convergence for nonergodic setting}
		\lim_{k \rightarrow \infty} \mathbb{E}
		\left[\left\|u_{k}-u^{\ast}\right\| \right]=0, 
		\enspace \text{and} \enspace
		\lim _{k \rightarrow \infty} \mathbb{E}
		\left[\left\|z_{k}-z^{\ast}\right\|\right]=0.
		\vspace{-3pt}
	\end{equation}	
\end{theorem} 
\begin{proof}
	By using Jensen's inequality and the convexity of $f$ and $g$, it follows from (\ref{eq: u-update}), (\ref{eq: z-update}) and (\ref{eq: adaptive omega}) that
	\vspace{-3pt}
	\begin{equation*}\label{eq: strong convergence for nonergodic setting 0}
		\theta_{k}^{2}\left( f(u_{k+1}) + g(z_{k+1}) \right)  
		\leq
		\theta_{k-1}^{2}\left( f(u_{k}) + g(z_{k}) \right) 
		+
		\theta_{k}\left( f(v_{k+1}) + g(s_{k+1})  \right) 
		\vspace{-3pt}
	\end{equation*}
	and
	\vspace{-7pt}
	\begin{equation*}\label{eq: strong convergence for nonergodic setting 1}
		\theta_{k}^{2}\left\langle \lambda, u_{k+1}-z_{k+1} \right\rangle 
		= 
		\theta_{k-1}^{2}\left\langle \lambda, u_{k}-z_{k} \right\rangle +
		\theta_{k}	\left\langle \lambda, v_{k+1}-s_{k+1} \right\rangle,	
		\vspace{-3pt}
	\end{equation*}
	which additionally yields that
	\vspace{-3pt}
	\begin{equation*}\label{eq: strong convergence for nonergodic setting 2}
		\begin{aligned}
			&\theta_{k}^{2}\left( f(u_{k+1}) + g(z_{k+1}) - f(u) - g(z) 
			- \left\langle \lambda, u_{k+1}-z_{k+1} \right\rangle \right)  \\
			&
			-
			\theta_{k-1}^{2}\left( f(u_{k}) + g(z_{k}) - f(u) - g(z) 
			- \left\langle \lambda, u_{k}-z_{k} \right\rangle \right) \\
			\leq
			&
			\theta_{k}\left( f(v_{k+1}) + g(s_{k+1}) - f(u) - g(z) -\left\langle \lambda, v_{k+1}-s_{k+1} \right\rangle \right).
			\vspace{-5pt}
		\end{aligned}
	\end{equation*}
	Letting $u=u^{\ast}$ and $z=z^{\ast}$ in the above inequality, and summing up it from $k=0$ through $K-1$, together with Lemma~\ref{lemma: one iteration} and the parameters given in (\ref{eq: nonergodic parameters}), we have after some arrangement that
	\begin{equation}\label{eq: strong convergence for nonergodic setting 3}
		\begin{aligned}
			&\theta_{K-1}^{2}\left( f(u_{K}) + g(z_{K}) - f(u^{\ast}) - g(z^{\ast}) 
			- \left\langle \lambda, u_{K}-z_{K} \right\rangle \right) 
			+
			\sum_{k=0}^{K-1} \frac{\rho\theta_{k}^{2}}{2}\left\|v_{k+1}-s_{k+1}\right\|^{2}  
			\\
			\leq
			&
			\sum_{k=0}^{K-1}
			\theta_{k}\left( 
			\left\langle \nabla f(v_{k})-G_{k}, v_{k+1}-u_{k}\right\rangle 
			+\left(\frac{L}{2}+\frac{\rho\theta_{k}}{2(1-\mu)} -\frac{(\eta+\rho)\theta_{k}}{2}\right)\left\|v_{k+1}-v_{k}\right\|^{2}
			\right) \\
			&
			+
			\sum_{k=0}^{K-1}
			\left( 
			\frac{(\eta+\rho)\theta_{k}^{2}}{2}\left( 
			\left\|v_{k}-u^{\ast}\right\|^{2} - \left\|v_{k+1}-u^{\ast}\right\|^{2} \right) 
			-\frac{\alpha\theta_{k}}{2}\left\|v_{k}-u^{\ast}\right\|^{2}
			\right) 
			\\
			&
			+
			\sum_{k=0}^{K-1} \theta_{k} \left\langle \nabla f(v_{k})-G_{k}, u_{k}-u^{\ast}\right\rangle
			+\sum_{k=0}^{K-1}\theta_{k}
			\left\langle \lambda_{k}-\psi_{k}, v_{k+1}-s_{k+1} \right\rangle 
			\\
			&
			+
			\frac{1}{2 \mu \rho}\left\|\lambda_{0}-\lambda\right\|^{2}  .
		\end{aligned}
	\end{equation}
	
	Using Young's inequality and the facts $\theta_{k} \geq (k+1)/2$ (see, e.g., \cite{beck2009fast}) and $\eta (1-\mu) - \mu\rho > \theta_{2\kappa}^{-1}L(1-\mu) > 0$, we obtain that 
	\vspace{-5pt}
	\begin{equation*}\label{eq: strong convergence for nonergodic setting 4}
		\begin{aligned}
			&
			\left\langle \nabla f(v_{k})-G_{k}, v_{k+1}-u_{k}\right\rangle 
			+\left(\frac{L}{2}+\frac{\rho\theta_{k}}{2(1-\mu)} -\frac{(\eta+\rho)\theta_{k}}{2}\right)\left\|v_{k+1}-v_{k}\right\|^{2}
			\\
			\leq
			&
			\frac{1-\mu}{2\left( \eta (1-\mu)-\mu \rho \right)\theta_{k}}\left\|\nabla f(v_{k})-G_{k}\right\|^{2} 
			+
			\frac{L}{2}\left\|v_{k+1}-v_{k}\right\|^{2}
			\\
			\leq
			&
			\frac{ \phi_{\theta_{2\kappa}} }{ 2\theta_{k} }
			\left\|\nabla f(v_{k})-G_{k}\right\|^{2} 
			+ \frac{L\nu}{2},
			\vspace{-5pt}
		\end{aligned}
	\end{equation*}
	for any $k\geq 0$, where the second inequality we have used the notations defined in (\ref{eq: notations nu and kappa}) and (\ref{eq: notation phi}). Particularly, noting that $\phi_{\theta_{k}} < \phi_{\theta_{2\kappa}}$ for all $k+1 \geq \kappa$, we use Young's inequality again and have  
	\vspace{-8pt}
	\begin{equation*}\label{eq: strong convergence for nonergodic setting 5}
		\begin{aligned}
			&
			\left\langle \nabla f(v_{k})-G_{k}, v_{k+1}-u_{k}\right\rangle 
			+\left(\frac{L}{2}+\frac{\rho\theta_{k}}{2(1-\mu)} -\frac{(\eta+\rho)\theta_{k}}{2}\right)\left\|v_{k+1}-v_{k}\right\|^{2}
			\\
			\leq
			&
			\frac{1-\mu}{2\left( \left( \eta (1-\mu)-\mu \rho \right)\theta_{k} - L(1-\mu) \right) }\left\|\nabla f(v_{k})-G_{k}\right\|^{2} 
			\\
			\leq
			&
			\frac{ \phi_{\theta_{2\kappa}} }{ 2\theta_{k} }
			\left\|\nabla f(v_{k})-G_{k}\right\|^{2}. 
		\end{aligned}
	\end{equation*}
	Then we conclude from the above observations and (\ref{eq: adaptive omega}) that 
	\begin{equation}\label{eq: strong convergence for nonergodic setting 6}
		\begin{aligned}
			& \sum_{k=0}^{K-1}
			\theta_{k}\left( 
			\left\langle \nabla f(v_{k})-G_{k}, v_{k+1}-u_{k}\right\rangle 
			+\left(\frac{L}{2}+\frac{\rho\theta_{k}}{2(1-\mu)} -\frac{(\eta+\rho)\theta_{k}}{2}\right)\left\|v_{k+1}-v_{k}\right\|^{2}
			\right)  \\
			\leq
			&
			\sum_{k=0}^{K-1} \frac{\phi_{\theta_{2\kappa}}}{2} 
			\left\|\nabla f(v_{k})-G_{k}\right\|^{2} 
			+ \frac{L\nu}{2}\sum_{k=0}^{2\kappa-1} \left( \theta_{k}^{2} - \theta_{k-1}^{2} \right) \\
			\leq
			&
			\frac{\phi_{\theta_{2\kappa}}}{2} \sum_{k=0}^{K-1} \left\|\nabla f(v_{k})-G_{k}\right\|^{2} 
			+ \frac{L \nu \theta_{2\kappa-1}^{2}}{2} .
			\vspace{-5pt}
		\end{aligned}
	\end{equation}
	Also, using (\ref{eq: adaptive omega}) with $\eta + \rho \leq \alpha$ gives
	\vspace{-3pt}
	\begin{equation}\label{eq: strong convergence for nonergodic setting 7}
		\begin{aligned}
			&\sum_{k=0}^{K-1}\left( 
			\frac{(\eta+\rho)\theta_{k}^{2}}{2}\left( 
			\left\|v_{k}-u^{\ast}\right\|^{2} - \left\|v_{k+1}-u^{\ast}\right\|^{2} \right) 
			-\frac{\alpha\theta_{k}}{2}\left\|v_{k}-u^{\ast}\right\|^{2} 
			\right) \\
			\leq
			&\sum_{k=0}^{K-1}\left( 
			\frac{(\eta+\rho)\theta_{k-1}^{2}}{2} \left\|v_{k}-u^{\ast}\right\|^{2}
			-
			\frac{(\eta+\rho)\theta_{k}^{2}}{2} \left\|v_{k+1}-u^{\ast}\right\|^{2}
			\right) 
			\leq
			0  .
			\vspace{-5pt}
		\end{aligned}
	\end{equation}
	From (\ref{eq: u-update}), (\ref{eq: z-update}) and (\ref{eq: adaptive omega}), we have 
	\begin{equation*}\label{eq: strong convergence for nonergodic setting 8}
		\begin{aligned}
			\left\|u_{k+1}-z_{k+1}\right\|^{2} 
			= 
			&
			(1-\theta_{k}^{-1})^2\left\|u_{k}-z_{k}\right\|^{2} + \theta_{k}^{-2}\left\|v_{k+1}-s_{k+1}\right\|^{2}\\
			& 
			+2\theta_{k}^{-1}(1-\theta_{k}^{-1})
			\left\langle u_{k}-z_{k},v_{k+1}-s_{k+1}\right\rangle,
		\end{aligned}
	\end{equation*}
	which, together with (\ref{eq: lambda-update})  implies  
	\vspace{-3pt}
	\begin{equation}\label{eq: strong convergence for nonergodic setting 9}
		\begin{aligned}
			& \sum_{k=0}^{K-1}\theta_{k}
			\left\langle \lambda_{k}-\psi_{k}, v_{k+1}-s_{k+1} \right\rangle
			=
			\sum_{k=0}^{K-1} - \mu\rho \theta_{k}^{2}(\theta_{k}-1)
			\left\langle u_{k}-z_{k},v_{k+1}-s_{k+1}\right\rangle\\
			= 
			& 
			- \frac{\mu\rho}{2}\sum_{k=0}^{K-1}
			\left( 
			\theta_{k}^{4}\left\|u_{k+1}-z_{k+1}\right\|^{2}
			-\theta_{k}^{4}(1-\theta_{k}^{-1})^2\left\|u_{k}-z_{k}\right\|^{2} 
			-\theta_{k}^{2}\left\|v_{k+1}-s_{k+1}\right\|^{2} 
			\right) \\
			\leq 
			& -\frac{\mu \rho}{2}\sum_{k=0}^{K-1}
			\theta_{k}^{2}\left\|v_{k+1}-s_{k+1}\right\|^{2}. 
			\vspace{-3pt}
		\end{aligned}
	\end{equation}
	Moreover, note that from (\ref{eq: u-update}), (\ref{eq: z-update}) and (\ref{eq: adaptive omega})
	\vspace{-5pt}
	\begin{equation*}\label{eq: strong convergence for nonergodic setting 10}
		\begin{aligned}
			\theta_{k}^{2}\left\|v_{k+1}-s_{k+1}\right\|^{2} 
			=
			\left\|\theta_{k}^{2}(u_{k+1}-z_{k+1}) - \theta_{k-1}^{2}(u_{k}-z_{k})\right\|^{2},
		\end{aligned}
	\end{equation*}
	and thus by Cauchy-Schwarz inequality,  we obtain that
	\vspace{-3pt}
	\begin{equation}\label{eq: strong convergence for nonergodic setting 11}
		\begin{aligned}
			\sum_{k=0}^{K-1} 
			\theta_{k}^{2}\left\|v_{k+1}-s_{k+1}\right\|^{2} 
			%			& 
			%			\geq 
			%			\sum_{k=0}^{K-1}
			%			\left(\theta_{k}^{2}\left\|u_{k+1}-z_{k+1}\right\|-\theta_{k-1}^{2}\left\|u_{k}-z_{k}\right\|
			%			\right)^2 \\
			& 
			\geq \frac{1}{K}
			\left(\sum_{k=0}^{K-1}
			\left( \theta_{k}^{2}\left\|u_{k+1}-z_{k+1}\right\|-\theta_{k-1}^{2}\left\|u_{k}-z_{k}\right\| \right) 
			\right)^2 \\
			& 
			\geq 
			\frac{\theta_{K-1}^4}{K}\left\|u_{K}-z_{K}\right\|^{2} .
			\vspace{-3pt}
		\end{aligned}
	\end{equation} 
	In addition, from the strong convexity of the functional $f$ and the optimality condition (\ref{eq: VI of reformulated MP}), we clearly have 
	\begin{equation}\label{eq: strong convergence for nonergodic setting 12}
		f(u_{K}) + g(z_{K}) - f(u^{\ast}) - g(z^{\ast}) 
		-\left\langle \lambda^{\ast}, u_{K}-z_{K} \right\rangle
		\geq
		\frac{\alpha}{2}\|u_{K} -u^{\ast} \|^{2}.
	\end{equation}
	Combining (\ref{eq: strong convergence for nonergodic setting 3}) with (\ref{eq: strong convergence for nonergodic setting 12}), substituting (\ref{eq: strong convergence for nonergodic setting 6}), (\ref{eq: strong convergence for nonergodic setting 7}), (\ref{eq: strong convergence for nonergodic setting 9}) and (\ref{eq: strong convergence for nonergodic setting 11}) into it, and arranging terms bring us
	\begin{equation}\label{eq: strong convergence for nonergodic setting 13}
		\begin{aligned}
			&\frac{\alpha \theta_{K-1}^{2}}{2} 
			\left\|u_{K}-u^{\ast}\right\|^{2} 
			+ 
			\frac{\rho(1-\mu)\theta_{K-1}^{4}}{2 K} 
			\left\|u_{K}-z_{K}\right\|^{2}  
			\\
			\leq\,
			&
			\frac{\phi_{\theta_{2\kappa}}}{2} \sum_{k=0}^{K-1} \left\|\nabla f(v_{k})-G_{k}\right\|^{2} 
			+ \sum_{k=0}^{K-1} \theta_{k} \left\langle \nabla f(v_{k})-G_{k}, u_{k}-u^{\ast}\right\rangle
            + \frac{L \nu \theta_{2\kappa-1}^{2}}{2} 
			+ \frac{1}{2 \mu \rho}\left\|\lambda_{0}-\lambda\right\|^{2}.
		\end{aligned}
	\end{equation}
	
	Since $v_{k},  u_{k}$ are $\mathcal{F}_{k}$-measurable and $\Xi_{k}$ is independent of $\Xi_{0},\ldots,\Xi_{k-1}$, we have
	\begin{equation}\label{eq: strong convergence for nonergodic setting 14}
		\begin{aligned}
			\mathbb{E}\left[ \left\langle \nabla f(v_{k})-G_{k}, u_{k}-u^{\ast}\right\rangle \right]
			=
			&
			\mathbb{E}\left[ 
			\mathbb{E}\left[ \left. \left\langle \nabla f(v_{k})-G_{k}, u_{k}-u^{\ast}\right\rangle \right| \mathcal{F}_{k} \right]  
			\right]\\
			=
			&
			\mathbb{E}\left[
			\left.\mathbb{E}\left[ \left\langle \nabla f(v)-G_{k}, u-u^{\ast}\right\rangle \right]\right|_{v=v_{k}, u=u_{k}}
			\right] \\
			=
			&
			0.
		\end{aligned}
	\end{equation}
	In addition, noting that the independence of the realizations $\Xi_{k}$ yields that for any $\xi_{k,i},\xi_{k,j} \in \Xi_{k}$
	\vspace{-10pt}
	\begin{equation*}\label{eq: strong convergence for nonergodic setting 15}
		\left.\mathbb{E}\left[
		\left\langle \nabla f(u)-\nabla F(u,\xi_{k,i}), 
		\nabla f(u)-\nabla F(u,\xi_{k,j})\right\rangle 
		\right]\right|_{u=u_{k}} = 0,
	\end{equation*}
	we obtain that
	\vspace{-5pt}
	\begin{equation}\label{eq: strong convergence for nonergodic setting 16}
		\begin{aligned}
			\mathbb{E}\left[ 
			\left.\left\|\nabla f(v_{k})-G_{k}\right\|^{2}\right|\mathcal{F}_{k} \right]
			&
			=
			\frac{1}{m_{k}^{2}}
			\mathbb{E}\left[ 
			\left.\left\| \sum_{i=1}^{m_{k}}\left( 
			\nabla f(v_{k})-\nabla F(v_{k},\xi_{k,i}) \right) \right\|^{2}\right|\mathcal{F}_{k} \right] \\
			% &
			% =
			% \frac{1}{m_{k}^{2}}
			% \left.\mathbb{E}\left[ 
			% \left\| \sum_{i=1}^{m_{k}}\left( 
			% \nabla f(v)-\nabla F(v,\xi_{k,i}) \right) \right\|^{2} \right]\right|_{v=v_{k}}
			% \\
			&
			=
			\frac{1}{m_{k}^{2}}\sum_{i=1}^{m_{k}}
			\left.\mathbb{E}\left[ 
			\left\| \nabla f(v)-\nabla F(v,\xi_{k,i})\right\|^{2} \right]\right|_{v=v_{k}}
			\\	
			&
			\leq \frac{\varrho^{2}}{m_{k}}.		
			\vspace{-5pt}					
		\end{aligned}
	\end{equation}
	
	Taking the expectation on both sides of (\ref{eq: strong convergence for nonergodic setting 13}), it follows from (\ref{eq: strong convergence for nonergodic setting 14}) and (\ref{eq: strong convergence for nonergodic setting 16}) that 
		\begin{equation}\label{eq: strong convergence for nonergodic setting 17}
		\begin{aligned}
			&
			\frac{\alpha \theta_{K-1}^{2}}{2} 
			\mathbb{E}\left[ \left\|u_{K}-u^{\ast}\right\|^{2} \right] 
			+ 
			\frac{\rho(1-\mu)\theta_{K-1}^{4}}{2 K} 
			\mathbb{E}\left[ \left\|u_{K}-z_{K}\right\|^{2} \right] 
			\\
			\leq\,
			&
			\frac{\phi_{\theta_{2\kappa}}\varrho^{2}}{2}\sum_{k=0}^{K-1}\frac{1}{m_{k}}
			+ \frac{L D \theta_{2\kappa-1}^{2}}{2} 
			+ \frac{1}{2 \mu \rho}\left\|\lambda_{0}-\lambda^{\ast}\right\|^{2}.
		\end{aligned}
	\end{equation}
	Using the facts $\theta_{k-1}\geq k/2$ and $\sum_{k=0}^{K-1} 1/m_{k} \leq K$ gives
	\begin{equation}\label{eq: strong convergence for nonergodic setting 18}
		\mathbb{E}\left[ \left\|u_{K}-u^{\ast}\right\|^{2} \right] 
		\leq
		\frac{4}{\alpha K}
		\left( 
		\phi_{\theta_{2\kappa}}\varrho^{2}
		+ \frac{L \nu \theta_{2\kappa-1}^{2}}{K} 
		+ \frac{1}{\mu \rho K}\left\|\lambda_{0}-\lambda^{\ast}\right\|^{2}
		\right) ,
	\end{equation}
	and 
	\begin{equation}\label{eq: strong convergence for nonergodic setting 19}
		\mathbb{E}\left[ \left\|u_{K}-z_{K}\right\|^{2} \right]
		\leq
		\frac{16}{\rho(1-\mu) K^{2}}
		\left( 
		\phi_{\theta_{2\kappa}}\varrho^{2}
		+ \frac{L \nu \theta_{2\kappa-1}^{2}}{K} 
		+ \frac{1}{\mu \rho K}\left\|\lambda_{0}-\lambda^{\ast}\right\|^{2}
		\right).
	\end{equation}
	Then, in view of $K \rightarrow \infty$, we have that
	\begin{equation*}
		\lim_{k \rightarrow \infty} \mathbb{E}
		\left[\left\|u_{k}-u^{\ast} \right\| \right]=0, 
		\enspace \text{and} \enspace
		\lim _{k \rightarrow \infty} \mathbb{E}
		\left[\left\|u_{k} - z_{k} \right\|\right]=0.
	\end{equation*}
	Combining the above two equalities with $u^{\ast} = z^{\ast}$, we obtain that
	\begin{equation*}
		\lim _{k \rightarrow \infty} \mathbb{E}
		\left[\left\|z_{k} - z^{\ast} \right\|\right]=0,
	\end{equation*}
	which completes the proof.
	%
	%\par\hspace{1ex}
\end{proof}

Although not the main goal of this work, the above results yield the nonergodic $O(1/K)$ rates for the squared distances between the primal iterates $u_{k}, z_{k}$ and the optimal primal solutions $u^{\ast}, z^{\ast}$, respectively, along with the nonergodic $O(1/K^{2})$ rate for the squared feasibility violation. Before establishing the faster nonergodic results, we set the batch sizes $m_{k}$ to be reciprocally summable, i.e., there exists a constant $M > 0$ such that
\begin{equation}\label{eq: batch size bound}
	\sum_{k=0}^{\infty} \frac{1}{m_{k}} = \frac{1}{M} < \infty.
\end{equation}

\begin{theorem}\label{theorem: convergence rate for nonergodic setting}
	Let the sequence $\left\lbrace  (u_{k}, z_{k}, \lambda_{k}) \right\rbrace$ be generated by Algorithm~\ref{alg:Stochastic Linearized ADMM} with the parameter settings (\ref{eq: adaptive omega}), (\ref{eq: nonergodic parameters}) and (\ref{eq: batch size bound}). 
	Under Assumptions~\ref{assumption: strongly convex and Lipschitz continuous} and \ref{assumption: unbiased estimate and bounded variance}, it holds that for any $c \geq 2\|\lambda^{\ast}\|$
	\begin{equation}\label{eq: convergence rate for nonergodic strongly convex setting}
		\mathbb{E} \left[ f(u_{K}) + g(z_{K}) - f(u^{\ast})- g(z^{\ast}) \right]
		\leq \frac{C_{1}}{K^{2}}, 
		\enspace \text{and~} \enspace
		\mathbb{E} \left[ \| u_{K}- z_{K} \| \right]
		\leq \frac{2C_{1}}{cK^{2}},
	\end{equation}	
	where
	\vspace{-10pt}
	\begin{equation*}
		C_{1} = \frac{2\varrho^{2}\phi_{\theta_{2\kappa}}}{M} 
		+ 2L\nu \theta_{2\kappa-1}^{2}
		+ \frac{4\left( c^{2} + \|\lambda_{0}\|^{2} \right)}{\mu\rho}.
		\vspace{-5pt}
	\end{equation*}	
\end{theorem}
\begin{proof}
	By substituting (\ref{eq: strong convergence for nonergodic setting 6}), (\ref{eq: strong convergence for nonergodic setting 7}) and (\ref{eq: strong convergence for nonergodic setting 9}) into (\ref{eq: strong convergence for nonergodic setting 3}), it follows that
	\begin{equation}\label{eq: convergence rate for nonergodic setting 0}
		\begin{aligned}
			&
			\theta_{K-1}^{2}\left( f(u_{K}) + g(z_{K}) - f(u^{\ast}) - g(z^{\ast}) 
			-\left\langle \lambda, u_{K}-z_{K} \right\rangle \right)  \\
			\leq\,
			&
			\frac{\phi_{\theta_{2\kappa}}}{2} \sum_{k=0}^{K-1} \left\|\nabla f(v_{k})-G_{k}\right\|^{2} 
			+ \sum_{k=0}^{K-1} \theta_{k} \left\langle \nabla f(v_{k})-G_{k}, u_{k}-u^{\ast}\right\rangle 
			+ \frac{L \nu \theta_{2\kappa-1}^{2}}{2} 
			+ \frac{1}{2 \mu \rho}\left\|\lambda_{0}-\lambda\right\|^{2} .
			\vspace{-5pt}
		\end{aligned}
	\end{equation}
	Then, taking expectation over both sides of the above inequality, and using the fact $\theta_{k-1}\geq k/2$ we have from (\ref{eq: strong convergence for nonergodic setting 14}) and (\ref{eq: strong convergence for nonergodic setting 16}) that
	\begin{equation}\label{eq: convergence rate for nonergodic setting 1}
		% \begin{aligned}
			% &
			\mathbb{E}\left[  f(u_{K}) + g(z_{K}) - f(u^{\ast}) - g(z^{\ast}) 
			-\left\langle \lambda, u_{K}-z_{K} \right\rangle \right] \\
			\leq
			% &
			\frac{1}{K^2}\left( 
			\frac{2\varrho^{2}\phi_{\theta_{2\kappa}}}{M}
			+ 2L\nu \theta_{2\kappa-1}^{2} 
			+ \frac{2}{\mu\rho}\left\|\lambda_{0}-\lambda\right\|^{2} 
			\right).
		% \end{aligned}
	\end{equation}
	Clearly, taking the supremum on both sides of (\ref{eq: convergence rate for nonergodic setting 1}) over $\|\lambda\| \leq c$ gives
	\vspace{-5pt}
	\begin{equation}\label{eq: convergence rate for nonergodic setting 2}
		\mathbb{E}\left[  f(u_{K}) + g(z_{K}) - f(u^{\ast}) - g(z^{\ast}) \right]
		+c \mathbb{E}\left[ \left\| u_{K}-z_{K}\right\|  \right]
		\leq
		\frac{C_{1}}{K^{2}},
		\vspace{-5pt}
	\end{equation} 
	and we get the first results in (\ref{eq: convergence rate for nonergodic strongly convex setting}). Moreover, for the saddle point $(u^{\ast},z^{\ast},\lambda^{\ast})$, it follows from (\ref{eq: convergence rate for nonergodic setting 2}) that
	\vspace{-10pt}
	\begin{equation}\label{eq: convergence rate for nonergodic setting 3}
		\begin{aligned}
			c \mathbb{E}\left[\left\| u_{K}-z_{K}\right\|  \right] 
			\leq\,
			&
			\mathbb{E}\left[ f(u^{\ast}) + g(z^{\ast}) - f(u_{K}) - g(z_{K}) \right] 
			+ \frac{C_{1}}{K^{2}} \\
			\leq\,
			& 
			\mathbb{E}\left[-\left\langle \lambda^{\ast}, u_{K}-z_{K} \right\rangle \right]
			+ \frac{C_{1}}{K^{2}} \\
			\leq \,
			&
			\frac{c}{2}\mathbb{E}\left[\left\| u_{K}-z_{K}\right\|  \right] 
			+ \frac{C_{1}}{K^{2}},
			\vspace{-5pt}
		\end{aligned}
	\end{equation}
	with $c>2\|\lambda^{\ast}\|$, which implies the second results in (\ref{eq: convergence rate for nonergodic strongly convex setting}). The proof is therefore complete.
	%
	%\par\hspace{1ex}	
\end{proof}

\subsection{General convex case}
\label{subsec: nonergodic convergence convex case}
In this subsection, we consider the general convex case, i.e., $\alpha = 0$, and illustrate the efficiency of our algorithmic framework by providing the fast and nonergodic rate of convergence results.  
We set the parameters to  
\begin{equation}\label{eq: nonergodic parameters in convex case}
	\rho_{k} = \rho,\enspace
	\eta_{k} = \eta,\enspace \text{and} \enspace
	\theta_{k}= k+1,
\end{equation}
where the positive constants satisfy $\eta >\mu\rho/(1-\mu) + L$. Now we present the nonergodic convergence rate for Algorithm~\ref{alg:Stochastic Linearized ADMM} in the general convex case.

\begin{theorem}\label{theorem: nonergodic convergence rate in convex case}
	Let the sequence $\left\lbrace  (z_{k}, u_{k}, \lambda_{k}) \right\rbrace$ be generated by Algorithm~\ref{alg:Stochastic Linearized ADMM} with the parameter settings (\ref{eq: batch size bound}) and (\ref{eq: nonergodic parameters in convex case}).	Under Assumptions~\ref{assumption: strongly convex and Lipschitz continuous} and \ref{assumption: unbiased estimate and bounded variance}, it holds that for any $c\geq 2\|\lambda^{\ast}\|$
	\vspace{-10pt}
	\begin{equation}\label{eq: convergence rate of nonergodic convex setting}
		\mathbb{E} \left[ f(u_{K}) + g(z_{K}) - f(u^{\ast})- g(z^{\ast}) \right]
		\leq \frac{C_{2}}{K}, 
		\enspace \text{and~} \enspace
		\mathbb{E} \left[ \| u_{K} - z_{K} \| \right]
		\leq \frac{2C_{2}}{cK},
		\vspace{-10pt}
	\end{equation}	
	where
	\vspace{-10pt}
	\begin{equation*}
		C_{2} = \frac{\varrho^{2}\phi_{1}}{2M} 
		+ \frac{\eta+\rho}{2} \left\|u_{0}-u^{\ast}\right\|^{2} 
		+ \frac{c^{2} + \|\lambda_{0}\|^{2} }{\mu\rho}.
		\vspace{-5pt}
	\end{equation*}
\end{theorem} 
\begin{proof}
	With the parameters given in (\ref{eq: nonergodic parameters in convex case}), and by using essentially the same arguments as for (\ref{eq: strong convergence for nonergodic setting 3}), it follows that 
	
	\begin{equation}\label{eq: nonergodic convergence rate in convex case 1}
		\begin{aligned}
			&\theta_{K-1}\left( f(u_{K}) + g(z_{K}) - f(u^{\ast}) - g(z^{\ast}) 
			- \left\langle \lambda, u_{K}-z_{K} \right\rangle \right)  
			\\
			\leq
			&
			\sum_{k=0}^{K-1}
			\left\langle \nabla f(v_{k})-G_{k}, v_{k+1}-u_{k}\right\rangle 
			+
			\left(\frac{L}{2}+\frac{\rho}{2(1-\mu)} -\frac{\eta+\rho}{2}\right)
			\left\|v_{k+1}-v_{k}\right\|^{2}\\
			&
			+
			\sum_{k=0}^{K-1}
			\frac{\eta+\rho}{2}\left( 
			\left\|v_{k}-u^{\ast}\right\|^{2} - \left\|v_{k+1}-u^{\ast}\right\|^{2} \right) 
			+
			\frac{1}{2 \mu \rho}\left\|\lambda_{0}-\lambda\right\|^{2} \\
			&
			+
			\sum_{k=0}^{K-1} \left\langle \nabla f(v_{k})-G_{k}, u_{k}-u^{\ast}\right\rangle
			+
			\sum_{k=0}^{K-1} 
			\left\langle \lambda_{k}-\psi_{k},v_{k+1}-s_{k+1}\right\rangle 
			% \\
			% &
			- 
			\sum_{k=0}^{K-1} \frac{\rho}{2}\left\|v_{k+1}-s_{k+1}\right\|^{2} .
		\end{aligned}
	\end{equation}
	Using Young's inequality, we get that
	\begin{equation}\label{eq: nonergodic convergence rate in convex case 2}
				% \begin{aligned}
				% &
			\left\langle \nabla f(v_{k})-G_{k}, v_{k+1}-u_{k}\right\rangle 
			+
			\left(\frac{L}{2}+\frac{\rho}{2(1-\mu)} -\frac{\eta+\rho}{2}\right)\left\|v_{k+1}-v_{k}\right\|^{2}
						% \\
			\leq
				% &\,
			\frac{\phi_{1}}{2}\left\|\nabla f(v_{k})-G_{k}\right\|^{2} 
					% \end{aligned}
	\end{equation}
	with $\eta >\mu\rho/(1-\mu) + L$. From (\ref{eq: u-update}), (\ref{eq: z-update}), and  (\ref{eq: nonergodic parameters in convex case}), we use the similar argument for deriving (\ref{eq: strong convergence for nonergodic setting 9}) to obtain that
	\begin{equation}\label{eq: nonergodic convergence rate in convex case 3}
		\begin{aligned}
			&
			\sum_{k=0}^{K-1} 
			\left\langle \lambda_{k}-\psi_{k},v_{k+1}-s_{k+1}\right\rangle
			=
			\sum_{k=0}^{K-1} -\mu\rho (\theta_{k}-1)
			\left\langle u_{k}-z_{k},v_{k+1}-s_{k+1}\right\rangle\\
			= 
			& 
			- \frac{\mu\rho}{2}\sum_{k=0}^{K-1} 
			\left( 
			\theta_{k}^{2}\left\|u_{k+1}-z_{k+1}\right\|^{2} 
			-\theta_{k}^{2}(1-\theta_{k}^{-1})^{2}\left\|u_{k}-z_{k}\right\|^{2}
			- \left\|v_{k+1}-s_{k+1}\right\|^{2} 
			\right) \\
			% = 
			% & 
			% {\color{red}- \frac{\mu\rho}{2}\sum_{k=0}^{K-1}
			% \left( 
			% \theta_{k}^{2}\left\|u_{k+1}-z_{k+1}\right\|^{2} 
			% -\theta_{k-1}^{2}\left\|u_{k}-z_{k}\right\|^{2}
			% - \left\|v_{k+1}-s_{k+1}\right\|^{2}
			% \right) }\\
			\leq 
			& 
			-\frac{\mu\rho}{2}\sum_{k=0}^{K-1}\left\|v_{k+1}-s_{k+1}\right\|^2.
		\end{aligned}
	\end{equation}
	Furthermore, we have from $u_{0}=v_{0}$
	\begin{equation}\label{eq: nonergodic convergence rate in convex case 4}
		\begin{aligned}
			&\sum_{k=0}^{K-1}
			\left( 
			\frac{\eta+\rho}{2}\left( 
			\left\|v_{k}-u^{\ast}\right\|^{2} - \left\|v_{k+1}-u^{\ast}\right\|^{2} \right) \right)
			\leq
			\frac{\eta+\rho}{2} \left\|u_{0}-u^{\ast}\right\|^{2}.
		\end{aligned}
	\end{equation}
	Substituting (\ref{eq: nonergodic convergence rate in convex case 2}), (\ref{eq: nonergodic convergence rate in convex case 3}), and (\ref{eq: nonergodic convergence rate in convex case 4}) into (\ref{eq: nonergodic convergence rate in convex case 1}) gives
	\begin{equation}\label{eq: nonergodic convergence rate in convex case 5}
		\begin{aligned}
			&
			\theta_{K-1}\left( f(u_{K}) + g(z_{K}) - f(u^{\ast}) - g(z^{\ast}) 
			-\left\langle \lambda, u_{K}-z_{K} \right\rangle \right)  \\
			\leq \,
			&
			\frac{\phi_{1}}{2} \sum_{k=0}^{K-1} \left\|\nabla f(v_{k})-G_{k}\right\|^{2} 
			+ \sum_{k=0}^{K-1}\left\langle \nabla f(v_{k})-G_{k}, u_{k}-u^{\ast}\right\rangle 
			% \\
			% &
			+ \frac{\eta+\rho}{2} \left\|u_{0}-u^{\ast}\right\|^{2} 
			+ \frac{1}{2 \mu \rho}\left\|\lambda_{0}-\lambda\right\|^{2}.
		\end{aligned}
	\end{equation}
	Then, by taking expectation on both sides of the above inequality, it follows from (\ref{eq: strong convergence for nonergodic setting 14}) and (\ref{eq: strong convergence for nonergodic setting 16}) that 
	\begin{equation}\label{eq: nonergodic convergence rate in convex case 6}
		\begin{aligned}
			&
			\mathbb{E}\left[  f(u_{K}) + g(z_{K}) - f(u^{\ast}) - g(z^{\ast}) 
			-\left\langle \lambda, u_{K}-z_{K} \right\rangle \right] 
			\\
			\leq\,
			&
			\frac{1}{K}\left( 
			\frac{\varrho^{2}\phi_{1}}{2M} 
			+ \frac{\eta+\rho}{2} \left\|u_{0}-u^{\ast}\right\|^{2} 
			+ \frac{1}{2\mu\rho}\left\|\lambda_{0}-\lambda\right\|^{2} 
			\right).
		\end{aligned}
	\end{equation}
	
	Letting $\lambda = \lambda^{\ast}$ in the above inequality, the desired results follow from the same arguments as we presented at the end of the proof of Theorem~\ref{theorem: convergence rate for nonergodic setting}.	
	%
	%\par\hspace{1ex}	 
\end{proof}

\begin{remark}
	In the general convex case, the work of \cite{fang2017faster} has shown that the lower complexity bound of ADMM-type methods for the separable equality constrained
	nonsmooth convex problems is exactly $O(1/K)$, which indicates that the nonergodic result of our algorithm is optimal.
\end{remark}

\section{Application to PDE-constrained optimization under uncertainty}
\label{sec: application to PDE-constrained optimization under uncertainty}
With the proven convergence results, it is promising to consider the algorithmic framework for various applications, especially certain model problems in PDE-constrained optimization, including the optimal control of elliptic equations whose coefficients are subject to uncertainty. Besides providing convergence rates to estimate the efficiency of the algorithmic framework, in this section we complement the results with bounds on the probabilities of large deviations for the application of our framework to the model problem of elliptic optimal control.

\subsection{Elliptic optimal control problem}
\label{subsec: random elliptic PDE-constrained optimization}
In this subsection, we consider the optimal sparse distributed control of elliptic PDE with random coefficients.
%arising from
%\subsubsection{Model problem and large-deviation properties}
The model problem is given by 
\vspace{-10pt}
\begin{equation}\label{eq: elliptic optimal control problem}
	\begin{aligned}
		\underset{u\in U_{ad}}{\min}~ 
		\mathbb{E}\left[\frac{1}{2}\left\| y(u,\xi(\omega)) - y_{d}\right\|_{L^{2}(D)}^{2} \right] + \frac{\alpha}{2}\left\| u \right\|_{L^{2}(D)}^{2} + \beta \left\| u \right\|_{L^{1}(D)},  
	\vspace{-5pt}
\end{aligned}
\end{equation}
with $y(u,\xi(\omega))$ being the solution to the following elliptic state equation 
\vspace{-2pt}
\begin{equation}\label{eq: elliptic state equation}
	\left\{
	\begin{aligned}
		-\nabla \cdot\left(a(x, \omega) \nabla y(x, \omega)\right) & =u(x), & & (x, \omega) \in D \times \Omega, \\
		y(x, \omega) & =0, & & (x, \omega) \in \partial D \times \Omega,
	\end{aligned}
	\right.
	\vspace{-3pt}
\end{equation}
where the physical domain $D \subseteq \mathbb{R}^{d}(d \geq 1)$ is a bounded Lipschitz domain with boundary $\partial D$, and the diffusion coefficient $a: D \times \Omega \rightarrow \mathbb{R}$ is a uniformly bounded and positive random field, i.e., there exist two constants $a_{\min}, a_{\max}$ such that for all $(x,\omega) \in D\times \Omega$,
\vspace{-2pt}
\begin{equation}\label{eq: uniformly bounded and positive random field}
	0 < a_{\min} < a(x,\omega) <a_{\max} < \infty.
	\vspace{-3pt}
\end{equation}
The admissible set $U_{ad}$ is defined by
\vspace{-2pt}
\begin{equation}\label{eq: elliptic admissible set}
	U_{ad} = 
	\left\lbrace u\in L^{2}(D); -\infty < u_{\min} \leq u(x) \leq u_{\max} < +\infty,~a.e.~x \in  D\right\rbrace, 
	\vspace{-3pt}
\end{equation}
where $ u_{\min},  u_{\max}$ are two given constants. 
%The existence and uniqueness of the solution to problem (\ref{eq: elliptic optimal control problem}) in the deterministic case were shown in [~]. For the stochastic case, the same can be proved in a similar way as discussed in [~], and more details can be found in [~]. 

To proceed with discussions, we introduce 
\vspace{-2pt}
\begin{equation}\label{eq: special elliptic case}
	F(u,\xi)  
	= \frac{1}{2}\left\| y(u,\xi(\omega)) - y_{d}\right\|_{L^{2}(D)}^{2} + \frac{\alpha}{2}\left\| u \right\|_{L^{2}(D)}^{2},
	\enspace
	g(u) = \beta \left\| u \right\|_{L^{1}(D)},
	\vspace{-3pt}
\end{equation} 
then it is clear that the problem (\ref{eq: elliptic optimal control problem}) is a special case of (\ref{eq: general model problem}) with $f(u) 
= \mathbb{E}\left[ F(u,\xi) \right] $. By the standard argument (see, e.g., \cite{martin2019stochastic}), the functional $f$ satisfies the $\alpha$-strong convexity and the gradient Lipschitz continuity with constant $L = \alpha + C_{P}^{4}/a_{\min}^{2}$, where $C_{P}>0$ is the Poincar\'{e} constant depending only on the domain $D$. For any $\omega \in \Omega$, the stochastic gradient $\nabla F(u,\xi) $ of the smooth part of the objective functional in (\ref{eq: elliptic optimal control problem}) can be calculated by 
\vspace{-2pt}
\begin{equation}\label{eq: elliptic stochastic gradient}
	\nabla F(u,\xi) = \nabla F(u,\xi(\omega)) = \alpha u + p(\cdot,\omega),
	\vspace{-3pt}
\end{equation}
where the adjoint variable $p$ is the solution to the following adjoint equation
\vspace{-2pt}
\begin{equation}\label{eq: adjoint elliptic equation}
	\left\{\begin{aligned}
		-\nabla \cdot\left(a(x, \omega) \nabla p(x, \omega)\right) & =y(x, \omega)-y_{d}(x), & & (x, \omega) \in D \times \Omega, \\
		p(x, \omega) & =0, & & (x, \omega) \in \partial D \times \Omega.
	\end{aligned}\right.
	\vspace{-3pt}
\end{equation}

Now let us look into the details of the unbiasedness and variance boundedness of the stochastic gradient $\nabla F(u,\xi)$ given in (\ref{eq: elliptic stochastic gradient}). 
%\begin{proposition}
%	
%\end{proposition}
%{\bf Unbiased stochastic gradient and bounded variance.}
Multiplying both sides of the adjoint equation (\ref{eq: adjoint elliptic equation}) by $p$ and integrating over $D$, we have 
\begin{equation}\label{eq: variation of adjoint equation}
  \begin{aligned}
a_{\min}\left\| \nabla p(\cdot,\omega) \right\|_{L^{2}(D)}^{2}
 \leq \int_{D} a(x,\omega) \nabla p(x,\omega) \cdot \nabla p(x,\omega) ~\mathrm{d}x 
= \int_{D} \left( y(x,\omega) - y_{d}(x) \right) p(x,\omega) ~\mathrm{d}x.
    \end{aligned}
\end{equation}
Using the above inequality and Poincar\'{e} inequality gives
\begin{equation}\label{eq: boundedness of p}
	\left\| p(\cdot,\omega) \right\|_{L^{2}(D)} 
	\leq 
	\frac{C_{P}^{2}}{a_{\min}} \left\| y(\cdot,\omega) - y_{d} \right\|_{L^{2}(D)}
	\leq
	\frac{C_{P}^{2}}{a_{\min}} \left( \left\| y(\cdot,\omega) \right\|_{L^{2}(D)} + \left\| y_{d} \right\|_{L^{2}(D)} \right). 
\end{equation}
Similarly, we can multiply the state equation (\ref{eq: elliptic state equation}) by $y$ on both sides and use Poincar\'{e} inequality to obtain
\vspace{-2pt}
\begin{equation}\label{eq: boundedness of y}
	\left\| y(\cdot,\omega) \right\|_{L^{2}(D)} \leq 	
	\frac{C_{P}^{2}}{a_{\min}} \left\| u\right\|_{L^{2}(D)}.
	\vspace{-3pt}
\end{equation}
Moreover, since the set $U_{ad}$ is bounded, we clearly have for any $u \in U_{ad}$
\begin{equation}\label{eq: boundedness of u}
	\left\| u \right\|_{L^{2}(D)} \leq \max\left\{ |u_{\min}|, |u_{\max}|\right\} \sqrt{|D|},
\end{equation}
and then we readily get
\vspace{-2pt}
\begin{equation}\label{eq: boundedness of SG}
	\begin{aligned}
		\left\| \nabla F(u,\xi) \right\|_{L^{2}(D)} 
		&
		\leq 
		\alpha \left\| u \right\|_{L^{2}(D)} + \left\| p(\cdot,\omega) \right\|_{L^{2}(D)} \\
		&
		\leq
		\left( \frac{C_{P}^{4}}{a_{\min}^{2}} + \alpha \right) \max\left\{ |u_{\min}|, |u_{\max}|\right\}  \sqrt{|D|} + \frac{C_{P}^{2}}{a_{\min}}\left\| y_{d} \right\|_{L^{2}(D)} .
	\end{aligned}
	\vspace{-3pt}
\end{equation}
By Lebesgue's dominated convergence theorem, it is evident for any  $u \in U_{ad}$ that 
\begin{equation}\label{eq: unbiased gradient}
	\nabla f(u) = \mathbb{E}\left[ \nabla F(u,\xi) \right] 
	= \alpha u +  \mathbb{E}\left[ p(\cdot,\omega) \right].
\end{equation}
In addition, for any $u \in U_{ad}$, using Jensen's inequality and the inequalities (\ref{eq: boundedness of p}), (\ref{eq: boundedness of y}) and (\ref{eq: boundedness of u})  yields  
\vspace{-2pt}
\begin{equation}\label{eq: bounded gradient bais}
	\begin{aligned}
		\left\| \nabla f(u) - \nabla F(u,\xi) \right\|_{L^{2}(D)}
		&
		\leq
		\left\|  \mathbb{E} \left[  p(\cdot,\omega) \right]  \right\|_{L^{2}(D)} 
		+ \left\|  p(\cdot,\omega) \right\|_{L^{2}(D)} \\
		&
		\leq
		2 \left( 
		\frac{C_{P}^{4}}{a_{\min}^{2}} \max\left\{ |u_{\min}|, |u_{\max}|\right\}  \sqrt{|D|} + \frac{C_{P}^{2}}{a_{\min}}\left\| y_{d} \right\|_{L^{2}(D)} 
		\right) \\
		&
		=
		2Q.
	\end{aligned}
	\vspace{-3pt}
\end{equation}
%then using Jensen's inequality yields 
%\begin{equation*}
%	\mathbb{E}\left[ \exp\left( 
%	\frac{\left\| \nabla f(u) - \nabla F(u,\xi) \right\|_{L^{2}(D)}^{2} }
%	{4 A }
%	\right)  \right]
%	\leq
%	\exp(1),
%\end{equation*}
which implies that
\vspace{-2pt}
\begin{equation}\label{eq: elliptic finite variance}
	\mathbb{E}\left[ \left\| \nabla f(u) - \nabla F(u,\xi) \right\|_{L^{2}(D)}^{2} \right]
	\leq
	4 Q^{2}.
	\vspace{-3pt}
\end{equation}

%{\bf Large-deviation of Algorithm~\ref{alg:Stochastic Linearized ADMM}}
%For any $k$ and $i$, we have from (\ref{eq: bounded gradient bais}) that
%\begin{equation}\label{eq: large-deviation of Alg for ellpetic MP 0}
%	\left\| \nabla f(v_{k}) - G_{k}\right\|_{L^{2}(D)}^{2}  
%	\leq
%	\frac{1}{m_{k}^{2}}\sum_{i=1}^{m_{k}}
%	\left\| \nabla f(v_{k}) - \nabla F(v_{k},\xi_{k,i})\right\|_{L^{2}(D)}^{2}  
%	\leq
%	\frac{4 Q^{2}}{m_{k}},
%\end{equation}
%
%Using Cauchy-Schwarz inequality in (\ref{eq: convergence rate for nonergodic setting 0}) and the similar argument for (\ref{eq: convergence rate for nonergodic setting 2}) and (\ref{eq: convergence rate for nonergodic setting 3}), then 

Based on the above observation, we can establish the large-deviation results associated with
Algorithm~\ref{alg:Stochastic Linearized ADMM} for solving problem (\ref{eq: elliptic optimal control problem}), which evaluates the quality of the solutions obtained by a single run.

\begin{theorem}\label{theorem: large-deviation of Alg for ellpetic MP}
	Let the sequence $\left\lbrace  (u_{k}, z_{k}, \lambda_{k}) \right\rbrace$ be generated by Algorithm~\ref{alg:Stochastic Linearized ADMM} for solving problem (\ref{eq: elliptic optimal control problem}) with the parameters preset according to either (\ref{eq: adaptive omega}), (\ref{eq: nonergodic parameters}) and (\ref{eq: batch size bound}), or  (\ref{eq: batch size bound}) and (\ref{eq: nonergodic parameters in convex case}).
	\begin{enumerate}
		\item[$\mathrm{(a)}$] 
		For the strongly convex case, i.e., $\alpha>0$, it holds that  for any $\delta >0$
		{\small 
    \begin{equation}\label{eq: large-deviation of Alg for strongly convex ellpetic MP 0}
			\begin{aligned}
				&
				\mathrm{Prob} \left\lbrace
				f(u_{K}) + g(z_{K}) - f(u^{\ast}) - g(z^{\ast}) 
				<
				\frac{C_{3}}{K^{2}} + \frac{\delta C_{4} \sqrt{\ln K}}{K} \right\rbrace 
				\geq
				1-\left( \frac{1}{K} \right)^{\frac{\delta^{2}}{3}} , \\
				&
				\mathrm{Prob} \left\lbrace
				\left\| u_{K}-z_{K}\right\|_{L^{2}(D)}
				<
				\frac{2 C_{3}}{c K^{2}} + \frac{2 \delta C_{4} \sqrt{\ln K}}{c K} \right\rbrace 
				\geq
				1-\left( \frac{1}{K} \right)^{\frac{\delta^{2}}{3}},
			\end{aligned}
		\end{equation}}
		where  
		\vspace{-5pt}
		\begin{equation*}\label{eq: large-deviation of Alg for ellpetic MP 13}
			C_{3} =
			\frac{8 \phi_{\theta_{2\kappa}} Q^{2}}{M } 
			+ 2 \sqrt{|D|} (u_{\max} - u_{\min}) (\alpha + \frac{C_{P}^{4}}{a_{\min}}) \theta_{2\kappa-1}^{2}
			+ \frac{4 \left( \left\|\lambda_{0}\right\|_{L^{2}(D)}^{2} + c^{2} \right) }{ \mu \rho},
			\vspace{-3pt}
		\end{equation*}
		and
		\vspace{-5pt}
		\begin{equation*}\label{eq: large-deviation of Alg for ellpetic MP 14}
			C_{4} =
			2 (u_{\max} - u_{\min}) Q 
			\sqrt{ \frac{ 2|D| }{ \left(1-\exp\{-2\}\right) M } }.
			\vspace{-5pt}
		\end{equation*}
		\item[$\mathrm{(b)}$] 
		For the general convex case, i.e., $\alpha=0$, it holds that for any $\delta >0$
		\vspace{-1pt}
		\begin{equation}\label{eq: large-deviation of Alg for convex ellpetic MP 0}
			\begin{aligned}
				&
				\mathrm{Prob} \left\lbrace
				f(u_{K}) + g(z_{K}) - f(u^{\ast}) - g(z^{\ast}) 
				<
				\frac{C_{5}}{K} + \frac{\delta C_{4} \sqrt{\ln K}}{K} \right\rbrace 
				\geq
				1-\left( \frac{1}{K} \right)^{\frac{\delta^{2}}{3}}, \\
				&
				\mathrm{Prob} \left\lbrace
				\left\| u_{K}-z_{K}\right\|_{L^{2}(D)}
				<
				\frac{2 C_{5}}{c K} + \frac{2 \delta C_{4} \sqrt{\ln K}}{c K} \right\rbrace 
				\geq
				1-\left( \frac{1}{K} \right)^{\frac{\delta^{2}}{3}},
			\end{aligned}
			\vspace{-5pt}
		\end{equation}
		where  \vspace{-10pt}
		\begin{equation*}\label{eq: large-deviation of Alg for ellpetic MP 17}
			C_{5} =
			\frac{2 \phi_{1} Q^{2}}{M} 
			+ \frac{\eta+\rho}{2} \left\|u_{0}-u^{\ast}\right\|^{2} 
			+ \frac{ \left\|\lambda_{0}\right\|_{L^{2}(D)}^{2} + c^{2} }{ \mu \rho}.
			\vspace{-5pt}
		\end{equation*}
	\end{enumerate}
\end{theorem}
\begin{proof}
	For any $v,u \in U_{ad}$, it follows from (\ref{eq: bounded gradient bais}) that
	%	\begin{equation}\label{eq: large-deviation of Alg for ellpetic MP 1}
		%		\left\| \nabla f(v_{k}) - \nabla F(v_{k},\xi_{k,i})\right\|_{L^{2}(D)}  
		%		\leq
		%		2 Q,
		%	\end{equation}
	%	and
	\vspace{-2pt}
	\begin{equation}\label{eq: large-deviation of Alg for ellpetic MP 2}
		\begin{aligned}
			\left| \left\langle \nabla f(v) - \nabla F(v,\xi) , u-u^{\ast}  \right\rangle_{L^{2}(D)}  \right| 
			&
			=
			\left| \int_{D} \left( \nabla f(v) - \nabla F(v,\xi) \right) \left( u-u^{\ast} \right)  ~\mathrm{d}x \right| \\
			&
			\leq 	
			2 \sqrt{|D|} (u_{\max} - u_{\min})  Q.
		\end{aligned}
		\vspace{-3pt}
	\end{equation}
	Denoting \vspace{-5pt}
	\begin{equation*}\label{eq: large-deviation of Alg for ellpetic MP 3}
		\Delta (v,u,\Xi_{k}) = \left\langle 
		\nabla f(v) - \sum_{i=1}^{m_{k}} \nabla F(v,\xi_{k,i}) , u-u^{\ast}  \right\rangle_{L^{2}(D)},
		\vspace{-5pt}
	\end{equation*}
	then by Hoeffding lemma and the independence of the realizations $\Xi_{k}$, we obtain from (\ref{eq: unbiased gradient}) and (\ref{eq: large-deviation of Alg for ellpetic MP 2}) that for any $v,u \in U_{ad}$ and $s>0$
	\vspace{-2pt}
	\begin{equation*}\label{eq: large-deviation of Alg for ellpetic MP 4}
		\begin{aligned}
			\mathbb{E}\left[ \exp\{ s\Delta (v,u,\Xi_{k}) \}  \right] 
			&
			= 
			\mathbb{E}\left[ 
			\exp\left\lbrace 
			\frac{s}{m_{k}} \sum_{i=1}^{m_{k}} \left\langle \nabla f(v) - \nabla F(v,\xi_{k,i}) , u-u^{\ast}  \right\rangle_{L^{2}(D)} \right\rbrace 
			\right] 
			\\
			&
			=
			\prod_{i=1}^{m_{k}} \mathbb{E}\left[  
			\exp\left\lbrace 
			\frac{s}{m_{k}} \left\langle \nabla f(v) - \nabla F(v,\xi_{k,i}) , u-u^{\ast}  \right\rangle_{L^{2}(D)} \right\rbrace 
			\right] \\
			&
			\leq
			\exp\left\lbrace \frac{ s^{2} |D| (u_{\max} - u_{\min})^{2} Q^{2} }{ 2 m_{k} } \right\rbrace, 
		\end{aligned}
		\vspace{-3pt}
	\end{equation*}
	which implies that the random variable $\Delta (v,u,\Xi_{k})$ is sub-Gaussian. From Theorem~2.6 of \cite{wainwright2019high}, we further get for any $v,u \in U_{ad}$
	\begin{equation}\label{eq: large-deviation of Alg for ellpetic MP 5}
		\mathbb{E}\left[ 
		\exp \left\lbrace  \frac{ \left(1-\exp\{-2\}\right) m_{k} \Delta (v,u,\Xi_{k})^{2} }{ 2 |D| (u_{\max} - u_{\min})^{2} Q^{2} } \right\rbrace   
		\right] 
		\leq\exp\{1\}.
	\end{equation}
	Since $v_{k},  u_{k} \in U_{ad}$ are $\mathcal{F}_{k}$-measurable and $\Xi_{k}$ is independent of $\Xi_{0},\cdots\Xi_{k-1}$, we obtain from Theorem 8.5 of \cite{kallenberg2021foundations} and (\ref{eq: large-deviation of Alg for ellpetic MP 5}) that
	\begin{equation}
		\begin{aligned}
			&\mathbb{E}\left[ \left. 
			\exp \left\lbrace  \frac{ \left(1-\exp\{-2\}\right) m_{k} \Delta (v_{k},u_{k},\Xi_{k})^{2} }{ 2 |D| (u_{\max} - u_{\min})^{2} Q^{2} } \right\rbrace   \right| \mathcal{F}_{k}
			\right]\\
			=
			& 
			\left. 
			\mathbb{E}\left[
			\exp \left\lbrace  \frac{ \left(1-\exp\{-2\}\right) m_{k} \Delta (v,u,\Xi_{k})^{2} }{ 2 |D| (u_{\max} - u_{\min})^{2} Q^{2} } \right\rbrace 
			\right]	\right|_{v=v_{k}, u=u_{k}} \\
			\leq
			&
			\exp\{1\}.
		\end{aligned}
	\end{equation}
	Also note that from (\ref{eq: strong convergence for nonergodic setting 14}), $\Delta (v_{k},u_{k},\Xi_{k})$ is the martingale-difference. Therefore, from the large-deviation theorem on the martingale-difference (see, e.g. Lemma~2 of \cite{lan2012validation}), we have for any $\epsilon >0$ that
    \begin{equation*}\label{eq: large-deviation of Alg for ellpetic MP 7}
		\text{Prob}\left\lbrace 
		\sum_{k=0}^{K-1} \theta_{k}\Delta (v_{k},u_{k},\Xi_{k}) > 
		\epsilon (u_{\max} - u_{\min}) Q \sqrt{ \frac{ 2 |D| }{1-\exp\{-2\}} \sum_{k=0}^{K-1}  \frac{\theta_{k}^{2}}{m_{k}} }
		\right\rbrace  
		\leq
		\exp\{-\frac{\epsilon^{2}}{3}\}.
	\end{equation*}
	Noting that $\theta_{k+1} > \theta_{k}$ for all $k$, we obtain
  \begin{equation}\label{eq: large-deviation of Alg for ellpetic MP 8}
		\text{Prob}\left\lbrace 
		\frac{1}{\theta_{K-1}^{2}}\sum_{k=0}^{K-1} \theta_{k}\Delta (v_{k},u_{k},\Xi_{k}) 
		< 
		\frac{\epsilon (u_{\max} - u_{\min}) Q}{\theta_{K-1}} 
		\sqrt{ \frac{ 2 |D| }{ \left(1-\exp\{-2\}\right) M } }
		\right\rbrace  
		\geq
		1-\exp\{-\frac{\epsilon^{2}}{3}\}.
	\end{equation}
	By using (\ref{eq: bounded gradient bais}) and Cauchy-Schwarz inequality in (\ref{eq: convergence rate for nonergodic setting 0}), and by the similar argument for (\ref{eq: convergence rate for nonergodic setting 2}) and (\ref{eq: convergence rate for nonergodic setting 3}),  we additionally have for any $c > 2\|\lambda^{*}\|$
  \vspace{-2pt}
  \begin{equation*}\label{eq: large-deviation of Alg for ellpetic MP 9}
		\begin{aligned}
			\theta_{K-1}^{2} \left( f(u_{K}) + g(z_{K}) - f(u^{\ast}) - g(z^{\ast}) 
			\right)  
			\leq
			&
			\frac{2 \phi_{\theta_{2\kappa}} Q^{2}}{M } 
			+ \frac{ \sqrt{|D|} (u_{\max} - u_{\min}) (\alpha + \frac{C_{P}^{4}}{a_{\min}}) \theta_{2\kappa-1}^{2}}{2} \\
			&
			+ \frac{\left\|\lambda_{0}\right\|_{L^{2}(D)}^{2} + c^{2}}{ \mu \rho} 
			+ \sum_{k=0}^{K-1} \theta_{k}\Delta (v_{k},u_{k},\Xi_{k}) ,
		\end{aligned}
		\vspace{-8pt}
	\end{equation*}
	and \vspace{-7pt}
	\begin{equation*}\label{eq: large-deviation of Alg for ellpetic MP 10}
		\begin{aligned}
			\frac{c \theta_{K-1}^{2}}{2} \left\| u_{K}-z_{K}\right\|_{L^{2}(D)} 
			\leq
			&
			\frac{2 \phi_{\theta_{2\kappa}} Q^{2}}{M } 
			+ \frac{ \sqrt{|D|} (u_{\max} - u_{\min}) (\alpha + \frac{C_{P}^{4}}{a_{\min}}) \theta_{2\kappa-1}^{2}}{2} \\
			&
			+ \frac{\left\|\lambda_{0}\right\|_{L^{2}(D)}^{2} + c^{2}}{ \mu \rho} 
			+ \sum_{k=0}^{K-1} \theta_{k}\Delta (v_{k},u_{k},\Xi_{k}) ,
		\end{aligned}
		\vspace{-3pt}
	\end{equation*}
	Consequently, combing the above inequalities with (\ref{eq: large-deviation of Alg for ellpetic MP 8}), respectively, and using the fact $\theta_{k-1}\geq k/2$, we conclude 
	\vspace{-2pt}
	\begin{equation}\label{eq: large-deviation of Alg for ellpetic MP 11}
		\text{Prob}\left\lbrace
		f(u_{K}) + g(z_{K}) - f(u^{\ast}) - g(z^{\ast}) 
		<
		\frac{C_{3}}{K^{2}} + \frac{\epsilon C_{4}}{K} \right\rbrace 
		\geq
		1-\exp\{-\frac{\epsilon^{2}}{3}\},
		\vspace{-3pt}
	\end{equation}
	and \vspace{-2pt}
	\begin{equation}\label{eq: large-deviation of Alg for ellpetic MP 12}
		\text{Prob}\left\lbrace
		\left\| u_{K}-z_{K}\right\|_{L^{2}(D)}
		<
		\frac{2 C_{3}}{c K^{2}} + \frac{2 \epsilon C_{4}}{c K} \right\rbrace 
		\geq
		1-\exp\{-\frac{\epsilon^{2}}{3}\}.
		\vspace{-3pt}
	\end{equation}
	Now we let $\epsilon = \delta\sqrt{\ln K}$ to obtain the desired results. 
	
	The probabilities of large deviations for the general convex case follow from (\ref{eq: nonergodic convergence rate in convex case 5}) and the similar arguments for the strongly convex case, and hence the details are skipped.
	%	For general convex case, by essentially the same arguments as above,  we can conclude
	%	\begin{equation}\label{eq: large-deviation of Alg for ellpetic MP 15}
		%		\text{Prob}\left\lbrace
		%		f(u_{K}) + g(z_{K}) - f(u^{\ast}) - g(z^{\ast}) 
		%		<
		%		\frac{C_{5}}{K} + \frac{\epsilon C_{4}}{K} \right\rbrace 
		%		\geq
		%		1-\exp\{-\frac{\epsilon^{2}}{3}\},
		%	\end{equation}
	%	and
	%	\begin{equation}\label{eq: large-deviation of Alg for ellpetic MP 16}
		%		\text{Prob}\left\lbrace
		%		\left\| u_{K}-z_{K}\right\|
		%		<
		%		\frac{2 C_{5}}{c K} + \frac{2 \epsilon C_{4}}{c K} \right\rbrace 
		%		\geq
		%		1-\exp\{-\frac{\epsilon^{2}}{3}\}.
		%	\end{equation}	
	%
	%\par\hspace{1ex}
\end{proof}

\subsection{Numerical experiments}
\label{sec: numerical experiments}
In this section, we report the test results of the algorithmic framework presented in Algorithm~\ref{alg:Stochastic Linearized ADMM}, applied to the constrained optimal control problems governed by the random elliptic PDE considered in subsection~\ref{subsec: random elliptic PDE-constrained optimization}, and numerically verify its efficiency on specific model problems.

\subsubsection{Parameter settings}
\label{subsubsec: parameter settings}
%{\textbf{Parameter settings}.}

In the strongly convex case, given the strong convexity modulus $\alpha$, we chose the parameters $\rho$ and $\eta$ in Algorithm~\ref{alg:Stochastic Linearized ADMM} such that
$\rho + \eta = \alpha$ and $\eta(1-\mu) = 2\rho\mu$ for a prescribed value of $\mu$. 

In the general convex case, the parameter $\rho$ was set to $\beta$, where $\beta$ is the specified sparsity-promoting parameter. Note that the choice of $\eta$ depends on the Lipschitz constant $L$. Since it is difficult to calculate the exact value of $L$, the value of $L$ was estimated by averaging the norms of $\nabla F$ over $10^{3}$ calls to the SFO. For a prescribed value of $\mu$, we set $\eta = \min\{ \mu\rho / (1-\mu) + 1.01L, \rho \}$.

We compared our method with the SPG method \cite{jofre2019variance,rosasco2020convergence}, the SSG method \cite{jain2021making,nemirovski2009robust}, and extensions of the adaptive SG method \cite{cao2022adaptive} that incorporate subgradient and proximal operators. For solving the constrained optimal control problems governed by random PDEs, the main difference among these methods lies in the choice of step size; therefore, in our experiments we adhered to the stepsize policy for each method in the literature.

Throughout our tests, all compared methods were initialized from the same starting point and employed the same stochastic gradient $G_{k}$ computed by (\ref{eq: average SG}). In addition, each realization of $\xi$ was generated uniformly at random, and the batch size $m_{k}$ was set to increase asymptotically according to $m_{k} = \lceil 0.5 \times k^{1.1} \rceil$. For the  simulations, the  piecewise  linear finite elements were used on the spatial domain $D$ with mesh size $2^{-5}$. Our codes were written in MATLAB R2023b and experiments were implemented on a
desktop computer with a 13th Gen Intel Core i7-13700K 3.40 GHz processor.

\subsubsection{Experiment results}
\label{subsubsec: experiment results}
Let $D= [0,1]^{2}$,  and we considered an example of problem (\ref{eq: elliptic optimal control problem}) in which the desired state is given by 
\begin{equation*}
	y_{d} = 
	\left\{
	\begin{aligned}
		-1,& \enspace x\in (0.25,0.75)^{2}, \\
		1,& \enspace \text{else}.
	\end{aligned}
	\right.
\end{equation*}
As in \cite{milz2023reliable}, the random diffusion coefficient is set to
\vspace{-2pt}
\begin{equation*}
	\begin{aligned}
		a(x,\omega) = 
		&\exp\left\lbrace  
		\xi^{(1)}(\omega)\cos(1.1\pi x_{1}) + \xi^{(2)}(\omega)\cos(1.2\pi x_{1}) + \xi^{(3)}(\omega)\sin(1.3\pi x_{2})
		\right.
		\\
		&
		\hspace{6ex}\left.
		+ \xi^{(4)}(\omega)\sin(1.4\pi x_{2}) 
		\right\rbrace , 
	\end{aligned}
	\vspace{-3pt}
\end{equation*}
where the random variables $\xi^{(1)}, \xi^{(2)}, \xi^{(3)}$ and $\xi^{(4)}$ are mutually independent and uniformly distributed in $[-1,1]$. The control constraints are $u_{\min}=-6$ and $u_{\max}=6$.

{\it \textbf{Strongly convex case.}} 
For $\alpha>0$, we used the above parameter rule with $\mu = 0.5$. To compare the efficiency of our method with the aforementioned SG-type methods, we generated $10^{4}$ samples from the same distribution and computed the empirical objective value at each iteration. The entire procedure was repeated for $50$ independent runs, and the average empirical objective values are shown in Figure~\ref{fig: Alg comparision for elliptic strongly convex case} for four $(\alpha,\beta)$ pairs. From the figure, we see that when $\alpha$ is large, the SSG method outperforms the SPG method and the extensions of the adaptive SG method, whereas when $\alpha$ is small, the SPG method in \cite{jofre2019variance,rosasco2020convergence} is preferable to the other SG-type methods. In addition, the proposed stochastic ADMM  performs best overall and achieves a much lower objective value within the same running time, especially when a small $(\alpha,\beta)$ pair is used.
\begin{figure}[tbhp]
\vspace{-15pt}
	\centering
	\subfloat[$\alpha=10^{-5},\beta=10^{-5}$]{
		\includegraphics[width=0.228\textwidth, height=0.205\textwidth]{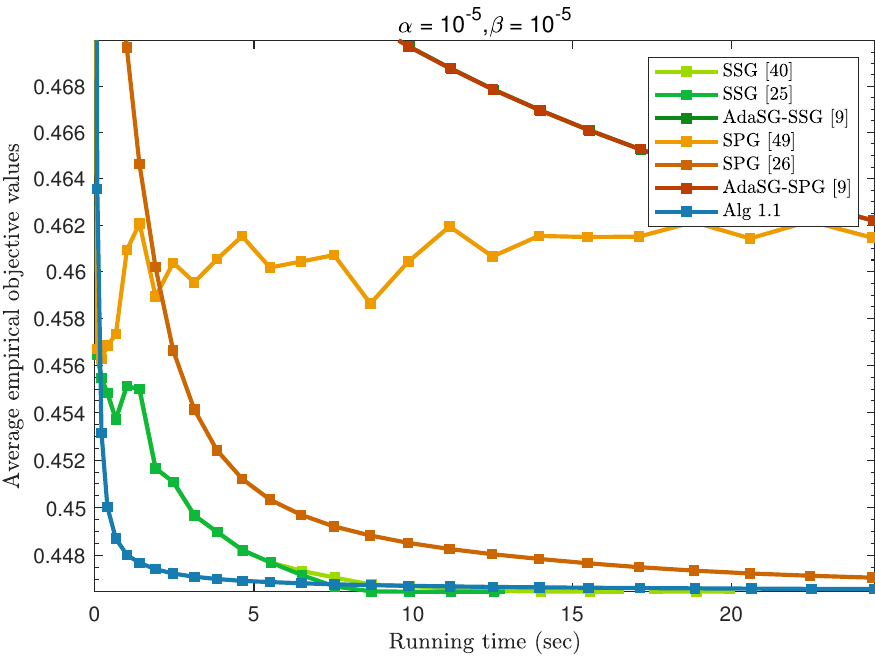}}
	\hspace{0.41em}
	\subfloat[$\alpha=10^{-5},\beta=10^{-6}$]{
		\includegraphics[width=0.228\textwidth, height=0.205\textwidth]{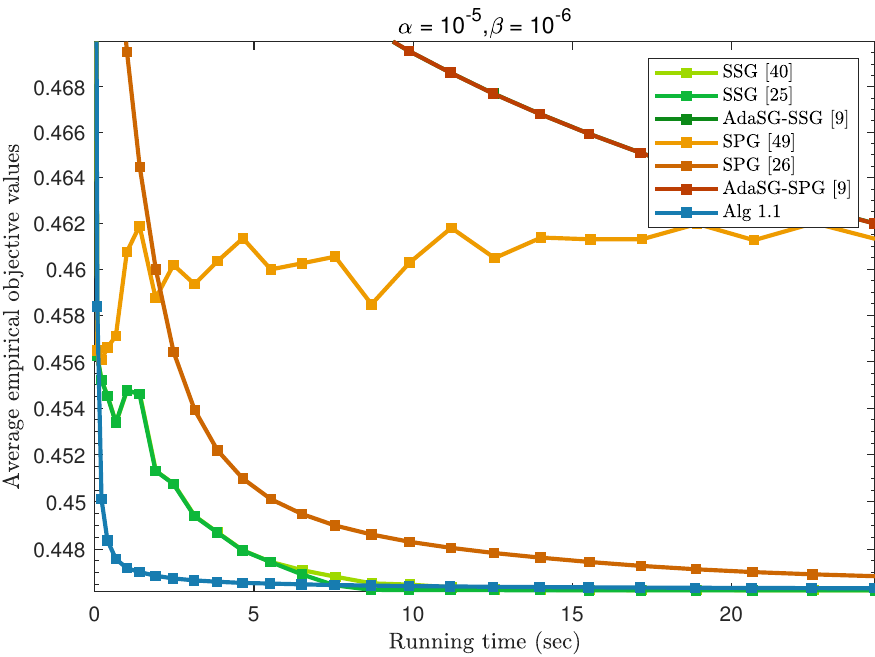}}
	\hspace{0.41em}
	\subfloat[$\alpha=10^{-6},\beta=10^{-5}$]{
		\includegraphics[width=0.228\textwidth, height=0.205\textwidth]{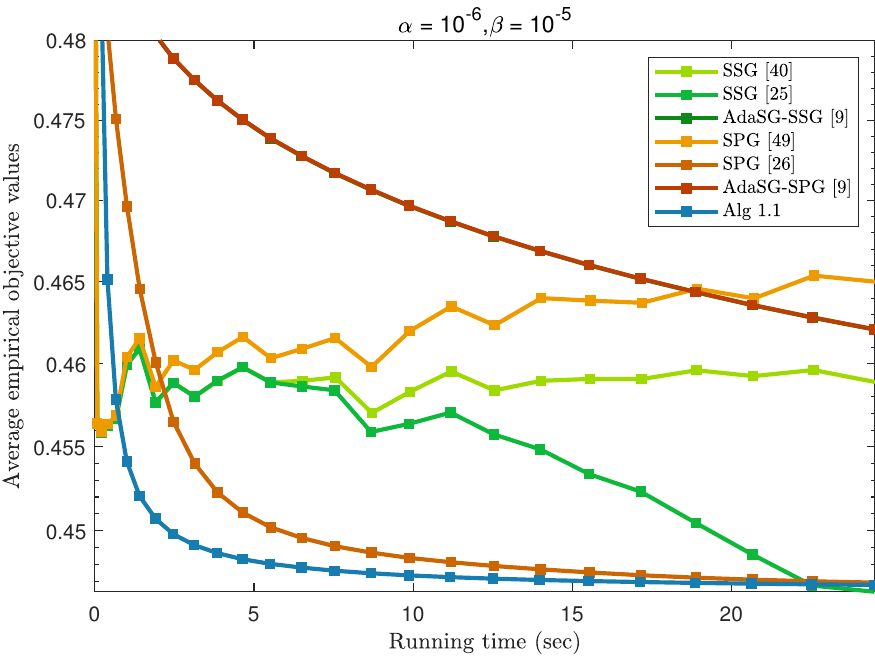}}
	\hspace{0.41em}
	\subfloat[$\alpha=10^{-6},\beta=10^{-6}$]{
		\includegraphics[width=0.228\textwidth, height=0.205\textwidth]{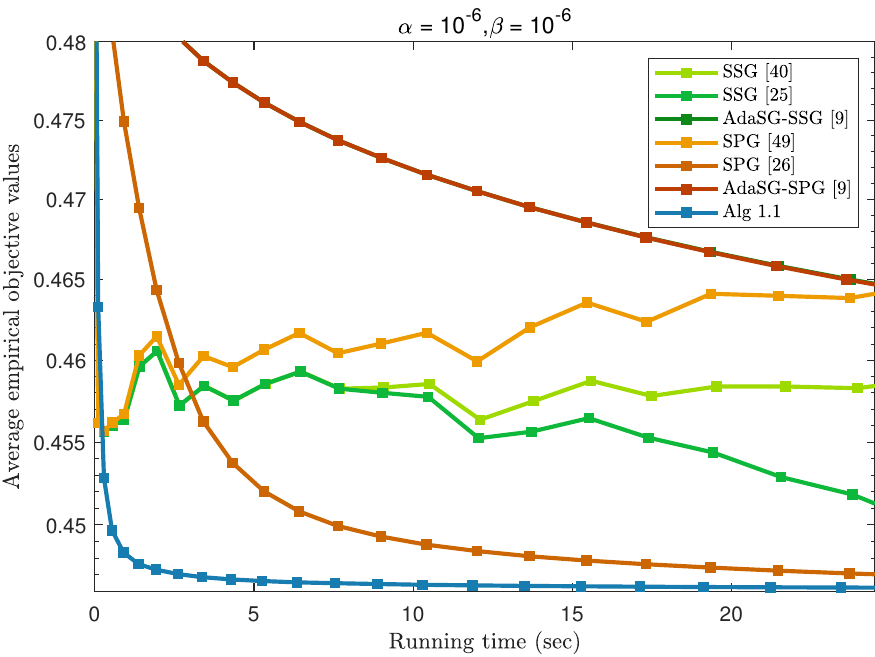}}
	\vspace{-7pt}
	\caption{Objective values of Algorithm~\ref{alg:Stochastic Linearized ADMM} and SG-type methods for different $(\alpha,\beta)$ pairs in strongly convex case.}
	\label{fig: Alg comparision for elliptic strongly convex case}
		\vspace{-8pt}
\end{figure}
\begin{figure}[tbhp]
	\vspace{-20pt}
	\centering
	\subfloat[$\alpha=10^{-5},\beta=10^{-5}$]{
		\includegraphics[width=0.242\textwidth, height=0.19\textwidth]{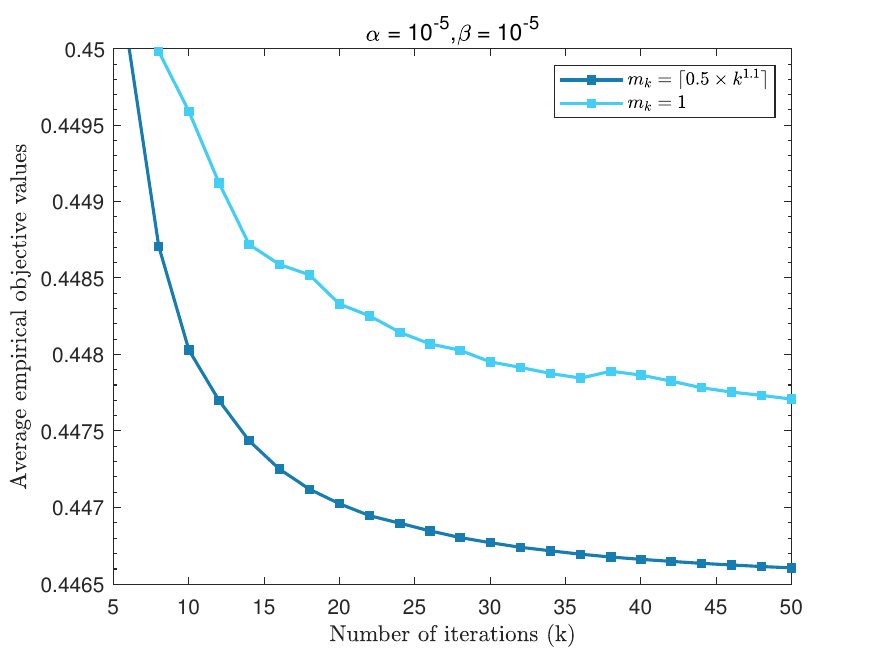}}
	\subfloat[$\alpha=10^{-5},\beta=10^{-6}$]{
		\includegraphics[width=0.242\textwidth, height=0.19\textwidth]{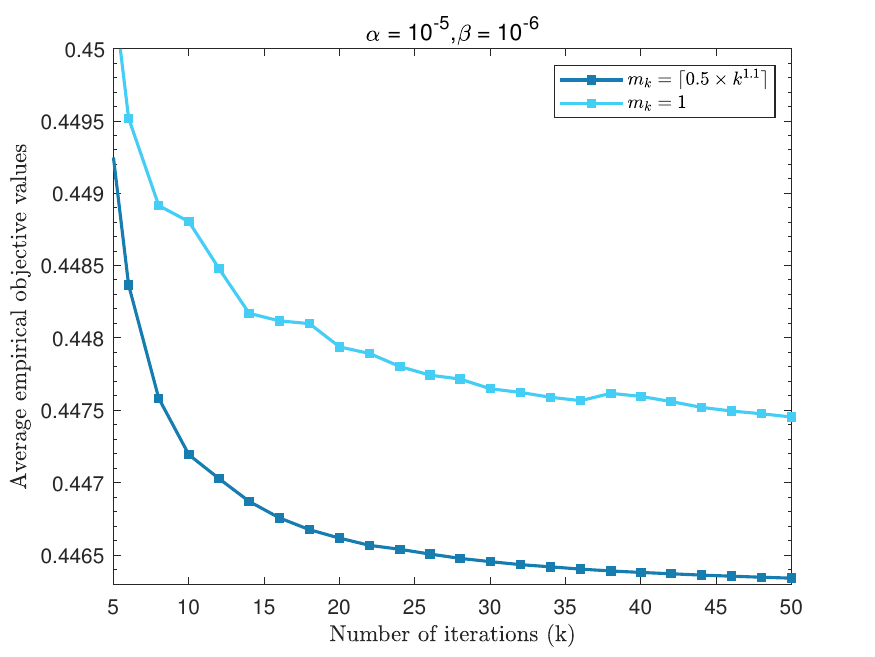}}
	\subfloat[$\alpha=10^{-6},\beta=10^{-5}$]{
		\includegraphics[width=0.242\textwidth, height=0.19\textwidth]{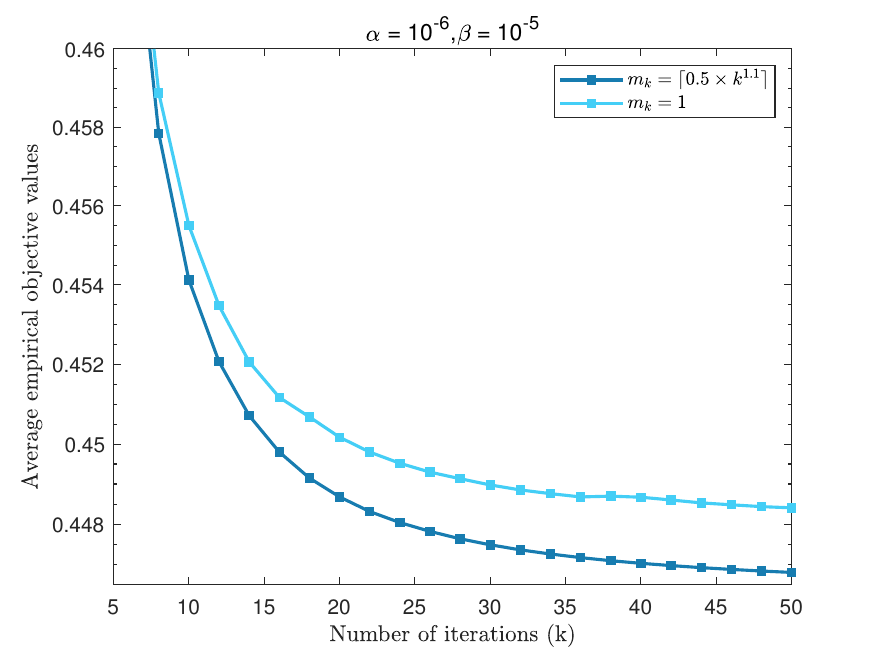}}
	\subfloat[$\alpha=10^{-6},\beta=10^{-6}$]{
		\includegraphics[width=0.242\textwidth, height=0.19\textwidth]{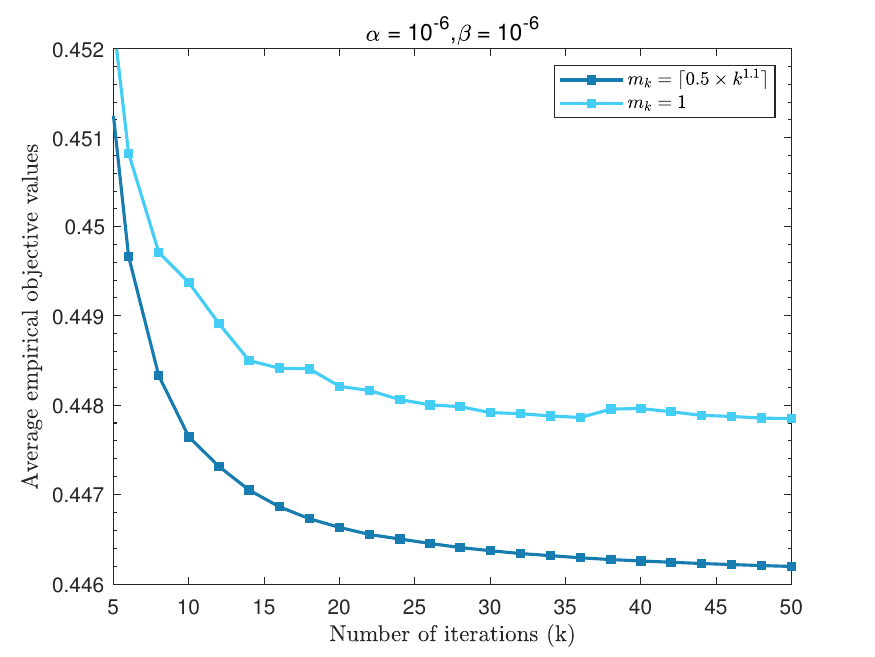}}
		\vspace{-7pt}
	\caption{Objective values of Algorithm~\ref{alg:Stochastic Linearized ADMM}  for different $(\alpha,\beta)$ pairs in the strongly convex case with different sampled stochastic gradients $G_{k}$.}
	\label{fig: Ellipse strongly convex VR-E}
\end{figure}

\vspace{-5pt}
We observe that the averaged $G_{k}$ does improve the efficiency of the proposed method when increasing batch size $m_{k}$. We further tested Algorithm~\ref{alg:Stochastic Linearized ADMM} with $m_{k}\equiv 1$ and compared its solution with that obtained by adopting $m_{k} = \lceil 0.5 \times k^{1.1} \rceil$ for a fixed number of $50$ iterations. Figure~\ref{fig: Ellipse strongly convex VR-E} shows the comparison results of the average empirical objective values. Fixing $\alpha = 10^{-4}$, we also report in Table~\ref{tab: Elliptical strongly convex sparsity}, for various values of $\beta$, the percentages of $D$ on which the computed control $u_{K} \neq 0$ under the two $m_{k}$ settings. Moreover, the numerical results for the computed control $u$ and the corresponding average empirical state $y$ are plotted in Figure~\ref{fig: Ellipse strongly convex sparse u} and \ref{fig: Ellipse strongly convex sparse y} for several values of $\beta$, respectively.

\vspace{-5pt}
\begin{table}[tbhp]
\footnotesize
\caption{Percentage of the domain $D$ where $u \neq 0$ for different values of $\beta$ in strongly convex case.}\label{tab: Elliptical strongly convex sparsity}
	\vspace{-10pt}
\begin{center}
  \begin{tabular}{|c|c|c|c|c|c|c|} \hline
   $\beta$& 0 & $5\times 10^{-3}$& $8\times 10^{-3}$& $1\times 10^{-2}$& $2\times 10^{-2}$& $3\times 10^{-2}$ \\ \hline
    $m_{k} = \lceil 0.5 \times k^{1.1} \rceil$ & $100\%$  &$93.44\%$ &$78.56\%$ &$70.97\%$ &$21.64\%$ &$0.02\%$ \\
    $m_{k} = 1$ &$100\%$ &$98.75\%$ &$91.88\%$ &$85.95\%$ &$49.74\%$ &$0.73\%$ \\  \hline
  \end{tabular}
\end{center}
\end{table}
\begin{figure}[tbhp]
		\vspace{-20pt}
	\centering
	\subfloat{
		\includegraphics[width=0.242\textwidth]{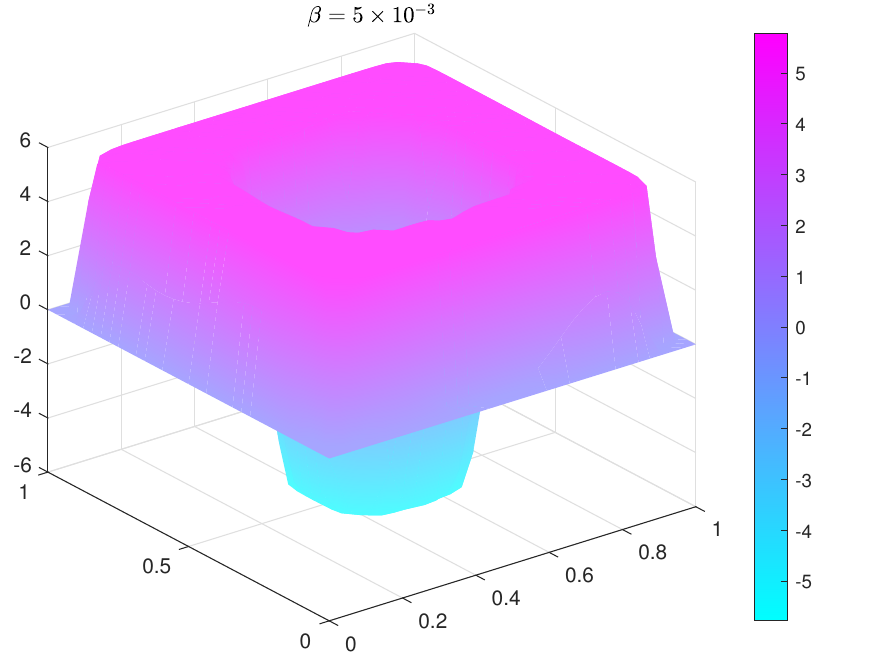}}
	\subfloat{
		\includegraphics[width=0.242\textwidth]{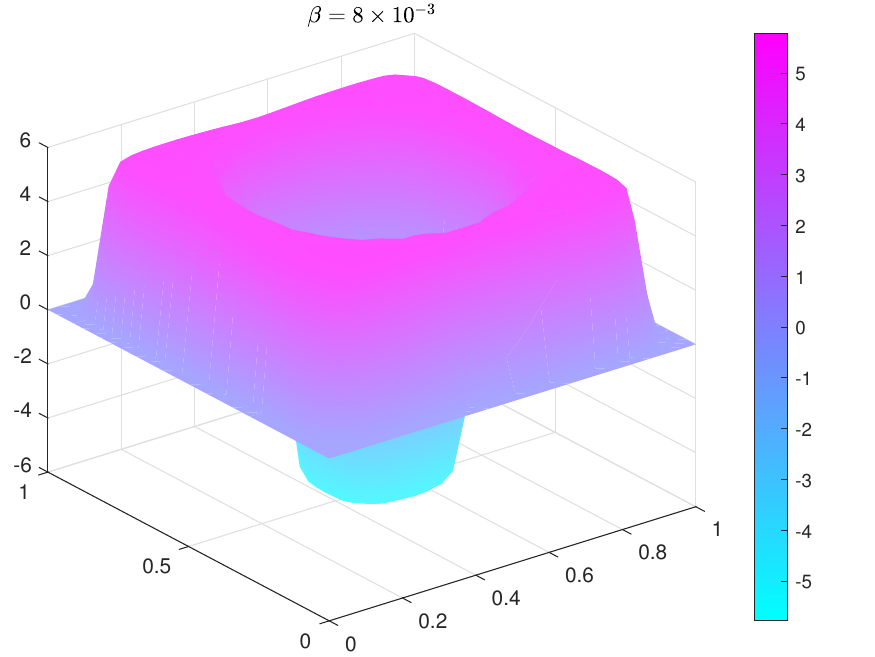}}
	\subfloat{
		\includegraphics[width=0.242\textwidth]{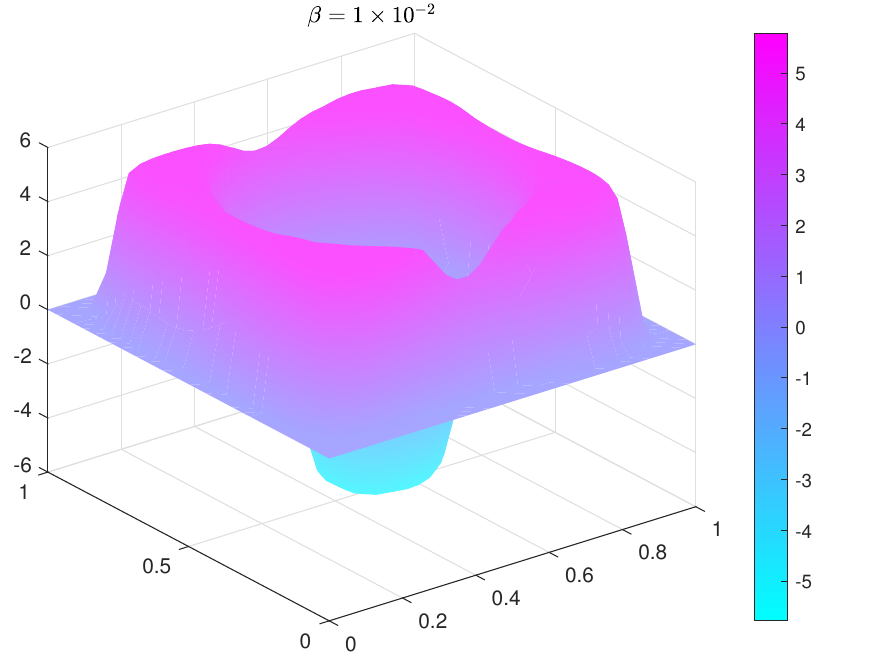}}
	\subfloat{
		\includegraphics[width=0.242\textwidth]{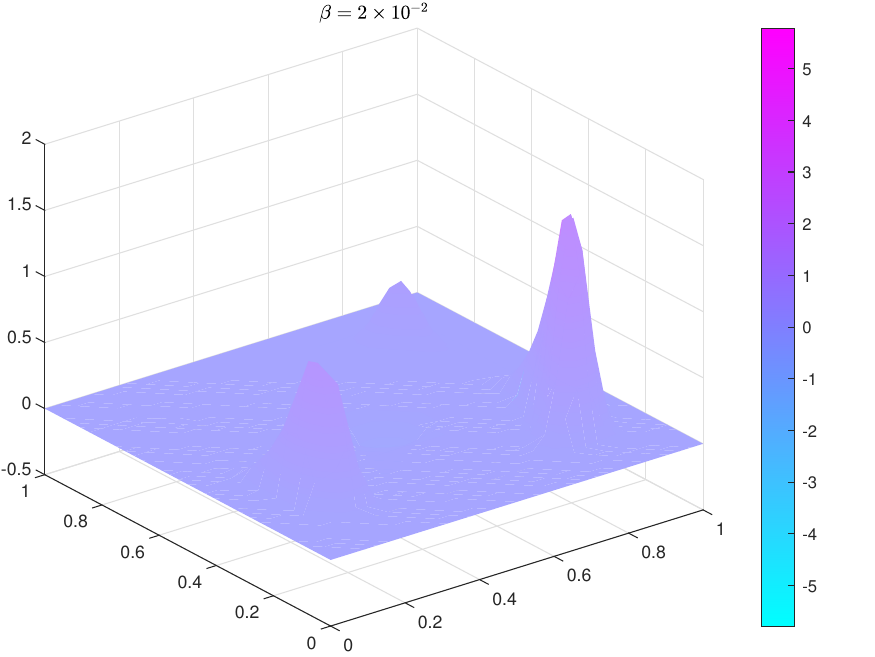}}
	\\
	\subfloat{
		\includegraphics[width=0.242\textwidth]{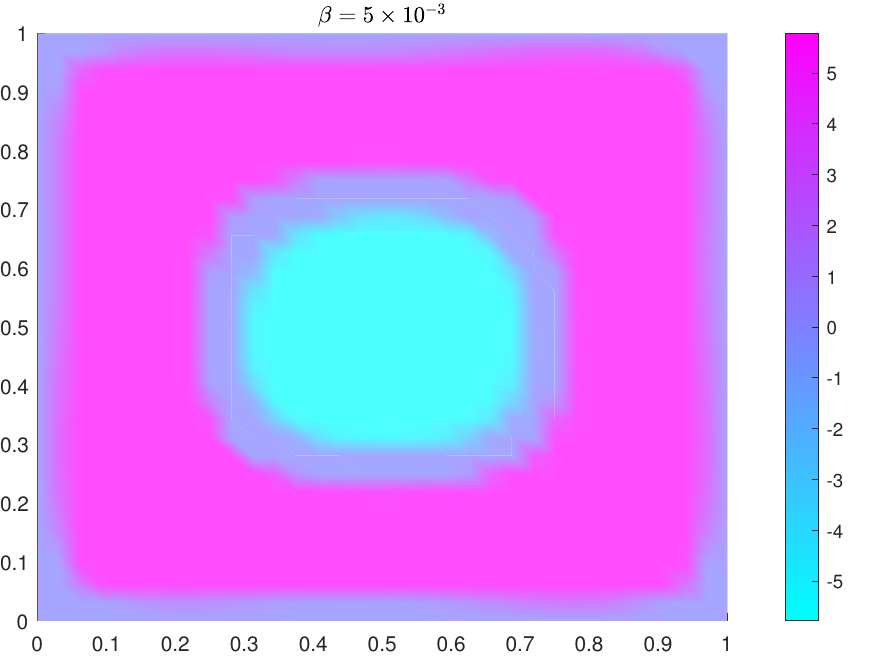}}
	\subfloat{
		\includegraphics[width=0.242\textwidth]{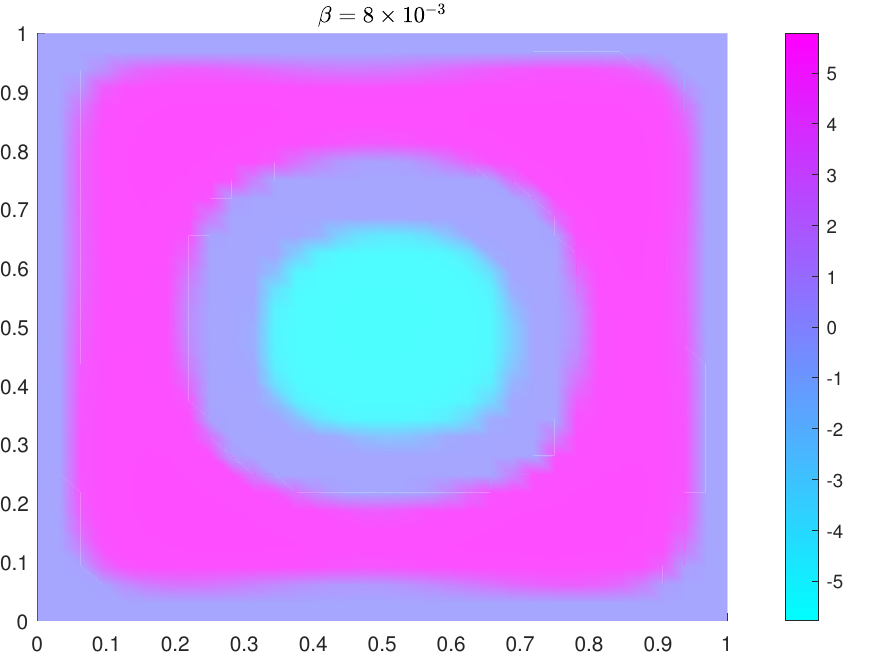}}
	\subfloat{
		\includegraphics[width=0.242\textwidth]{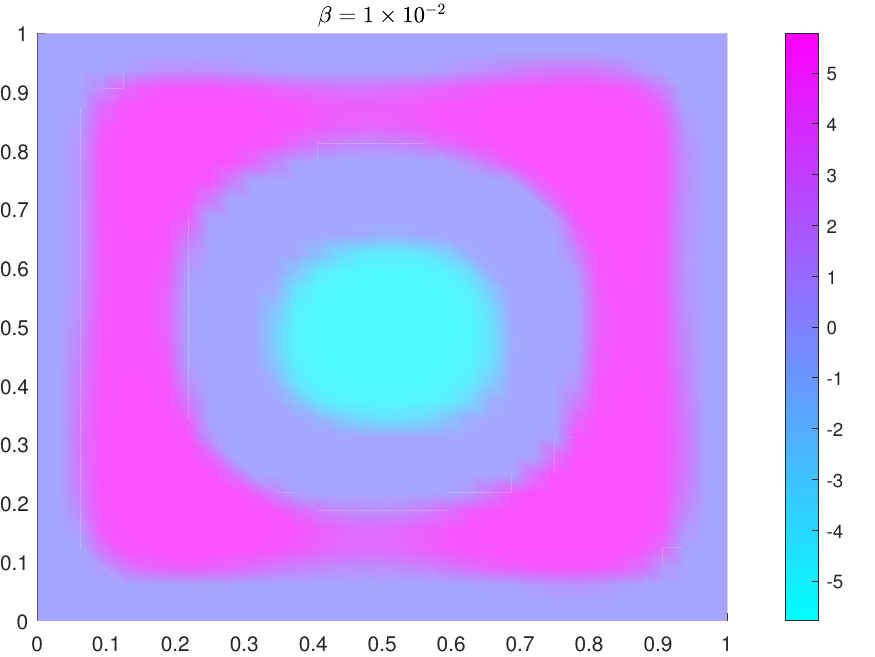}}
	\subfloat{
		\includegraphics[width=0.242\textwidth]{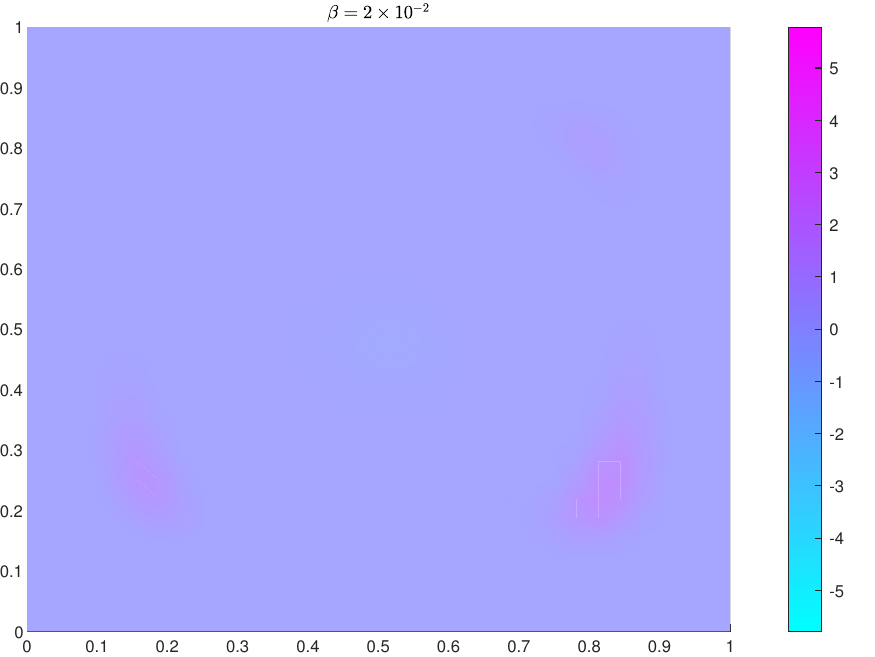}}
		\vspace{-10pt}
	\caption{Control variables for several values of $\beta$ with fixed $\alpha = 10^{-4}$ in strongly convex case. Top row: principal view; bottom row: top view.}
	\label{fig: Ellipse strongly convex sparse u}
\end{figure}
\begin{figure}[tbhp]
		\vspace{-20pt}
	\centering
	\subfloat{
		\includegraphics[width=0.242\textwidth]{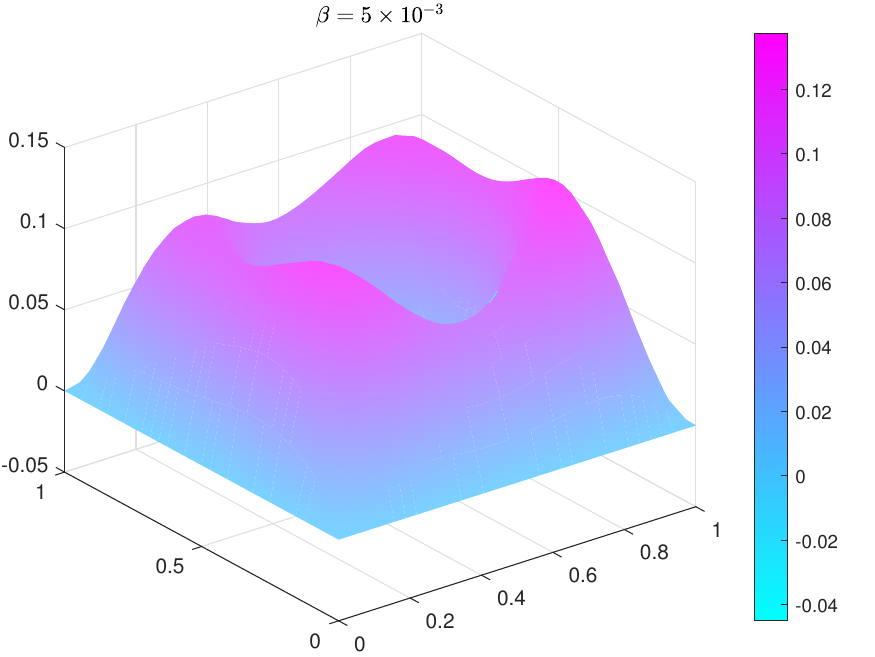}}
	\subfloat{
		\includegraphics[width=0.242\textwidth]{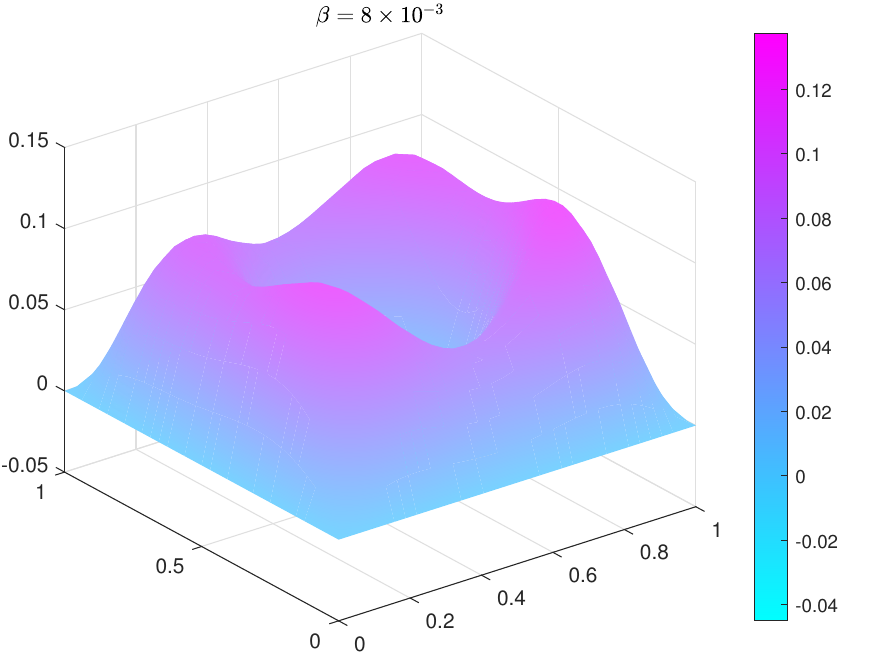}}
	\subfloat{
		\includegraphics[width=0.242\textwidth]{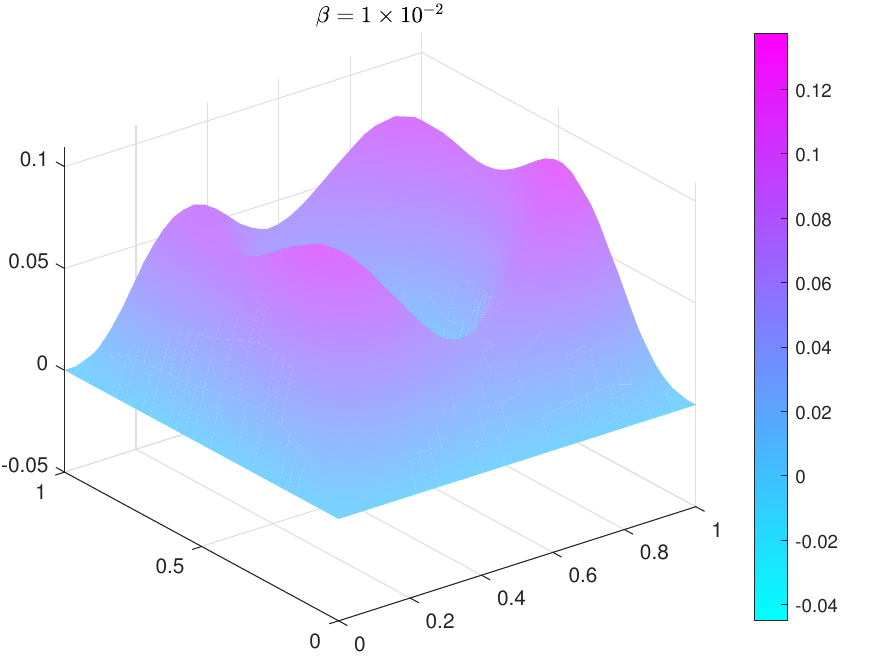}}
	\subfloat{
		\includegraphics[width=0.242\textwidth]{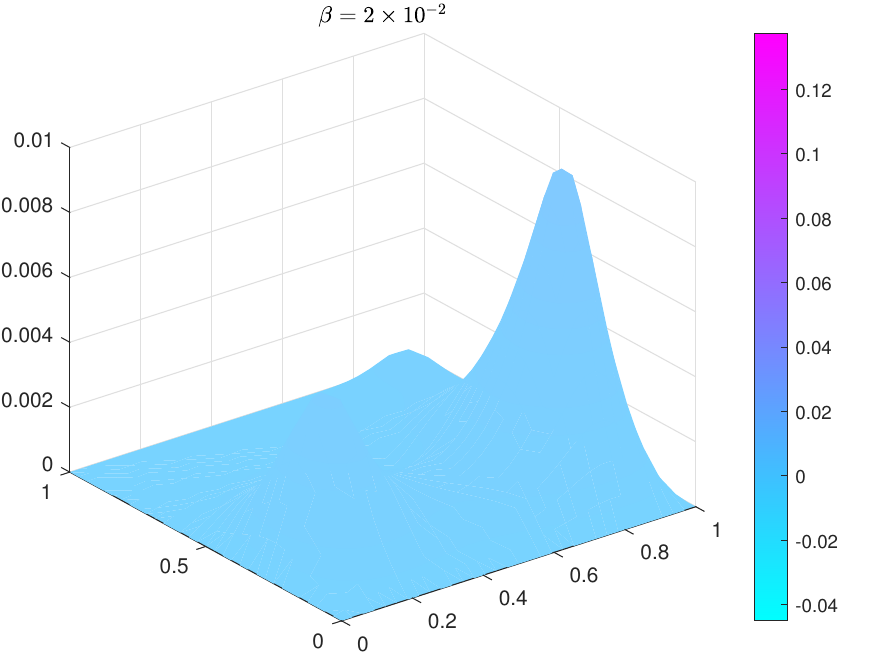}}
	\\
	\subfloat{
		\includegraphics[width=0.242\textwidth]{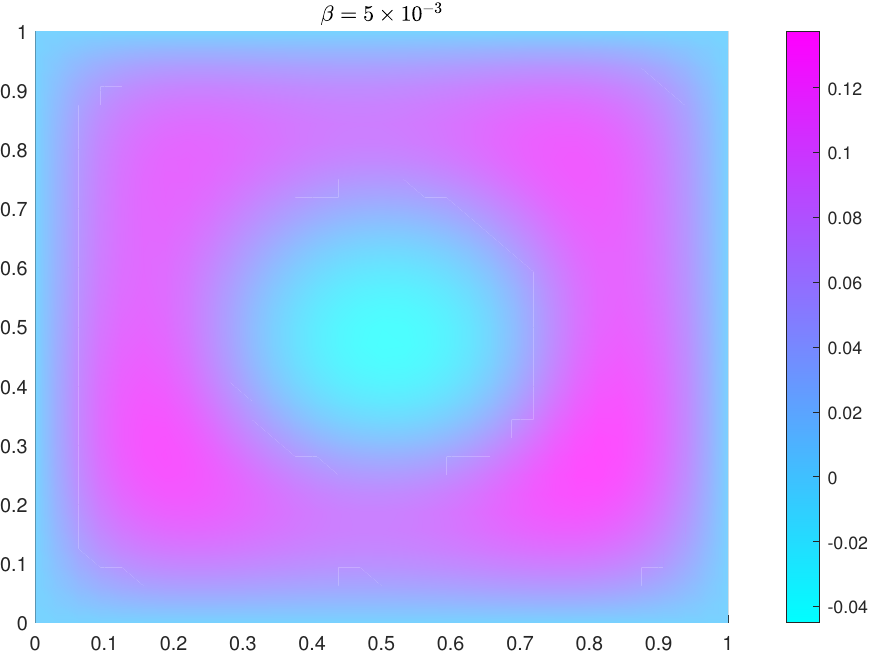}}
	\subfloat{
		\includegraphics[width=0.242\textwidth]{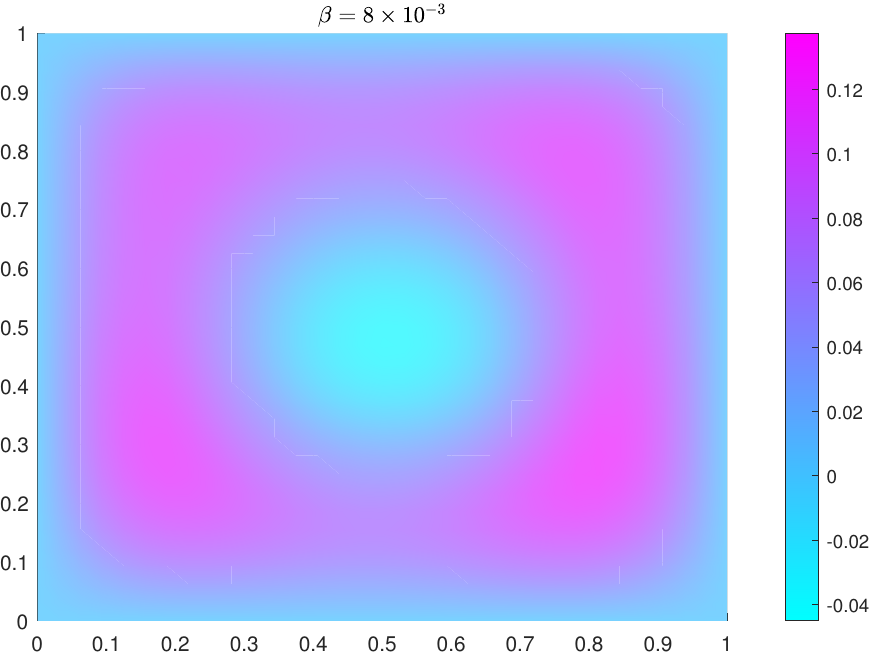}}
	\subfloat{
		\includegraphics[width=0.242\textwidth]{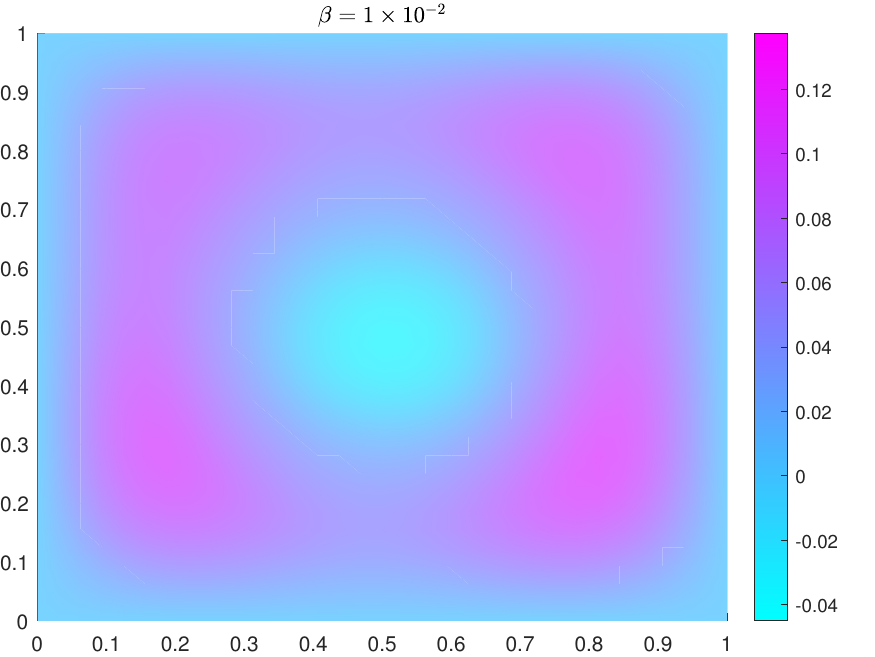}}
	\subfloat{
		\includegraphics[width=0.242\textwidth]{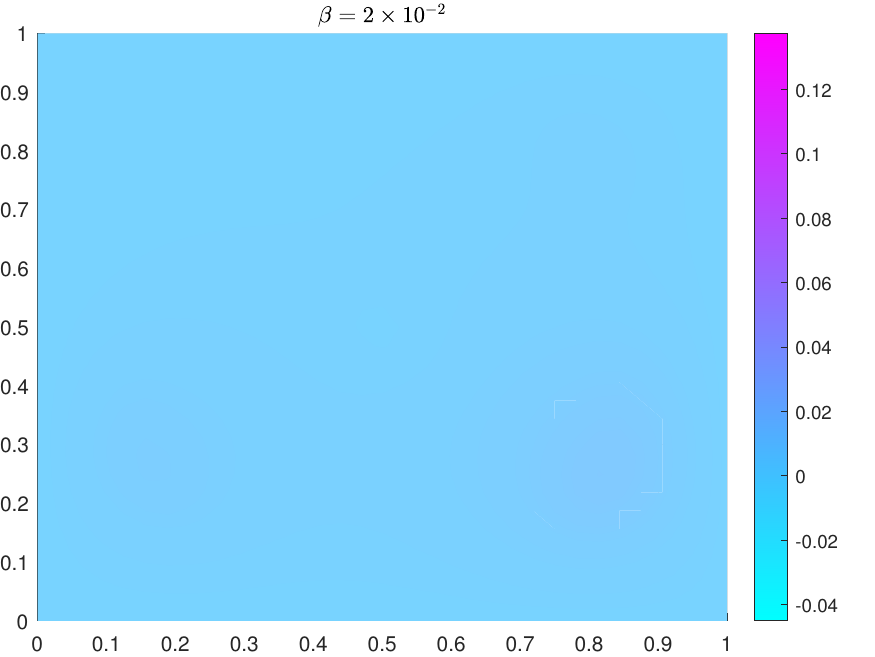}}
		\vspace{-10pt}
	\caption{State variables for several values of $\beta$ with fixed $\alpha = 10^{-4}$ in strongly convex case. Top row: principal view; bottom row: top view.}
	\label{fig: Ellipse strongly convex sparse y}
\end{figure}

\begin{figure}[H]
		\vspace{-20pt}
	\centering
	\subfloat[$\alpha=10^{-5},\beta=10^{-5}$]{
		\includegraphics[width=0.242\textwidth, height=0.19\textwidth]{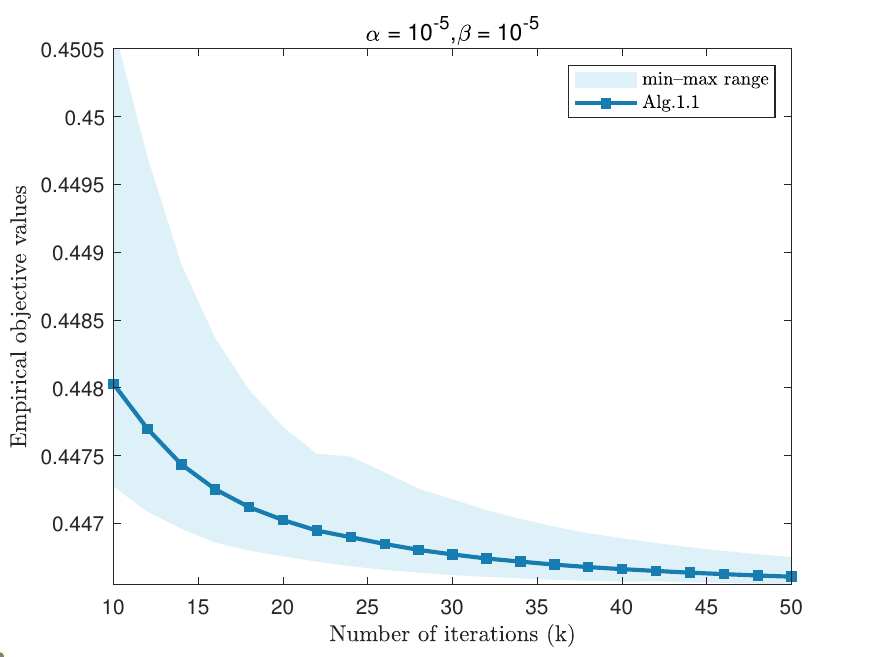}}
	\subfloat[$\alpha=10^{-5},\beta=10^{-6}$]{
		\includegraphics[width=0.242\textwidth, height=0.19\textwidth]{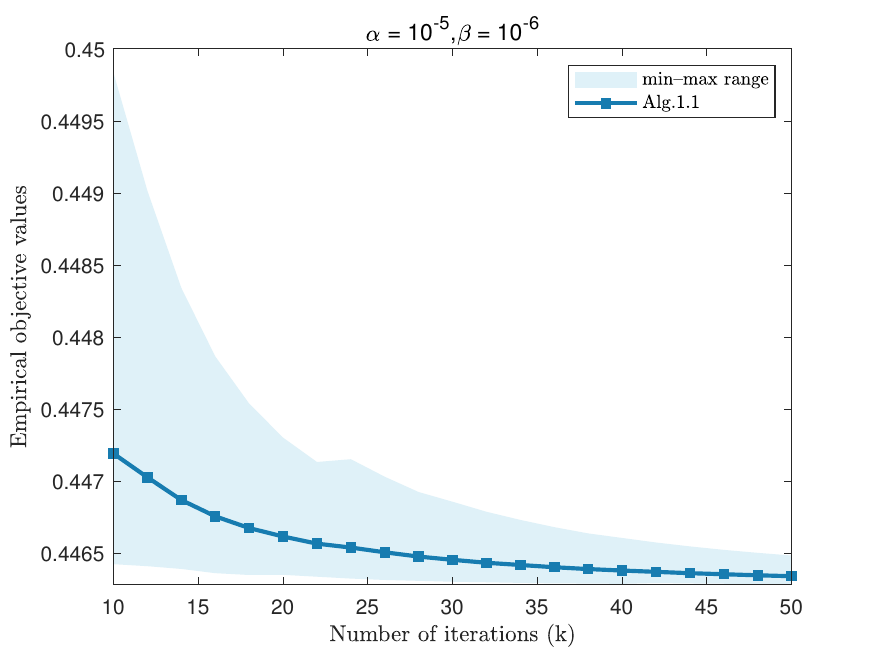}}
	\subfloat[$\alpha=10^{-6},\beta=10^{-5}$]{
		\includegraphics[width=0.242\textwidth, height=0.19\textwidth]{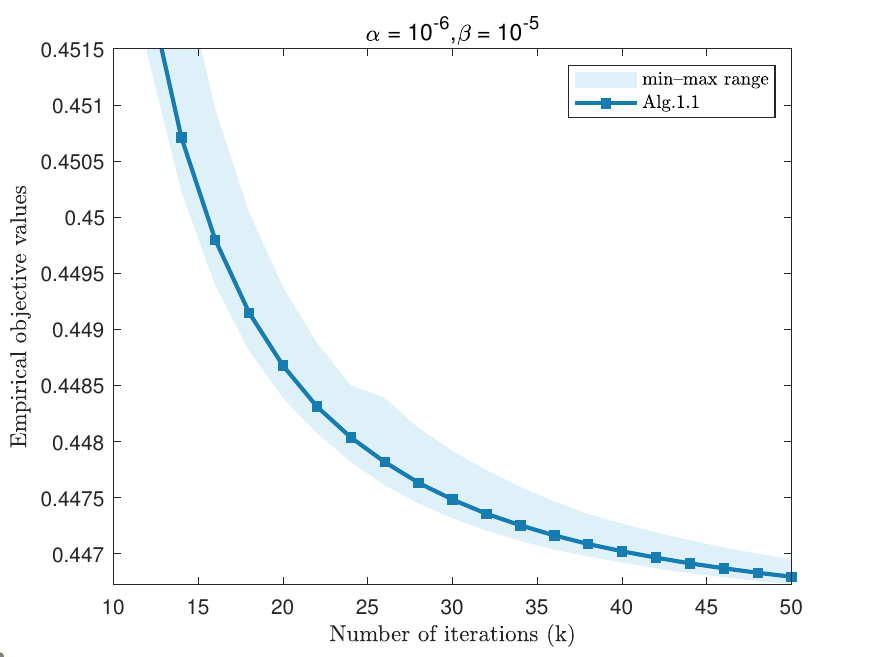}}
	\subfloat[$\alpha=10^{-6},\beta=10^{-6}$]{
		\includegraphics[width=0.242\textwidth, height=0.19\textwidth]{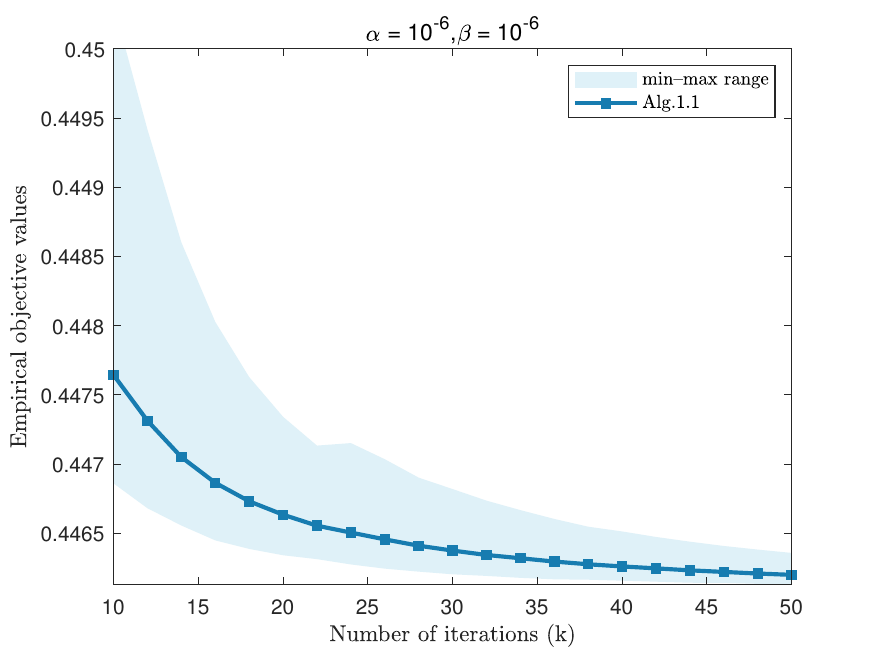}}
	\vspace{-7pt}
	\caption{High-probability convergence of Algorithm~\ref{alg:Stochastic Linearized ADMM} for different $(\alpha,\beta)$ pairs in the strongly convex case.}
	\label{fig: Ellipse strongly convex highProb}
\end{figure}

Next, for $50$ independent runs of Algorithm~\ref{alg:Stochastic Linearized ADMM}, Figure~\ref{fig: Ellipse strongly convex highProb} illustrates the run-to-run variability of the empirical objective values: the shaded region denotes the  $\min$--$\max$ range across runs, and the solid line represents the average empirical objective values. As the number of iterations increases, the progressive shrinkage of this $\min$--$\max$ envelope is consistent with the high-probability convergence of Algorithm~\ref{alg:Stochastic Linearized ADMM} for different $(\alpha,\beta)$ pairs.

{\it \textbf{Convex case.}} 
For the general convex case, we set $\mu=0.5$ in our parameter rule. Similarly to the strongly convex case, we compared the methods in terms of the average empirical objective values. Figure~\ref{fig: Ellipse convex compare} plots the results for different values of $\beta$. From the figures, we see that our method reaches much lower objective values than the others. 
\begin{figure}[tbhp]
		\vspace{-15pt}
	\centering
	\subfloat[$\beta=10^{-4}$]{
		\includegraphics[width=0.242\textwidth, height=0.22\textwidth]{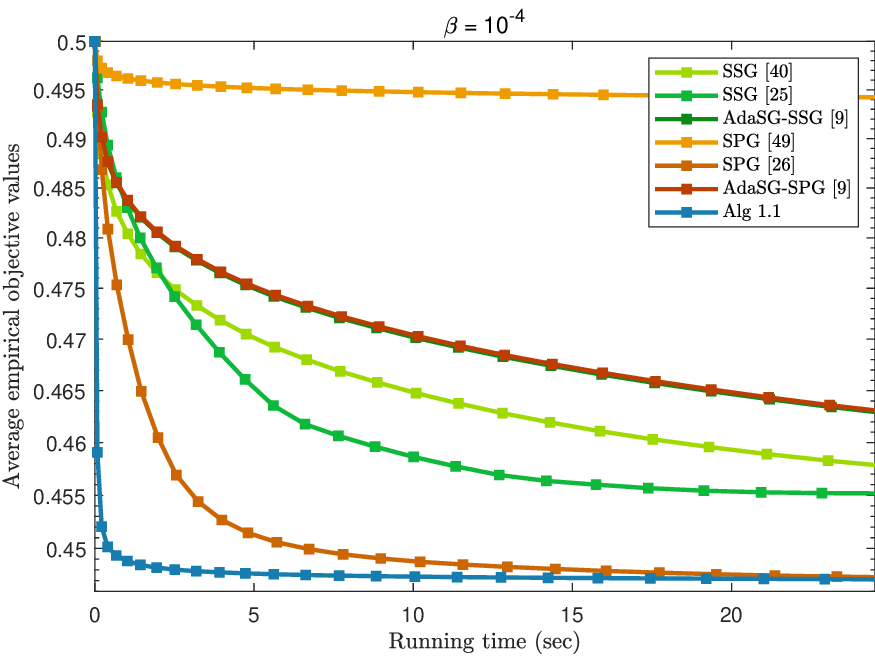}}
	\hspace{5em}
	\subfloat[$\beta=10^{-5}$]{
		\includegraphics[width=0.242\textwidth, height=0.22\textwidth]{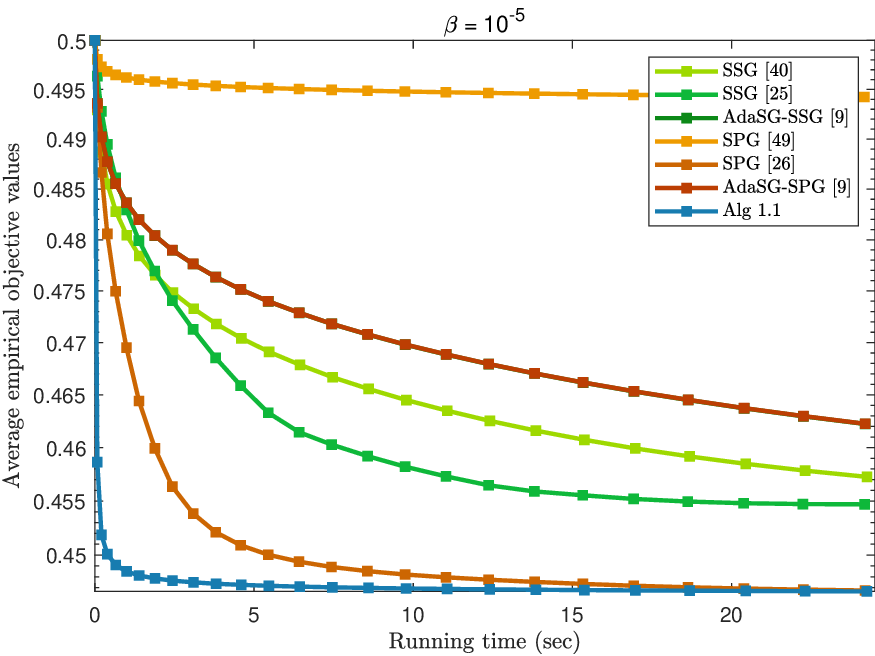}}
			\vspace{-7pt}
	\caption{Objective values of Algorithm~\ref{alg:Stochastic Linearized ADMM} and SG-type methods for different values of $\beta$ in the convex case.}
	\label{fig: Ellipse convex compare}
\end{figure}
\begin{figure}[tbhp]
			\vspace{-30pt}
	\centering
	\hspace{0.1em}
	\subfloat[$\beta=10^{-4}$]{
		\includegraphics[width=0.275\textwidth, height=0.215\textwidth]{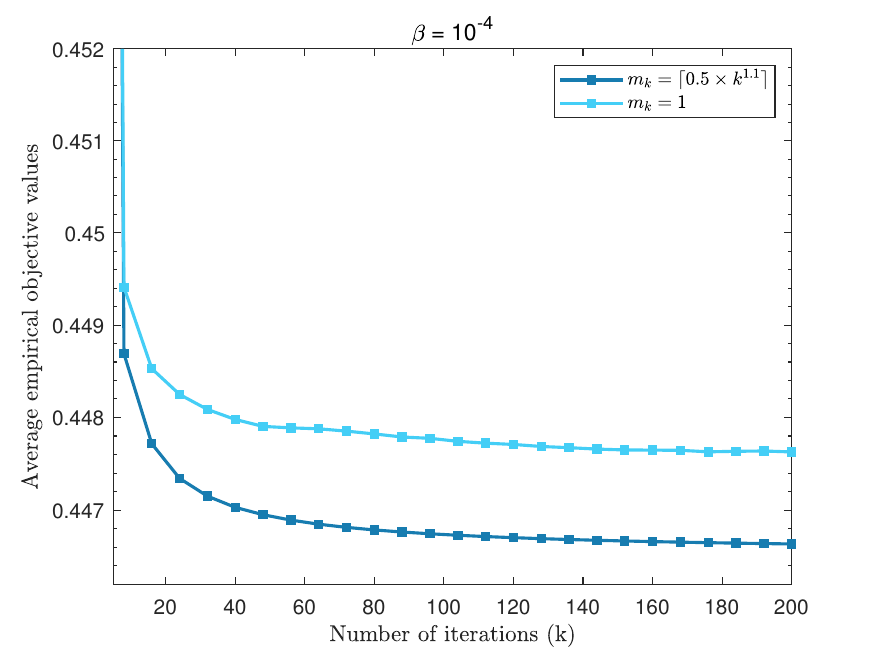}}
	\hspace{3.6em}
	\subfloat[$\beta=10^{-5}$]{
		\includegraphics[width=0.275\textwidth, height=0.215\textwidth]{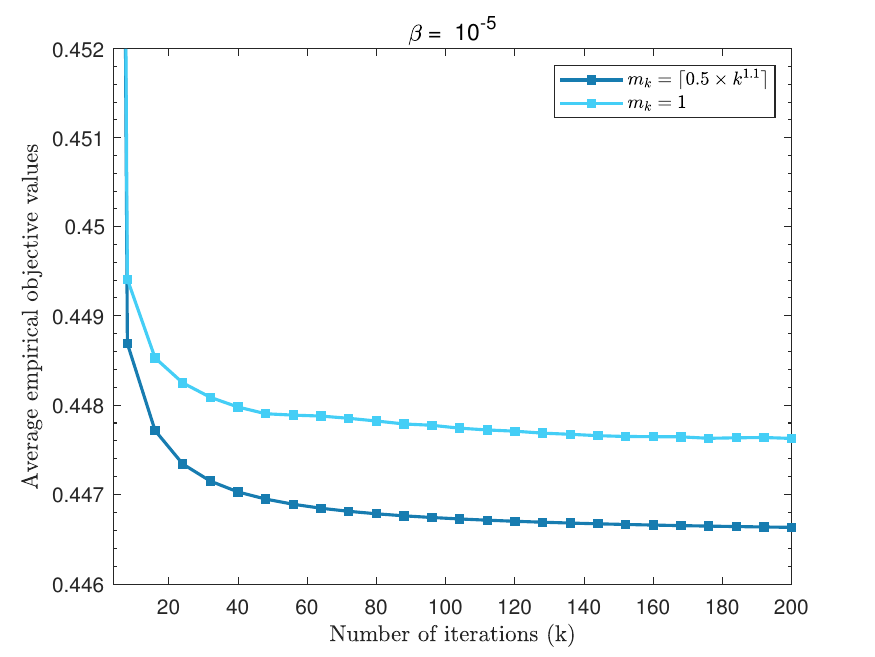}}
	\vspace{-7pt}
	\caption{Objective values of Algorithm~\ref{alg:Stochastic Linearized ADMM} for different values of $\beta$ in the convex case with different sampled stochastic gradients $G_{k}$.}
	\label{fig: Ellipse convex VR-E}
\end{figure}

\begin{figure}[htbp]
			\vspace{-20pt}
	\centering
	\subfloat{
		\includegraphics[width=0.242\textwidth]{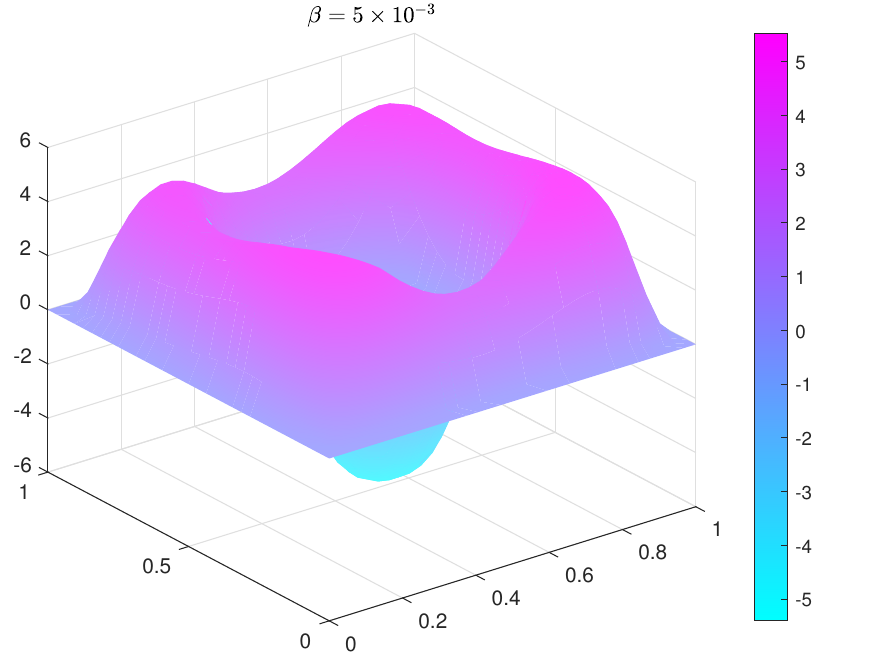}}
	\subfloat{
		\includegraphics[width=0.242\textwidth]{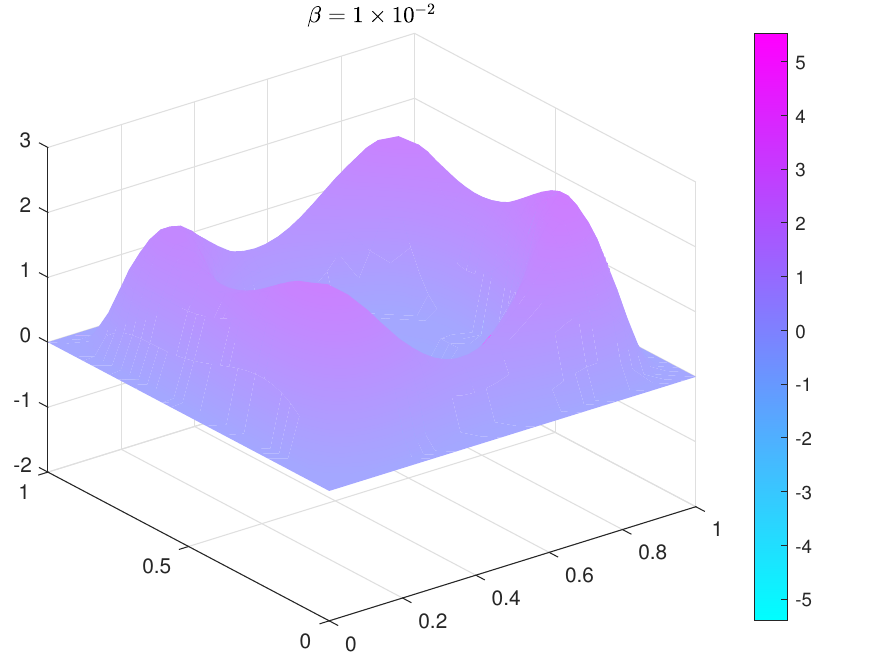}}
	\subfloat{
		\includegraphics[width=0.242\textwidth]{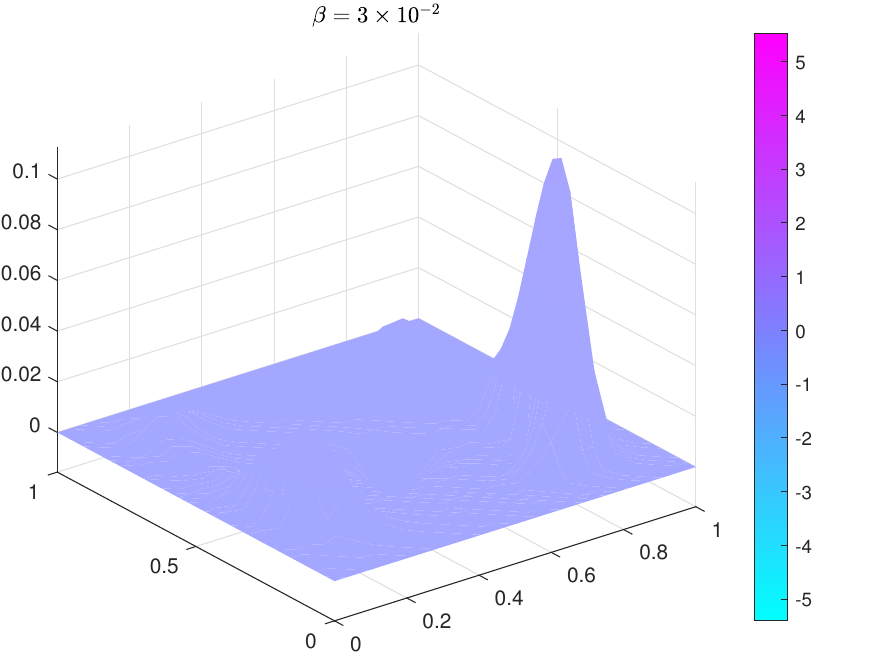}}
	\subfloat{
		\includegraphics[width=0.242\textwidth]{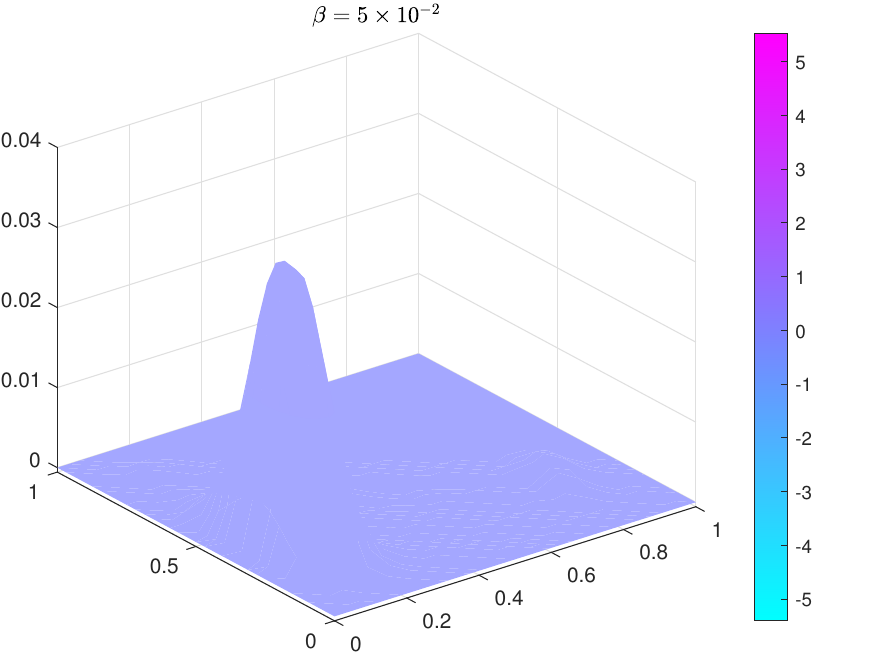}}\\
	\subfloat{
		\includegraphics[width=0.242\textwidth]{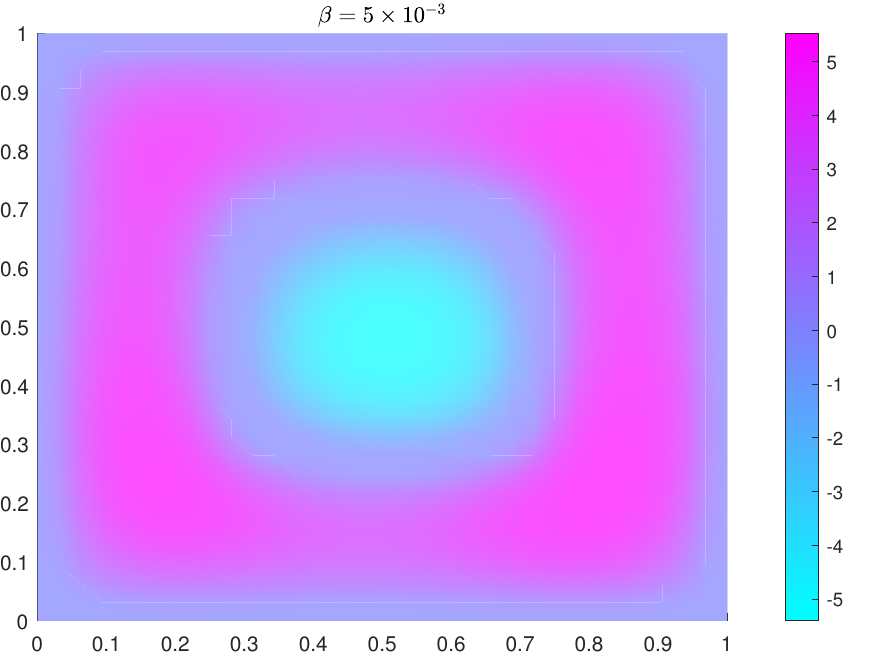}}
	\subfloat{
		\includegraphics[width=0.242\textwidth]{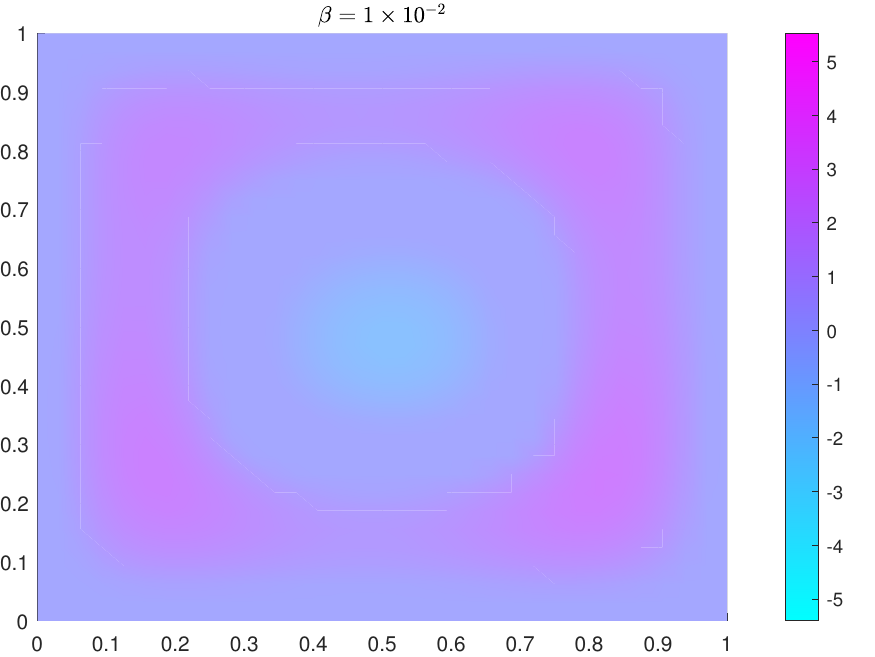}}
	\subfloat{
		\includegraphics[width=0.242\textwidth]{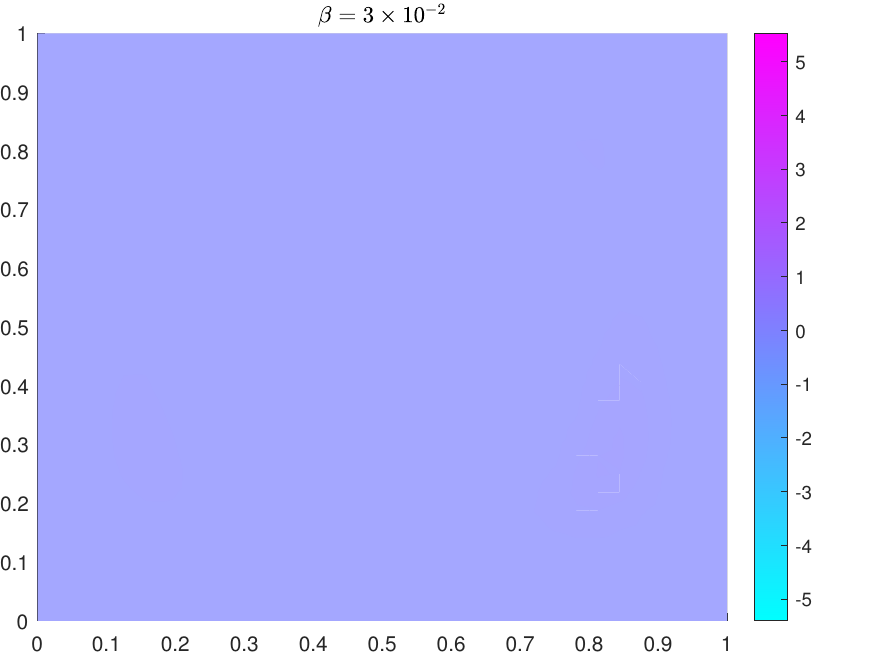}}
	\subfloat{
		\includegraphics[width=0.242\textwidth]{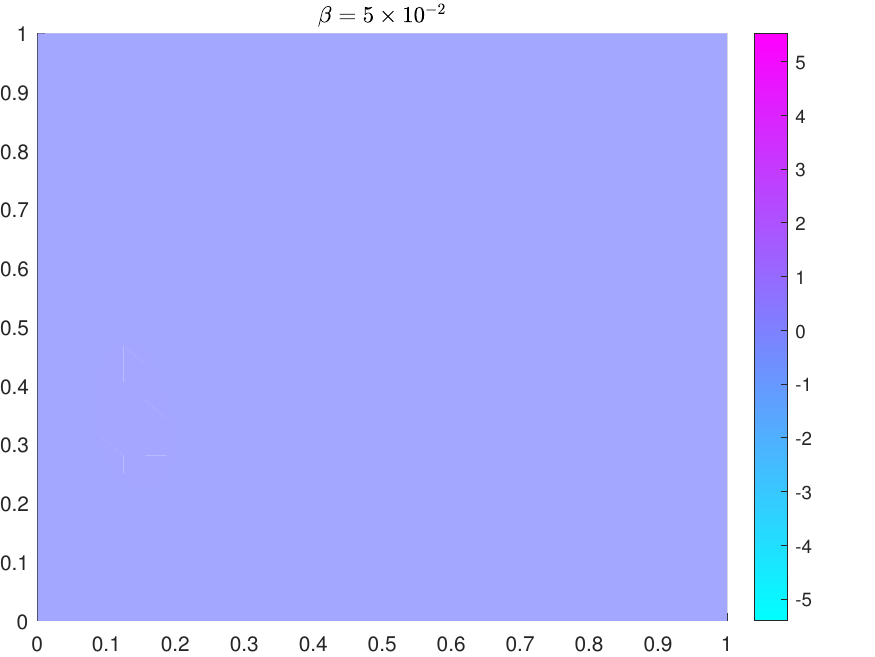}}
		\vspace{-10pt}
	\caption{Control variables for several values of $\beta$ in convex case. Top row: principal view; bottom row: top view.}
	\label{fig: Ellipse convex sparse u}
\end{figure}

Furthermore, we fixed the iteration number at $50$ and compared the average empirical objective values and the sparsity of the computed control under the two batch size settings $m_{k}=1$ and $m_{k} = \lceil 0.5 \times k^{1.1} \rceil$ for different $\beta$. The results are presented in Figure~\ref{fig: Ellipse convex VR-E} and Table~\ref{tab: Elliptical convex sparsity} respectively, which further verify that averaging $G_{k}$ has a crucial effect on the efficiency of our method. 
Additionally, for several values of $\beta$, Figure~\ref{fig: Ellipse convex sparse u} and \ref{fig: Ellipse convex sparse y} plot the computed control $u$ and the corresponding average empirical state $y$, respectively.

We also report the fluctuations of the empirical objective values over $50$ independent runs of Algorithm~\ref{alg:Stochastic Linearized ADMM} in Figure~\ref{fig: Ellipse convex highProb}, which is compatible with the high-probability convergence of the algorithm.

\begin{table}[htbp]
\footnotesize
\vspace{-5pt}
\caption{Percentage of the domain $D$ where $u \neq 0$ for different values of $\beta$ in convex case.}\label{tab: Elliptical convex sparsity}
		\vspace{-10pt}
\begin{center}
  \begin{tabular}{|c|c|c|c|c|c|c|} \hline
   	$\beta$& $3\times 10^{-3}$ & $5\times 10^{-3}$& $10^{-2}$& $3\times 10^{-2}$& $5\times 10^{-2}$& $ 10^{-1}$ \\ \hline
    $m_{k} = \lceil 0.5 \times k^{1.1} \rceil$ & $100\%$ &$99.90\%$ &$95.32\%$ & $62.02\%$ &$34.44\%$ &$0.73\%$ \\
    $m_{k} = 1$ &$100\%$ &$100\%$ &$99.06\%$ &$72.01\%$ &$43.81\%$ &$2.19\%$ \\  \hline
  \end{tabular}
\end{center}
\end{table}

\begin{figure}[H]
	\vspace{-20pt}
	\centering
	\subfloat{
		\includegraphics[width=0.242\textwidth]{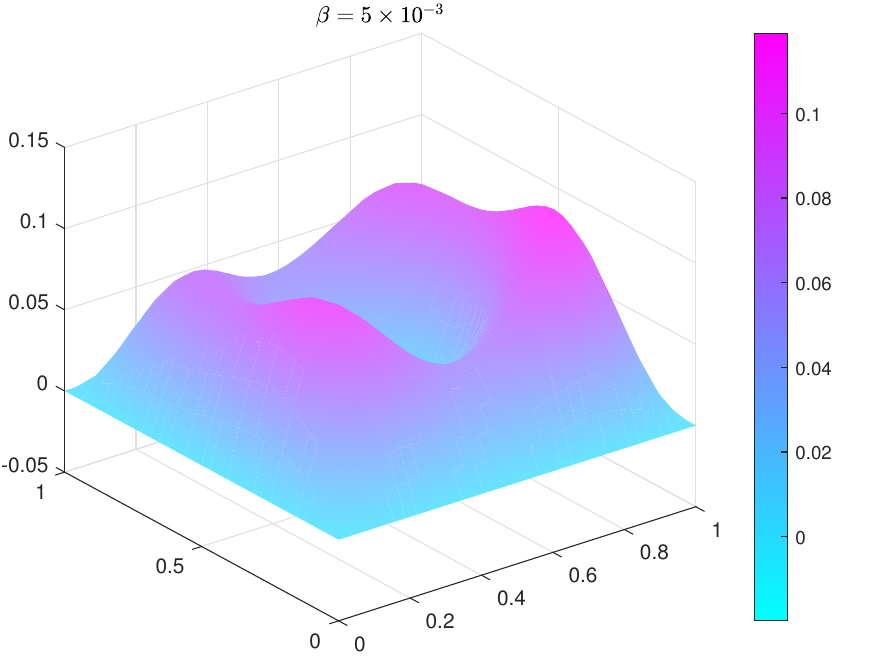}}
	\subfloat{
		\includegraphics[width=0.242\textwidth]{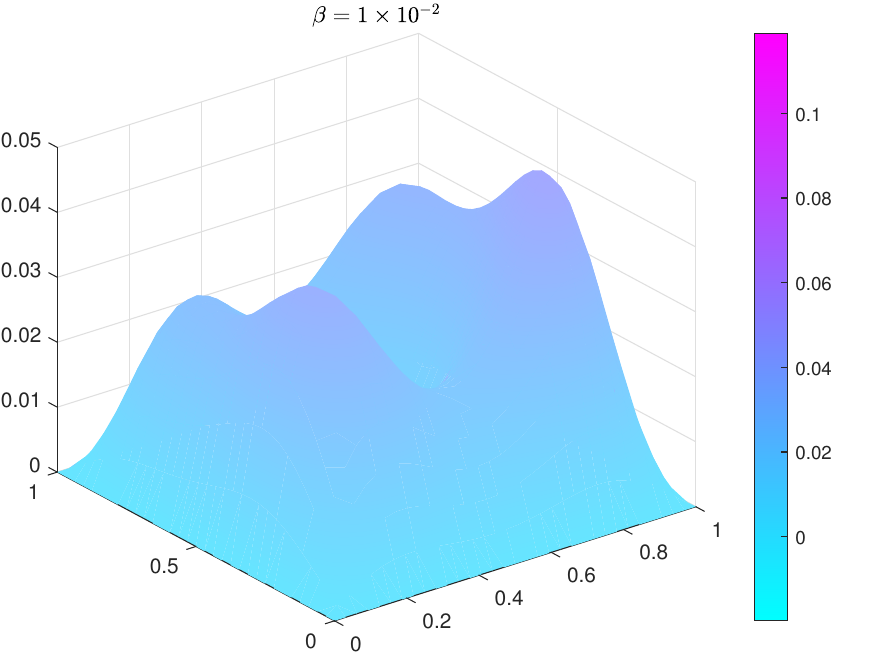}}
	\subfloat{
		\includegraphics[width=0.242\textwidth]{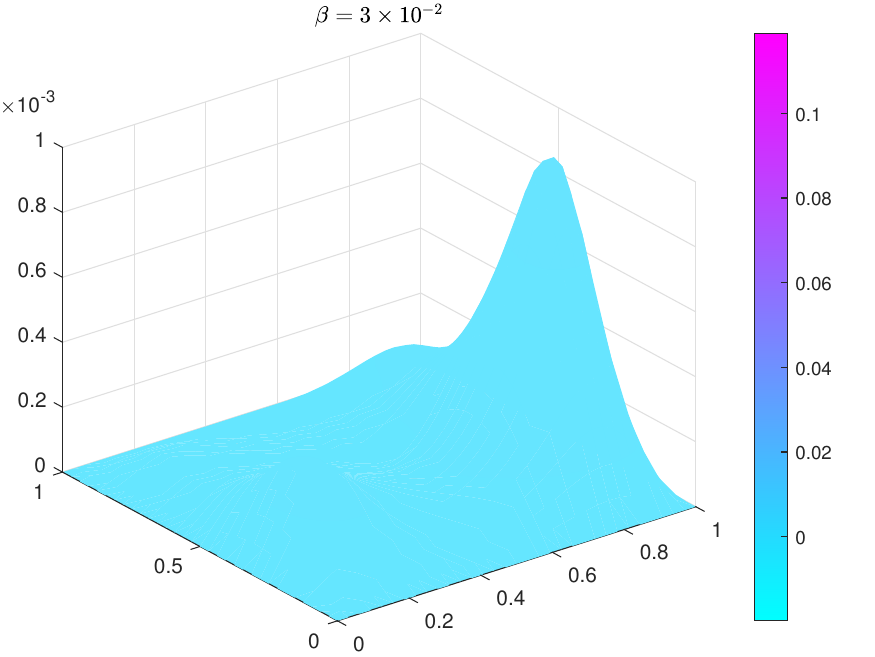}}
	\subfloat{
		\includegraphics[width=0.224\textwidth]{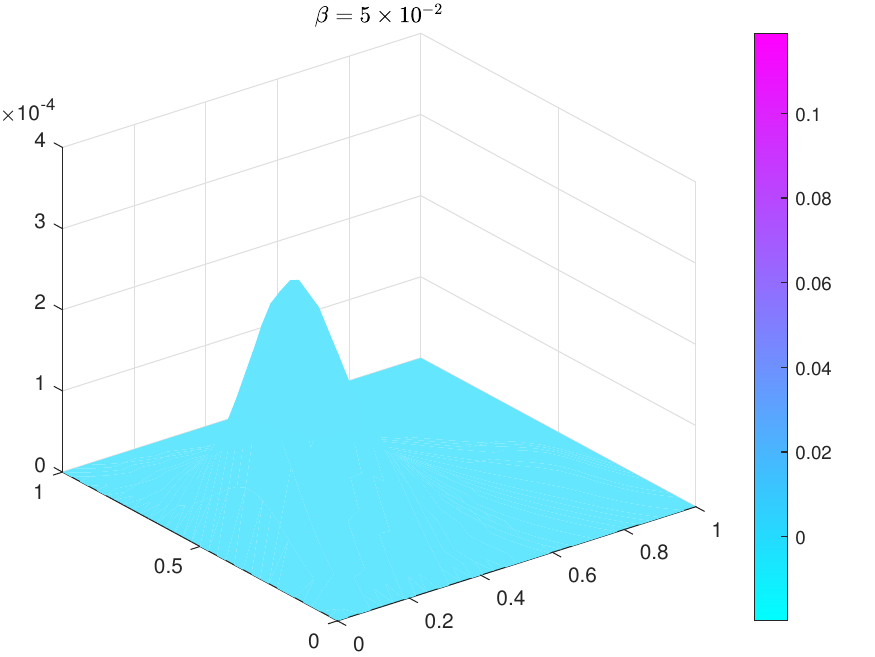}}
	\\
	\subfloat{
		\includegraphics[width=0.242\textwidth]{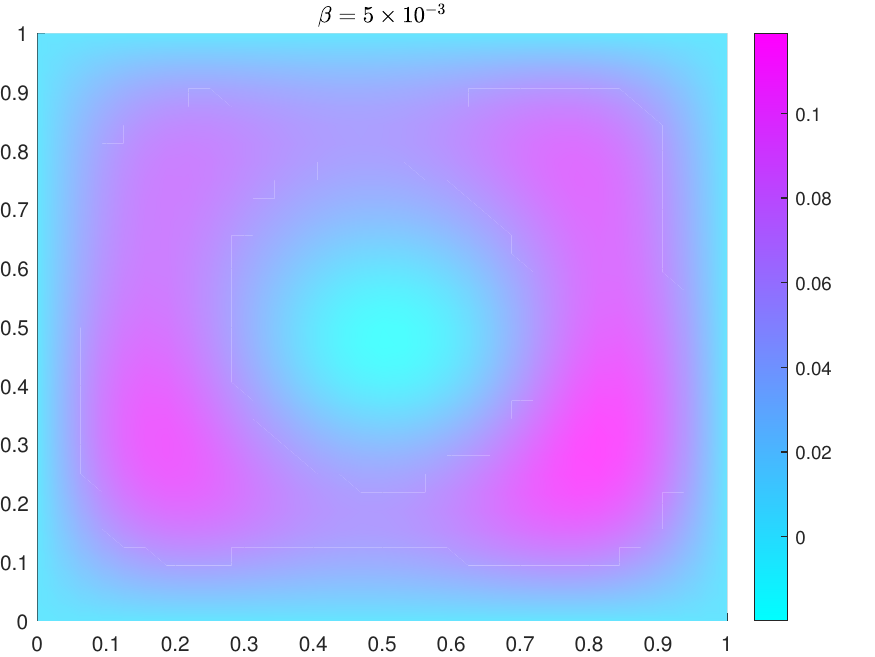}}
	\subfloat{
		\includegraphics[width=0.242\textwidth]{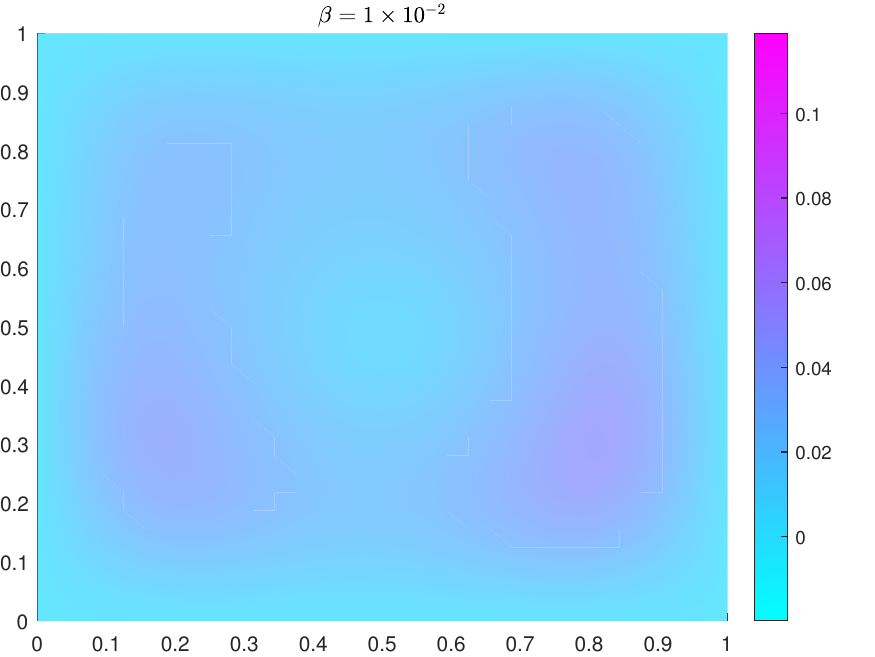}}
	\subfloat{
		\includegraphics[width=0.242\textwidth]{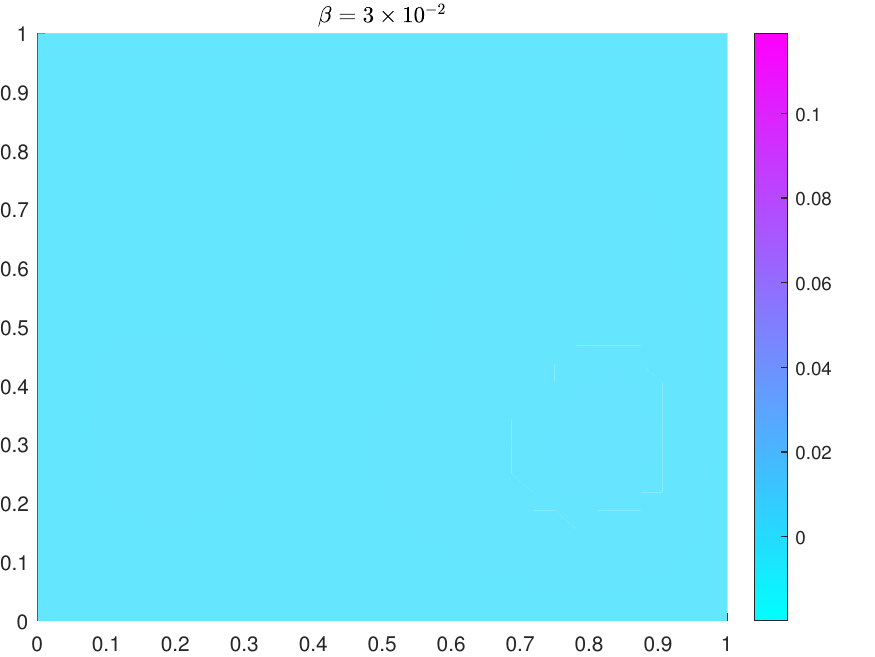}}
	\subfloat{
		\includegraphics[width=0.242\textwidth]{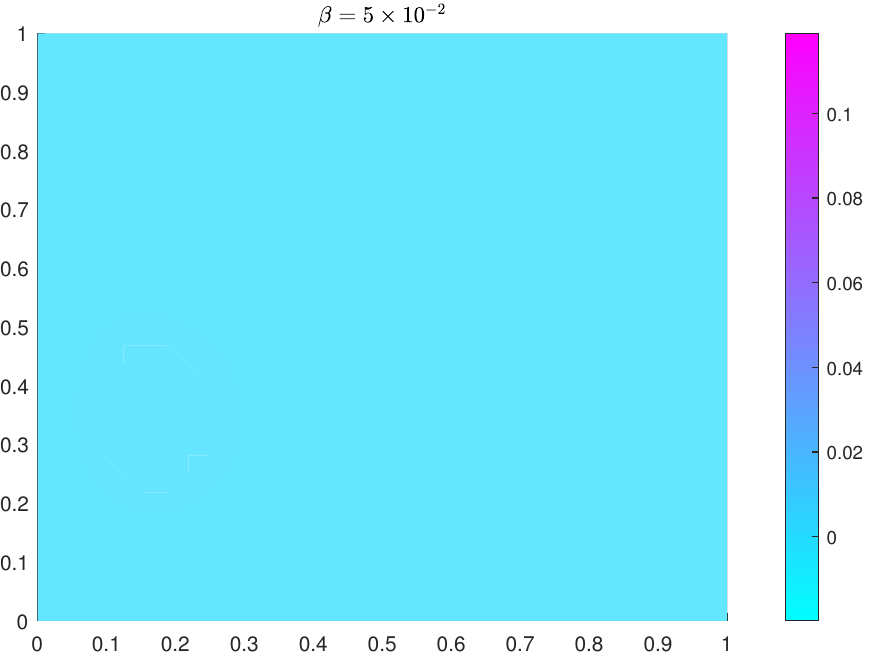}}
		\vspace{-15pt}
	\caption{State variables for several values of $\beta$ in convex case. Top row: principal view; bottom row: top view.}
	\label{fig: Ellipse convex sparse y}
\end{figure}
\begin{figure}[H]
	\vspace{-32pt}
	\centering
	\subfloat[$\beta=10^{-4}$]{
		\includegraphics[width=0.275\textwidth,height = 0.22\textwidth]{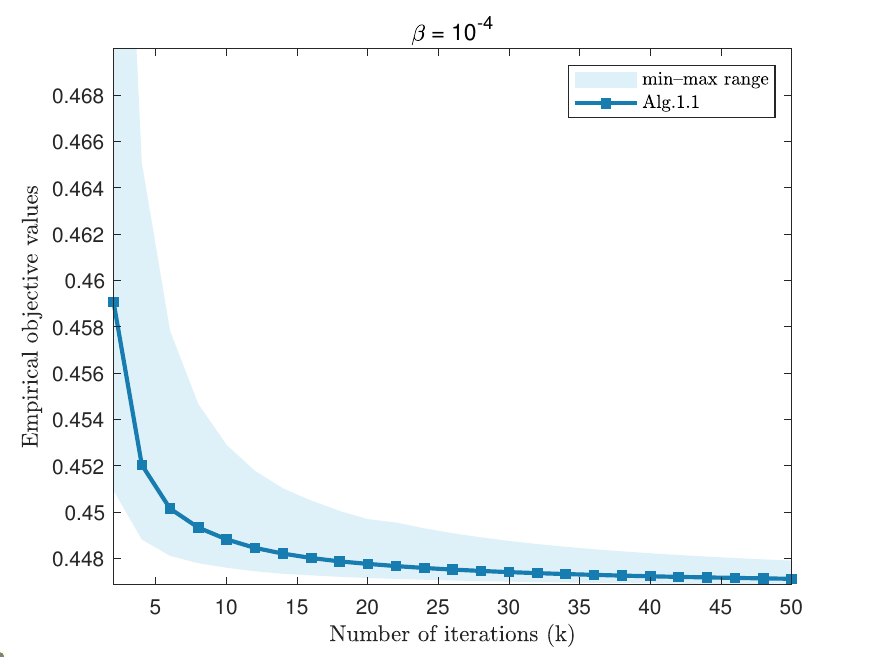}}
	\hspace{5em}
	\subfloat[$\beta=10^{-5}$]{
		\includegraphics[width=0.275\textwidth,height = 0.22\textwidth]{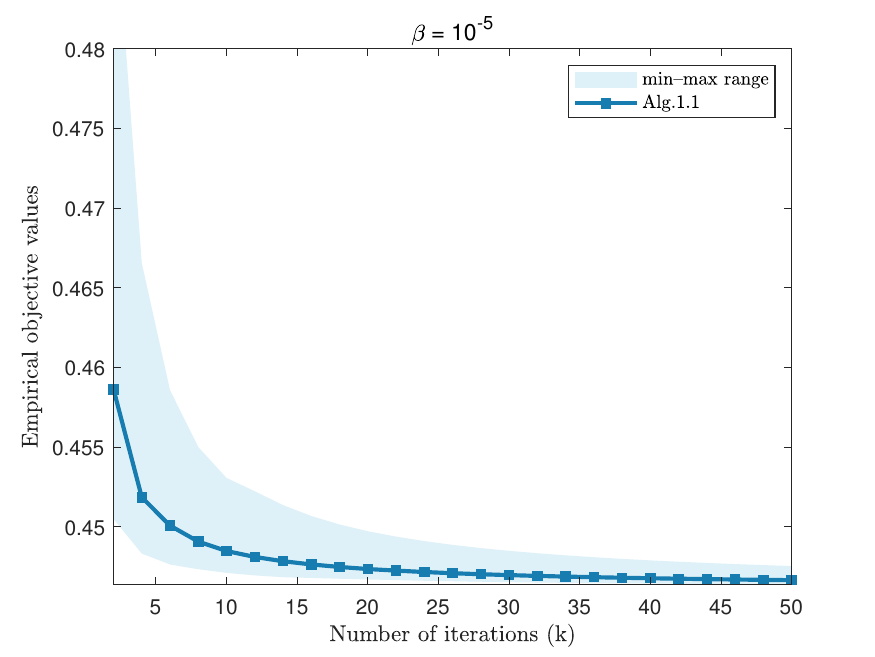}}
			\vspace{-7pt}
	\caption{High-probability convergence of Algorithm~\ref{alg:Stochastic Linearized ADMM} for different values of $\beta$ in general convex case}
	\label{fig: Ellipse convex highProb}
\end{figure}

\vspace{-10pt}
\section{Conclusions}
\label{sec:conclusions}
%
%Some conclusions here.
In this paper, we considered a class of stochastic composite convex optimization problems motivated by PDE-constrained optimization under uncertainty, and proposed a stochastic ADMM scheme based on stochastic approximation and linearization. 
At each iteration, the method decoupled the smooth and nonsmooth components of the objective, and both subproblems admitted closed-form solutions even with PDE-constraints. 
The nonergodic convergence results were analyzed for the strongly convex case and the general convex case, including faster sublinear convergence rates. Then, we focused on the application of our method to constrained optimal control problems governed by random PDEs, and assessed its efficiency via large-deviation probability estimates. Numerical experiments demonstrated that the proposed method is promising and competitive.

% \section*{Acknowledgments}

\bibliographystyle{siamplain}
\bibliography{references}
\end{document}

\maketitle

\section{A detailed example}

Here we include some equations and theorem-like environments to show
how these are labeled in a supplement and can be referenced from the
main text.
Consider the following equation:
\begin{equation}
  \label{eq:suppa}
  a^2 + b^2 = c^2.
\end{equation}
You can also reference equations such as \cref{eq:matrices,eq:bb} 
from the main article in this supplement.

\lipsum[100-101]

\begin{theorem}
An example theorem.
\end{theorem}

\lipsum[102]
 
\begin{lemma}
An example lemma.
\end{lemma}

\lipsum[103-105]

Here is an example citation: \cite{KoMa14}.

\section[Proof of Thm]{Proof of \cref{thm:bigthm}}
\label{sec:proof}

\lipsum[106-112]

\section{Additional experimental results}
\Cref{tab:smfoo} shows additional
supporting evidence. 

\begin{table}[htbp]
\footnotesize
  \caption{Example table.}\label{tab:smfoo}
\begin{center}
  \begin{tabular}{|c|c|c|} \hline
   Species & \bf Mean & \bf Std.~Dev. \\ \hline
    1 & 3.4 & 1.2 \\
    2 & 5.4 & 0.6 \\ \hline
  \end{tabular}
\end{center}
\end{table}

\bibliographystyle{siamplain}
\bibliography{references}

%% file: ex_shared.tex
\usepackage{lipsum}
\usepackage{amsfonts}
\usepackage{graphicx}
\usepackage{epstopdf}
\usepackage{algorithmic}
\ifpdf
  \DeclareGraphicsExtensions{.eps,.pdf,.png,.jpg}
\else
  \DeclareGraphicsExtensions{.eps}
\fi

\newsiamremark{remark}{Remark}
\newsiamremark{hypothesis}{Hypothesis}
\crefname{hypothesis}{Hypothesis}{Hypotheses}
\newsiamthm{claim}{Claim}
\newsiamremark{fact}{Fact}
\crefname{fact}{Fact}{Facts}

\headers{Faster Stochastic ADMM}{Weihua Deng, Haiming Song, Hao Wang, and Jinda Yang}

\title{Faster Stochastic ADMM for Nonsmooth Composite Convex Optimization in Hilbert Spaces\thanks{Submitted to the editors \today.
}}

\author{Weihua Deng\thanks{School of Mathematics and Statistics, Lanzhou University, Lanzhou, Gansu, China
  (\email{dengwh@lzu.edu.cn}, \email{yangjinda@lzu.edu.cn}).}
\and Haiming Song\thanks{School of Mathematics, Jilin University, Changchun, Jilin, China 
  (\email{songhaiming@jlu.edu.cn}, \email{haowang23@mails.jlu.edu.cn}).}
\and Hao Wang\footnotemark[3]
\and Jinda Yang\footnotemark[2]}

\usepackage{amsopn}